\documentclass[10pt,a4paper,reqno]{amsart} 
\usepackage{latexsym}
\usepackage{verbatim,array}
\usepackage{epsfig}
\usepackage{rotating}
\usepackage{amssymb}
\usepackage[T1]{fontenc}
\usepackage{afterpage}
\usepackage{color,hyperref}

\usepackage{amssymb,amsfonts,amsmath,stmaryrd}
\usepackage{trees,pifont, tikz, subfigure}
\usetikzlibrary{plotmarks,shapes,arrows,positioning}
\SetSymbolFont{stmry}{bold}{U}{stmry}{m}{n}
 
\graphicspath{{Figures/}}
 
\makeatletter
\def\testb#1{\testb@i#1,,\@nil}%
\def\testb@i#1,#2,#3\@nil{%
  \draw[->, thick] (O) --++(#1);
  \ifx\relax#2\relax\else\testb@i#2,#3\@nil\fi}
\makeatother   
\newcommand{\makediagr}[1]{
    \coordinate (O) at (0,0); \coordinate (N) at (0,0.8);
    \coordinate (NE) at (0.8,0.8); \coordinate (E) at (0.8,0);
    \coordinate (SE) at (0.8,-0.8); \coordinate (S) at (0,-0.8);
    \coordinate (SW) at (-0.8,-0.8);\coordinate (W) at (-0.8,0);
    \coordinate (NW) at (-0.8,0.8); \coordinate (B1) at (1.2,1.2);
    \coordinate (B2) at (-1.2,-1.2);
    \testb{#1}
} 
\newcommand{\diagr}[1]{
  \begin{tikzpicture}[scale=0.8]\makediagr{#1}\end{tikzpicture}
}

\newcommand{\diag}[1]{
  \begin{tikzpicture}[scale=0.3]\makediag{#1}\end{tikzpicture}
}

\newcommand{\makediag}[1]{
	\coordinate (-22) at (-2,2); \coordinate (-21) at (-2,1); \coordinate (-20) at (-2,0); \coordinate (-2-1) at (-2,-1); \coordinate (-2-2) at (-2,-2);
	\coordinate (-12) at (-1,2); \coordinate (-11) at (-1,1); \coordinate (-10) at (-1,0); \coordinate (-1-1) at (-1,-1); \coordinate (-1-2) at (-1,-2);
	\coordinate (02) at (0,2); \coordinate (01) at (0,1); \coordinate (0-1) at (0,-1); \coordinate (0-2) at (0,-2);
	\coordinate (12) at (1,2); \coordinate (11) at (1,1); \coordinate (10) at (1,0); \coordinate (1-1) at (1,-1); \coordinate (1-2) at (1,-2);
	\coordinate (22) at (2,2); \coordinate (21) at (2,1); \coordinate (20) at (2,0); \coordinate (2-1) at (2,-1); \coordinate (2-2) at (2,-2);
    \coordinate (O) at (0,0);
    \coordinate (A1) at (-2,-2); \coordinate (A2) at (-2,-1); \coordinate (A3) at (-2,0); \coordinate (A4) at (-2,1); \coordinate (A5) at (-2,2);
    \coordinate (B1) at (-2,-2); \coordinate (B2) at (-1,-2); \coordinate (B3) at (0,-2); \coordinate (B4) at (1,-2); \coordinate (B5) at (2,-2);
    
    \draw[color=red!20] (A1) --++ (4,0); \draw[color=red!20] (A2) --++ (4,0); \draw[color=red!20] (A3) --++ (4,0); 
    \draw[color=red!20] (A4) --++ (4,0); \draw[color=red!20] (A5) --++ (4,0);
    \draw[color=red!20] (B1) --++ (0,4); \draw[color=red!20] (B2) --++ (0,4); \draw[color=red!20] (B3) --++ (0,4); 
    \draw[color=red!20] (B4) --++ (0,4); \draw[color=red!20] (B5) --++ (0,4); 
    \testbb{#1}
} 

\makeatletter
\def\testbb#1{\testbb@i#1,,\@nil}%
\def\testbb@i#1,#2,#3\@nil{%
  \draw[->] (O) --++(#1);%  \draw (O) --++(#1);
  \ifx\relax#2\relax\else\testbb@i#2,#3\@nil\fi}
\makeatother

\catcode`\@=11
\def\section{\@startsection{section}{1}%
 \z@{.7\linespacing\@plus\linespacing}{.5\linespacing}%
 {\normalfont\bfseries\scshape\centering}}

\def\subsection{\@startsection{subsection}{2}%
  \z@{.5\linespacing\@plus\linespacing}{.5\linespacing}%
  {\normalfont\bfseries\scshape}}

\def\subsubsection{\@startsection{subsubsection}{3}%
 \z@{.5\linespacing\@plus\linespacing}{-.5em}%{.5\linespacing}%
 {\normalfont\bfseries}}
\catcode`\@=12

\newtheorem{Theorem}{Theorem}%[section]
\newtheorem{Lemma}[Theorem]{Lemma}
\newtheorem{Proposition}[Theorem]{Proposition}
\newtheorem{Definition}[Theorem]{Definition}
\newtheorem{Corollary}[Theorem]{Corollary}

\newtheorem{Conjecture}[Theorem]{Conjecture}

\newcommand{\sym}{{\rm sym} }
\def\qed{$\hfill{\vrule height 3pt width 5pt depth 2pt}$}
\def\qee{$\hfill{\Box}$}

\newfont{\bbold}{msbm10 scaled \magstep1}
\newfont{\bbolds}{msbm7 scaled \magstep1}

\newcommand{\ns}{\mathbb{N}}%{\mbox{\bbold N}}

\newcommand{\zs}{\mathbb{Z}}%{\mbox{\bbold Z}}

\newcommand{\qs}{\mathbb{Q}}%{\mbox{\bbold Q}}

\newcommand{\rs}{\mathbb{R}}%{\mbox{\bbold R}}
\newcommand{\cs}{\mathbb{C}}%{\mbox{\bbold C}}
\newcommand{\fps}{formal power series}

\newcommand{\bxun}{\bar x_1}%{\overline {x}_1}
\newcommand{\bxde}{\bar x_2}%{\overline {x}_2}

\newcommand{\bx}{\bar x}%{\overline x}
\newcommand{\bu}{\bar u}
\newcommand{\bv}{\bar v}

\newcommand{\by}{\bar y}%{\overline y}
\newcommand{\bz}{\bar z}

\newcommand{\Bx}{\mathbf{x}} % Bold x
\newcommand{\Bi}{\mathbf{i}} % Bold i
\newcommand{\Bu}{\mathbf{u}} 
\newcommand{\Bv}{\mathbf{v}} 
\newcommand{\Bw}{\mathbf{w}} 
\newcommand{\Bs}{\mathbf{s}}
\newcommand{\Ba}{\mathbf{a}}

\newcommand{\om}{\omega}
\newcommand{\GK}{\mathbb{K}}
\newcommand{\bGK}{\overline\GK}
\newcommand{\hGK}{\widehat\GK}

\newcommand{\GF}{\mathbb{F}}

\newcommand{\cG}{\mathcal G}
\newcommand{\cK}{\mathcal K}
\newcommand{\cN}{\mathcal N}
\newcommand{\cS}{\mathcal S}
\newcommand{\bcS}{\overline{\mathcal S}}

\newcommand{\PP}{\mathcal P}
\newcommand{\X}{\mathcal X}
\newcommand{\cQ}{\mathcal Q}
\newcommand{\cU}{\mathcal U}
\newcommand{\cV}{\mathcal V}
\newcommand{\cT}{\mathcal T}

\newcommand\atopfix[2]{\genfrac{}{}{0pt}{}{#1}{#2}}

\DeclareMathOperator{\id}{id}

\newcommand{\gf}{generating function}
\newcommand{\gfs}{generating functions}

\newcommand{\xp}{x_1}
\newcommand{\xm}{x_2}

\newcommand{\bone} {\bar 1}

\def\emm#1,{{\em #1}}

\newcommand{\vareps}{\varepsilon}

\def\twoFone#1#2#3#4{{_2F_1}\biggl(\begin{matrix}
  {#1}\kern.707em {#2}\\{#3}
\end{matrix}\,\bigg|\,#4\biggr)}

\begin{document} \title[Counting walks with large steps in an
orthant]{Counting walks with large steps in an orthant}

\author[A. Bostan]{Alin Bostan}
\author[M. Bousquet-M\'elou]{Mireille Bousquet-M\'elou}
\author[S. Melczer]{Stephen Melczer}

\thanks{S.M. was supported by the University of Waterloo, an Eiffel Fellowship, an NSERC Graduate Scholarship and Postdoctoral Fellowship, and NSF grant DMS-1612674.}

\address{AB: INRIA Saclay,  1 rue Honor{\'e} d'Estienne d'Orves, F-91120 Palaiseau, France} 
\email{Alin.Bostan@inria.fr}
 
\address{MBM: CNRS, LaBRI, Universit\'e de Bordeaux, 351 cours de la Lib\'eration,  F-33405 Talence Cedex, France} 
\email{bousquet@labri.fr}

\address{SM: Department of Mathematics, University of Pennsylvania, 209 S. 33rd Street, Philadelphia, PA 19104, USA.}
\email{smelczer@sas.upenn.edu}



\begin{abstract}

In the past fifteen years, the enumeration of lattice walks with steps taken
in a prescribed set $\cS$ and confined to a given cone, especially the first
quadrant of the plane, has been intensely studied. As a result, the \gf s of
quadrant walks are now well-understood, provided the allowed steps are
\emph{small}, that is $\cS \subset \{-1, 0,1\}^2$. In particular, having small
steps is crucial for the definition of a certain group of bi-rational
transformations of the plane. It has been proved that this group is finite if
and only if the corresponding \gf\ is D-finite (that is, it satisfies a linear
differential equation with polynomial coefficients). This group is also the
key to the uniform solution of 19 of the 23 small step models possessing a
finite group.

In contrast, almost nothing is known for walks with arbitrary steps. In this
paper, we extend the definition of the group, or rather of the associated
orbit, to this general case, and generalize the above uniform solution of
small step models. When this approach works, it invariably yields a D-finite
\gf. We apply it to many quadrant problems, including some infinite families.

After developing the general theory, we consider the $13\ 110$ two-dimensional
models with steps in $\{-2,-1,0,1\}^2$ having at least one $-2$ coordinate. We
prove that only 240 of them have a finite orbit, and solve 231 of them with
our method. The 9 remaining models are the counterparts of the 4 models of the
small step case that resist the uniform solution method (and which are known
to have an algebraic \gf). We conjecture D-finiteness for their generating
functions, but only two of them are likely to be algebraic. We also prove
non-D-finiteness for the $12\ 870$ models with an infinite orbit, except for
16 of them.

\end{abstract} 

\date{\today}

\maketitle

%%%%%%%%%%%%%%%%%%%%%%%%%%%%%%%%%%%%%%%%%%%%%%%%%%%%%%%%%%%%%%
\section{Introduction}
%%%%%%%%%%%%%%%%%%%%%%%%%%%%%%%%%%%%%%%%%%%%%%%%%%%%%%%%%%%%%%

The enumeration of planar lattice walks confined to the quadrant has received
a lot of attention over the past fifteen years. The basic question reads as
follows: given a finite step set $\cS\subset \zs^2$ and a starting point $P\in
\ns^2$, what is the number $q_n$ of $n$-step walks, starting from $P$ and
taking their steps in $\cS$, that remain in the non-negative quadrant $\ns^2$?
This is a versatile question, since such walks encode in a natural fashion
many discrete objects (systems of queues, Young tableaux and their
generalizations, among others). More generally, the study of these walks fits
in the larger framework of walks confined to cones. These walks are also much
studied in probability theory, both in a discrete~\cite{duraj,FaIaMa99} and in
a continuous~\cite{deblassie,garbit-raschel} setting. From a technical point
of view, counting walks in the quadrant is part of a general program aiming at
solving functional equations {that involve} \emm divided differences with
respect to several variables, (or \emm discrete, partial differential
equations): see {Equation~\eqref{dd} below} for a typical example,
and~\cite[Sec.~2]{mbm-gessel} for a general discussion on these equations.

On the combinatorics side, much attention has focused on the \emm nature, of
the associated \gf\ $Q(t)=\sum_n q_n t^n$. Is it rational in $t$, as for
unconstrained walks? Is it algebraic over $\qs(t)$, as for walks confined to a
(rational) half-space? More generally, is $Q(t)$ the solution of a linear
differential equation with polynomial coefficients in $\qs[t]$? (in short: is
it \emph{D-finite}?) The answer depends on the step set and, to a lesser
extent, on the starting point.

A systematic study was initiated in~\cite{mishna-jcta,BoMi10} for walks
starting at the origin $(0,0)$ and taking only \emm small, steps (that is,
$\cS\subset \{-1, 0, 1\}^2$). For these walks, a complete classification is
now available (Figure~\ref{fig:class-2D}). In particular, the {trivariate}
\gf\ $Q(x,y;t)$ that also records the coordinates of the endpoint of the walk
is D-finite if and only if a certain group $G$ of bi-rational transformations
is finite. The proof involves an attractive variety of tools, ranging from
{elementary} power series algebra~\cite{Bous05,mishna-jcta,BoMi10,mbm-gessel}
to complex analysis~\cite{KuRa12,raschel-unified}, computer
algebra~\cite{BoKa10,KaKoZe08}, probability
theory~\cite{denisov-wachtel,duraj} and number theory~\cite{BoRaSa14}. The
most recent results on this topic discriminate, among non-D-finite models,
those that are still \emm D-algebraic, (that is, satisfy polynomial
differential equations) from those that are
not~\cite{BeBMRa-FPSAC-16,BeBMRa-17,DHRS-17,DHRS-sing}. Remarkably, a new tool
then comes into play: differential Galois theory.

\begin{figure}
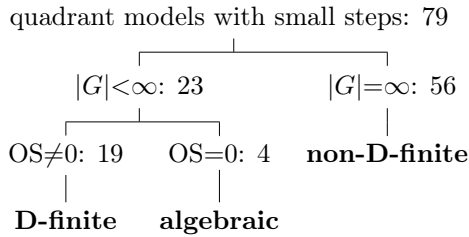

\centering
\begin{center}\edgeheight=5pt\nodeskip=1em\leavevmode
\tree{quadrant models with small steps: 79 
}{
\tree{ $|G|{<}\infty$: 23
}{
\tree{ OS${\neq}0$: {{19
}}}
{\tree{\textbf{{D-finite}}}{}}
\tree{ OS${=}0$:
            {{4
}}}{\tree{\textbf{{algebraic}}}{}}
}
\tree{ $|G|{=}\infty$:
          {56}}{\tree{\textbf{{non-D-finite}}}{}}
}
\end{center}

\caption{Classification of quadrant walks with small steps. The group of the walk is denoted by $G$, and OS stands for the \emm orbit sum,, a rational function which vanishes precisely for algebraic models. The 4 algebraic models are those of Figure~\ref{fig:alg}.} 
\label{fig:class-2D} 
\end{figure}

\begin{figure}[htb]
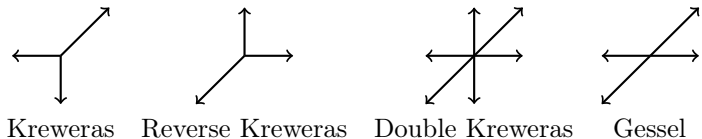

  \centering
   \begin{tabular}{cccc}
    $\diagr{NE,S,W}$  & $  \diagr{SW,E,N}$
            & $\diagr{NE,S,W,SW,E,N} $& $\diagr{E,W,NE,SW}$\\
    Kreweras& Reverse Kreweras & Double Kreweras & Gessel
  \end{tabular}
  \caption{The four algebraic small step models in the quadrant, with their usual names.}
  \label{fig:alg}
\end{figure}

\medskip Contrasting with the precision of this classification is the case of
quadrant walks with \emph{arbitrary steps}, for which it is fair to say that
almost nothing is known. Indeed, the small step assumption is crucial in all
methods used in the small step case, aside from two of them: the computer
algebra approach of~\cite{BoKa10,KaKoZe08} can in principle be adapted to any
steps, provided one is able to \emm guess, differential or algebraic equations
for the solution; and the asymptotic estimates of~\cite{denisov-wachtel} do
not require assumptions on the size of the steps. But even the definition of
the group that is central in the classification requires small steps. The
complex analytic approach of~\cite{KuRa12} that has proved very powerful for
small steps seems difficult to extend, and the first attempts have not yet led
to any explicit solution, nor indications on the nature of the
\gfs~\cite{fayolle-raschel-big}. {The classical reflection
principle~\cite{gessel-zeilberger} requires that no walk crosses the $x$- or
$y$-axis without actually touching it, which is equivalent to a small step
condition.}

The study of quadrant walks with arbitrary steps is not only a natural
mathematical challenge. It is also motivated by ``real life'' examples. For
instance, certain orientations of planar maps were recently shown by Kenyon et
al.~\cite{kenyon-bip} to be in bijection with quadrant walks taking their
steps in $\{(-p,0), (-p+1,1), \ldots, (0,p), (1,-1)\}$. In the forthcoming
paper~\cite{BoFuRa17} it is shown that the method of the current article
solves all these models. Other examples can be found in queuing theory, where
several clients may arrive, or be served, at the same time (think of ski-lifts
in a ski resort!). Also, a problem as innocuous as counting walks on the
square lattice confined to the cone bounded by the $x$-axis (for $x$ positive)
and the line $y=2x$  becomes, after a linear transformation, a quadrant
problem with large steps (Figure~\ref{fig:slope2}). Moreover, our study raises
intriguing combinatorial questions, which can be seen as an \emm a posteriori,
motivation of this work. For instance, some walks with large steps turn out to
be counted by simple hypergeometric numbers, for reasons that remain
combinatorially mysterious (see for instance Propositions~\ref{prop:12-a}
and~\ref{prop:hard-18}). Furthermore, our study gives rise to attractive
conjectures involving nine large step analogues of the four algebraic models
of Figure~\ref{fig:alg} (Section~\ref{sec:interesting}). We hope that this
paper will have a progeny as rich as its small step counterpart~\cite{BoMi10}.

\begin{figure}[htb]
  \centering
  \includegraphics[height=23mm]{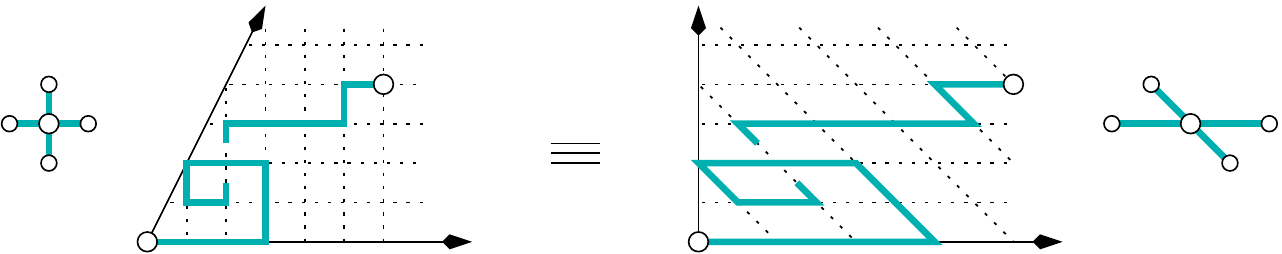}
  \caption{A square lattice walk confined to a wedge becomes a
    quadrant walk with large steps.}
  \label{fig:slope2}
\end{figure}

Our aim here is primarily to extend to arbitrary steps (and arbitrary
dimension, for walks confined to the orthant~$\ns^d$) a power series approach
that was introduced in~\cite{BoMi10} to solve the 19 easiest small step
models, namely those of the leftmost branch of Figure~\ref{fig:class-2D}. The
group is lost, but the associated orbit survives. When the method works, it
yields an expression of the \gf\, as the non-negative part of an algebraic
series --- a form which implies D-finiteness. On the negative side, we give a
criterion that simultaneously implies that the orbit of a
2-dimensional model is infinite
and that its \gf\ is not D-finite. We provide evidence that in 2D,
the finiteness of the orbit may still be related to the D-finiteness of the
solution. This is based in particular on the systematic exploration of
quadrant walks with steps in $\{-2,-1,0,1\}^2$.

Before we give more details on our
results, let us examine the solution of a simple small step model, as
presented in~\cite{BoMi10}. 

%=============================================
\subsection{A basic example: \texorpdfstring{$\cS= \{\searrow,\leftarrow,\uparrow\}$}{S = \{SE,E,N\}}}
\label{sec:basic}
%=============================================

We denote by $q(i,j;n)$ the number of walks with steps in $\cS$
 that start at $(0,0)$, end at $(i,j)$ and remain in the non-negative quadrant $\ns^2$. The associated \gf\ is 
\[
Q(x,y;t) :=\sum_{i,j,n\ge 0} q(i,j;n) x^i y^j t^n.
\]
We will find an explicit expression for this power series
using a four-step approach, sometimes
called the \emm algebraic kernel method, and borrowed from~\cite{BoMi10}, which we then generalize in the rest of the paper.

\smallskip\noindent
{\bf A functional equation.}
A step-by-step construction of quadrant walks with steps in $\{\searrow,\leftarrow,\uparrow\}$ yields the functional equation
\begin{equation}
\label{eqfunc-base}
Q(x,y)=1+t(x\by + \bx +y) Q(x,y) -tx\by Q(x,0) -t\bx Q(0,y),
\end{equation}
where we write $\bx:=1/x$, $\by:=1/y$ and {replace} $Q(x,y;t)$ by $Q(x,y)$ to
lighten notation. In this equation the constant term 1 stands for the
empty walk. The next term counts quadrant walks extended by one of our
three steps. The final two terms remove the contributions of the two
``forbidden moves'': either we have extended a walk ending on the
$x$-axis by a $\searrow$ step (term $-tx\by Q(x,0)$) or we have
extended a walk ending on the $y$-axis by a $\leftarrow$ step (term
$-t\bx Q(0,y)$). Observe that the above equation can also be written in a
form that involves two divided differences, one in $x$ and the other
in $y$:
\begin{equation}\label{dd}
Q(x,y)=1+ty Q(x,y) +tx\ \frac{Q(x,y)-Q(x,0)}{y} +t\ \frac{Q(x,y)- Q(0,y)}x.
\end{equation}
We refer to~\cite[Sec.~2]{mbm-gessel} for a general discussion on
equations involving divided differences with respect to two variables
(those that  involve divided differences with respect to one
variable only are known to have algebraic solutions~\cite{mbm-jehanne}).
We rewrite~\eqref{eqfunc-base} as
\begin{equation}
\label{eqfunc-base-RS}
K(x,y) xy Q(x,y)=xy -R(x) -S(y) ,
\end{equation}
where $R(x)= tx^2Q(x,0)$, $S(y)=ty Q(0,y)$, and $K(x,y)=1-t(x\by + \bx +y)$ is the \emm kernel, of the equation. Observe the decoupling of the $x$ and $y$ variables in the right-hand side. We call the bivariate series $R(x)$ and $S(y)$ \emm sections,.

\smallskip\noindent{\bf  The group of the walk.}
We now define two bi-rational transformations $\Phi$ and $\Psi$, acting
on pairs
 $(u,v)$ of coordinates (which will be, typically, algebraic
functions of $x$ and $y$):
\[ 
\Phi : (u,v) \mapsto ( \bu v, v) \qquad \hbox{and} \qquad \Psi: (u,v) \mapsto (u, u\bv).
\]
Each transformation fixes one coordinate, and transforms the other \emm so as to leave the step polynomial $u\bv + \bu +v$ unchanged,. Both transformations are involutions, and the orbit of $(x,y)$ under the action of $\Phi$ and $\Psi$ consists of 6 elements: 
\[ 
(x,y)  
 {\overset{\Phi}{\longleftrightarrow}} (\bx y,y)
 {\overset{\Psi}{\longleftrightarrow}} (\bx y,\bx)
 {\overset{\Phi}{\longleftrightarrow}} (\by,\bx)
 {\overset{\Psi}{\longleftrightarrow}} (\by,x\by)
 {\overset{\Phi}{\longleftrightarrow}} (x,x\by)
 {\overset{\Psi}{\longleftrightarrow}} (x,y).
\]
The group generated by $\Phi$ and $\Psi$ is thus the dihedral group of order 6.

\smallskip\noindent{\bf  A section-free equation.}
We now write, for each element $(x',y')$ of the orbit, the functional equation~\eqref{eqfunc-base-RS} with $(x,y)$ replaced by $(x',y')$: 
\begin{align}
K(x,y)\ xy Q(x,y) &= xy -R(x) -S(y), \nonumber\\
{K(x,y)}\  \bx y^2 Q(\bx y,y) &=  \bx y^2 -R(\bx y )- S(y),\nonumber
\\
{K(x,y)}\  \bx^2 y  Q(\bx y,\bx) &=  \bx^2 y  -R(\bx y )- S(\bx), \label{6-eqs}
\\
\vdots \qquad\qquad &= \qquad\qquad\vdots \nonumber\\
{K(x,y)}\  x^2\by Q(x,x\by) &= x^2\by   - R(x) - S(x\by).\nonumber
\end{align}
Due to the definition of $\Phi$ and $\Psi$, two consecutive equations have one section $R( \cdot)$ or $S(\cdot)$ in common. Thus,  the alternating sum of our 6 equations has a right-hand side \emph{free from sections}: 
\begin{multline}
 {K(x,y)} 
 \Big( xyQ(x,y) -\bx y^2 Q(\bx y,y)+ \bx ^2 y Q(\bx y,\bx)
 -\bx \by Q(\by,\bx)+x\by^2Q(\by,x\by)-x^2\by Q(x,x\by)\Big)
\\= {xy-\bx y^2+ \bx ^2 y -\bx \by+x\by^2-x^2\by}.\label{alt-tandem}
\end{multline}
The right-hand side of this equation is the \emm
  orbit sum, occurring in the classification of Figure~\ref{fig:class-2D}.
Equivalently,
\begin{multline*}
  xyQ(x,y) -\bx y^2 Q(\bx y,y)+ \bx ^2 y Q(\bx y,\bx)
-\bx \by Q(\by,\bx)+x\by^2Q(\by,x\by)-x^2\by Q(x,x\by)
\\={\frac
{xy \left( 1-\bx\by \right)  
\left( 1-\bx^2y \right)  
\left( 1-x\by ^2\right)}
{1-t(y+ \bx+ x\by)}}.
\end{multline*}

\smallskip\noindent{\bf Extracting $\boldsymbol{Q(x,y)}$.}
The last equation, combined with the fact that $Q(x,y)$ is a power series in~$t$
with polynomial coefficients in $x$ and $y$, characterizes $Q(x,y)$
uniquely: indeed, the series $xy Q(x,y)$ has coefficients in $xy
\qs[x,y]$, and thus involves only positive powers of $x$ and $y$.  But
the monomials occurring in each of the five other terms of the
left-hand side involve either a negative power of~$x$, or a negative
power of $y$ (or both). Hence the series $xy Q(x,y)$ is obtained by
expanding the  right-hand side as a series in $t$ with coefficients in
$\qs[x,\bx, y, \by]$, and then collecting terms with
positive powers of $x$ and $y$. We denote this extraction by:  
\[
xy Q(x,y) = [x^{>} y^{>}] \frac
{xy \left( 1-\bx\by \right)  
\left( 1-\bx^2y \right)  
\left( 1-x\by ^2\right)}
{1-t(y+ \bx+ x\by)}.
\]
Equivalently, upon dividing by $xy$, the series $Q(x,y)$ is obtained
by 
collecting the \emm non-negative part, in $x$ and $y$ of a rational function:
\[
Q(x,y)= [x^{\ge } y^{\ge }]\frac{  \left( 1-\bx\by \right)  
\left( 1-\bx^2y \right)  
\left( 1-x\by ^2\right)}{1-t(\bx+y+x\by)}.
\]
This explicit expression has strong consequences.
First, it guarantees that $Q(x,y)$ is D-finite~\cite{lipshitz-diag}. 
Second, expanding $(\bx+y+x\by)^n$ in powers of $x$ and $y$,  
it delivers a
hypergeometric expression for the  number of  walks of length $n=3m+2i+j$ ending at $(i,j)$:
\[
q(i,j;n)=\frac{(i+1)(j+1)(i+j+2) (3m+2i+j)!}{m!(m+i+1)!(m+i+j+2)!}.
\]
We conclude this example with a remark for the combinatorially inclined
readers: since walks with steps in $\cS=
\{\searrow,\leftarrow,\uparrow\}$ give a simple encoding of  Young
tableaux of height at most 3, the above formula is just the translation
in terms of walks of the classical hook 
formula~\cite[\S3.10]{sagan-book}.

%==============================================
\subsection{Outline of the paper}
%================================================

Based on the above example, we can now describe our results more precisely.
The next four sections {present} the extension to arbitrary steps (and
dimension) of the four stages involved in the above solution. The principles
of our approach are robust enough to be applicable to the enumeration of \emm
weighted, walks, which can be especially interesting in a probabilistic
context. We give many examples to illustrate these stages, but also to show
how they can fail: indeed, since our method only solves 19 of the 79 small
step models in the quadrant, we know in advance that it \emm has, to fail for
some models. Two obstacles can already be seen in the classification of
Figure~\ref{fig:class-2D}: the group (or what is left of it, namely its orbit)
can be infinite, and the orbit sum can vanish. Interestingly, we provide in
Section~\ref{sec:infinite-orbit} a criterion that implies simultaneously the
infiniteness of the orbit and the non-D-finiteness of the \gf.

In Sections~\ref{sec:1D} and~\ref{sec:hadamard}, we show that our approach
applies systematically in dimension 1 (walks on a half-line) and for the
so-called \emm Hadamard models, in dimension 2. Working in dimension 1 is the
least one can ask for, as walks on a half-line are very well
understood~\cite{banderier-flajolet,bousquet-petkovsek-recurrences,gessel-factorization}.
It is worth noting that the form of our solution is not exactly the standard
form obtained by earlier approaches. The second result, dealing with Hadamard
models, is more interesting as it seems that many models with finite orbit are
Hadamard. In the small step case for instance, 16 of the 19 models solvable by
our approach (that is, 16 of the 23 D-finite models) are of Hadamard type.

In Section~\ref{sec:m21} we apply these principles to the classification of
models with steps in $\{-2,-1,0,1\}^2$. Several results are still conjectural,
but in a sense we obtain a perfect analogue of the small step classification
shown in Figure~\ref{fig:class-2D}: our approach solves all 231 models with a
finite orbit and a non-vanishing orbit sum (Figure~\ref{classificationm2}).
For each of them, we express $Q(x,y;t)$ as the non-negative part in $x$ and
$y$ of an explicit rational function. Exactly 227 of these 231 solved models
are in fact Hadamard. This leaves out 9 models with a finite group but orbit
sum zero, for which we state several attractive conjectures. Finally, we
establish non-D-finiteness for the $12\ 870$ models with an infinite orbit,
except for 16 of them, which we still conjecture to be non-D-finite.

In Section~\ref{sec:asympt} we show that the form of the solutions that we
obtain is well-suited to the asymptotic analysis of their coefficients, and we
work out {explicitly} the analysis for the 4 non-Hadamard models with a finite
orbit solved in Section~\ref{sec:m21}.

We conclude in Section~\ref{sec:final} with a number of remarks and open
questions. \\

\noindent {\bf Notation and definitions.} For the sake of compactness we often
encode a step {into a word} consisting of its coordinates, with a bar above
negative coordinates: for example, the step $(-2,3,-5) \in \zs^3$ {will} be
denoted $\bar 23\bar 5$. Similarly, as used above, we use a bar over variables
to denote their reciprocals, so that $\bx=1/x$. A \emm small forward step, has
its coordinates in $\{1, 0, -1, -2, \ldots\}$ while a \emm large forward step,
has at least one coordinate larger than 1. We define similarly small and large
backward steps. A \emm small step, has only coordinates in $\{-1, 0, 1\}$. In
two dimensions, small steps can be identified by the compass directions, and
we sometimes draw them pictorially with arrows: for instance, $(1,1)$ can be
denoted $\nearrow$.

For a ring $R$, we denote by $R[x]$ (resp.~$R[[x]]$) the ring of polynomials
(resp. formal power series) in $x$ with coefficients in $R$. If $R$ is a
field, then $R(x)$ stands for the field of rational functions in $x$, and
$R((x))$ is the field of Laurent series in $x$ (that is, series of the form
$\sum_{n\ge n_0} a_n x^n$, with $n_0\in \zs$). This notation is generalized to
several variables in the usual way. For instance, the generating function
$Q(x,y;t)$ of walks restricted to the first quadrant is a series in
$\qs[x,y][[t]]$. We shall also consider \emm fractional power series,, namely
power series in a (positive) fractional power of $x$, and finally Puiseux
series, which are Laurent series in a fractional power of the variable. We
recall that if $R$ is an algebraically closed field, then Puiseux series in
$x$ with coefficients in~$R$ form an algebraically closed
field~(see~\cite{abhyankar-book} or~\cite[Chap.~6]{stanley-vol2}).

 If $F(u;t)$  is a power series in $t$ whose coefficients are Laurent series in $u$, 
\[
F(u;t)=\sum_{n\ge0} t^n \left(\sum_{i\ge i(n)} u^i f(i;n)\right),
\]
we denote by $[u^{>}]F(u;t)$ the \emm positive part of $F$ in $u$,:
\[
 [u^{>}]F(u;t)= \sum_{n\ge 0} t^n \left(\sum_{i>0 } u^i f(i;n)\right).
 \]
 We define the \emm non-negative part, $[u^{\ge}]F(u;t)$ in a similar
 fashion, by retaining as well the constant term in $u$.

We recall that a series $Q(x,y;t)$ is \emm algebraic, if there exists a
non-zero polynomial $P\in \qs[x,y,t,s]$ such that $P(x,y,t,Q(x,y;t))=0$. It is
\emm D-finite, (with respect to the variable $t$) if the vector space over
$\qs(x,y,t)$ spanned by the iterated derivatives $\partial_t^m Q(x,y;t)$, for
$m \ge 0$, has finite dimension (here $\partial_t$ denotes differentiation
with respect to $t$). The latter definition can be adapted to D-finiteness in
several variables, for instance $x$, $y$ and $t$: in this case we require
D-finiteness with respect to \emm each, variable
separately~\cite{lipshitz-df}. Every algebraic series is
D-finite~\cite[Prop.~2.3]{lipshitz-df}. If $Q(x,y;t)$ is D-finite in its three
variables, then so are $Q(0,0;t)$ and $Q(1,1;t)$. For a one-variable series
$F(t)=\sum f_n t^n$, D-finiteness is equivalent to the existence of a linear
recurrence relation with polynomial coefficients in $n$ satisfied by the
{coefficients} sequence $(f_n)$.

We often denote by $F_t$ the derivative ${\partial_t F}$ of a series $F(t)$.
This notation is generalized to several variables. For instance, $F_{t,u}$
stands for ${\partial_t\partial _u F}$.

%%%%%%%%%%%%%%%%%%%%%%%%%%%%%%%%%%%%%%%%%%%%%%%%%%%%%%%%%%%%%%
\section{A functional equation}
\label{sec:eqfunc}
%%%%%%%%%%%%%%%%%%%%%%%%%%%%%%%%%%%%%%%%%%%%%%%%%%%%%%%%%%%%%%

Let $d\ge 1$ and {let} $\cS$ be a finite subset of $\zs^d$. We would like to
count walks that take their steps in $\cS$, start from the origin and are
confined to {the orthant}~$\ns^d$. We denote by $q(i_1, \ldots , i_d;n)$ the
number of such walks consisting of $n$ steps and ending at $(i_1, \ldots,
i_d)$, and by $Q(x_1,\ldots, x_d;t)$ the associated \gf:
\[
Q(x_1, \ldots, x_d;t)\equiv Q(x_1, \ldots, x_d):=
\sum_{(i_1, \ldots, i_d , n ) \in \ns^{d+1}} q(i_1, \ldots, i_d;n) x_1^{i_1}
\cdots x_d^{i_d}t^n.
\]
Note that we often omit the dependence of $Q$ in $t$. The notation $Q$ refers
to the two-dimensional case (walks in a \emm quadrant,), from which we will
borrow most of our examples. In that case, we use the variables $x$ and $y$
instead of $x_1$ and $x_2$.

We use bold notation for multivariate quantities, so that $\Bx=(x_1, \ldots,
x_d)$, and for a $d$-tuple $\Bi=(i_1, \ldots, i_d) \in \zs^d$ we use the
abbreviation $\Bx^{\Bi} = x_1^{i_1} \cdots x_d^{i_d}$. The \emm step
polynomial, of a model (also called the \emm characteristic polynomial,) is
\begin{equation}
\label{S-def}
S(x_1, \ldots, x_d)=S(\Bx) = \sum_{\Bs \in\cS} {\Bx}^{\Bs}.
\end{equation}
The step polynomial is a Laurent polynomial in the variables $x_i$; here every
step has weight 1, but our approach can be adapted to the enumeration of
weighted walks with weights in some field $\GF$ (for
instance~{$\GF=\rs$} in a probabilistic context).

One can always write for the \gf\ $Q(x_1, \ldots, x_d;t)$ a functional
equation defining this series, based on a step-by-step construction of walks
confined to $\ns^d$, as was done in~\eqref{eqfunc-base}. This functional
equation is linear in the main series $Q(x_1, \ldots, x_d;t)$ {and,}
when the terms are grouped on one side of the {equation,} the
coefficient {in front of} $Q(x_1,\dots,x_d;t)$ is the \emm kernel,
\[ 
K(x_1,\ldots, x_d)= 1- t S(x_1, \ldots, x_d).
\]
The equation also involves unknown series that only depend on \emm some, of
the variables $x_1, \ldots, x_d$ {(and on $t$)}, {such as, for instance}, the series $Q(x,0)$ and
$Q(0,y)$ in~\eqref{eqfunc-base}. These series are called \emm sections, (of
$Q$). Let us consider a few examples.

\medskip
\noindent{\bf Example A.} 
Take $d=1$, and $\cS=\{\bar 1, 2\}$. The equation satisfied by the series $Q(x;t)\equiv Q(x)$ reads
\[
{Q(x)= 1+t(\bx +x^2) Q(x) -t \bx Q(0),}
\]
where the term $-t\bx Q(0)$ removes forbidden moves from position 0 to position $-1$. Equivalently, with $K(x)=1-t(\bx +x^2)$, 
{the previous equation reads}
\begin{equation}
\label{eqf-1D-bar}
K(x) Q(x) = 1- t\bx Q(0).
\end{equation}\qee

\medskip \noindent{\bf Example B.} Still with $d=1$, we now reverse the steps
of the previous example so as to have a long backward step, and study
$\cS=\{\bar 2, 1\}$. Extending a walk $w$ by the step $-2$ is now forbidden as
soon as $w$ ends at position $0$ or $1$. Hence, denoting by $Q_i\equiv Q_i(t)$
the length \gf\ of walks ending at position $i$, the equation satisfied by
$Q(x)$ reads
\[
{Q(x)= 1+t(\bx^2 +x) Q(x) -t \bx^2 Q_0 -t \bx Q_1,}
\]
or equivalently, with $K(x)=1-t(\bx^2 +x)$,
\begin{equation}\label{eqf-1D}
K(x) Q(x) = 1- t \bx^2 Q_0 -t \bx Q_1.
\end{equation}
Observe that $Q_0=Q(0)$ and $Q_1=
\partial_xQ(0)$. {The {occurrence} of a large backward step results in one
  more section on the right-hand side.} \qee

\medskip
\noindent{\bf Example C: Gessel's walks.} 
We return to two-dimensional models, now with the step set $\cS=\{\rightarrow, \nearrow, \leftarrow, \swarrow\}$. Appending a south-west step is forbidden as soon as the walk ends at abscissa or ordinate zero. The functional equation thus reads: 
\[
Q(x,y)=1+ t ( x+xy+\bx + \bx\by) Q(x,y) -t\bx Q(0,y) - t\bx \by
(Q(x,0)+Q(0,y)-Q(0,0)).
\]
The term in $Q(0,0)$ avoids removing twice walks {that end} at $(0,0)$. 
Equivalently, with $K(x,y)=1-t( x+xy+\bx + \bx\by)$, 
\[
K(x,y) Q(x,y) = 1 -t \bx(1+\by) Q(0,y) - t\bx\by (Q(x,0)-Q(0,0)).
\]
\qee

\medskip

\noindent{\bf Example D: {A model with a large forward step and a large backward step.}}
We now take $\cS=\{ \bar 20, \bar 11, 02, 1\bar 1\}$. Quadrant walks
formed of these steps, starting and ending at the origin, are known to
be in bijection with \emm bipolar orientations of
quadrangulations,~\cite{kenyon-bip,BoFuRa17}. The functional equation
reads 
\[
Q(x,y)=1+t(\bx^2+\bx y+y^2+x\by ) Q(x,y) -t\bx^2
\left(Q_{0,-}(y)+xQ_{1,-}(y)\right) -t\bx y Q_{0,-}(y)- tx\by Q(x,0),
\]
where $x^iQ_{i, -}(y)$ counts quadrant walks ending {at} $x$-coordinate~$i$. Note that $Q_{0,-}(y)=Q(0,y)$. We can rewrite {the functional equation}, using $K(x,y)=1-t( \bx^2+\bx y+y^2+x\by )$, as
\begin{equation}
\label{eqfunc:quadrangulations}
K(x,y) Q(x,y)= 1 -t \bx(\bx+y)Q_{0,-}(y) -t\bx Q_{1,-}(y)- tx\by Q(x,0).
\end{equation}\qee

\medskip

\noindent{\bf Example E: A model in three dimensions.}
We now take $\cS=\{\bar 1 \bar 1 \bar 1 ,\bar 1 \bar 1 1,\bar 1 10,100\}$. As
{for Gessel's walks} (Example C), the functional equation involves
inclusion-exclusion so as to avoid excluding several times the same move, and
one obtains:
\begin{multline}
   K(x,y,z)  Q(x,y,z)= 1  
-t\bx(\by \bz+\by z+ y) Q( 0,y,z)
-t\bx\by (\bz+z) Q(x, 0,z)
-t\bx\by\bz Q(x,y, 0)\\
\label{eqfunc-3D}
+t\bx\by (\bz+z) Q( 0, 0,z)
+ t\bx\by\bz Q( 0, y,  0)
+t\bx\by\bz Q(x, 0, 0)
-t\bx\by\bz Q( 0, 0, 0),
\end{multline}
where the kernel is
\[
K(x,y,z)=1-t( \bx\by\bz +\bx\by z+\bx y+x).
\]\qee

\medskip
After {seeing} all these examples, the reader should be convinced that a
functional equation can be written for any model $\cS$.
{We only give its general form} in two cases: first in
dimension two, and then for models with small backward steps. In
dimension {two}, the equation reads:
\begin{equation}\label{eq:quadrant}
K(x,y)Q(x,y)=1 - t \sum_{(k, \ell) \in \cS} x^k y^\ell
\left( \sum_{0\le i <-k} x^i Q_{i,-}(y) + \sum_{0\le  j<-\ell} y^j  Q_{-,j}(x) -
\sum_{\substack{0\le  i<-k\\ 0\le j<-\ell}} x^i y^j Q_{i,j}\right),
\end{equation}
where $K(x,y)=1-tS(x,y)$ is the kernel, $x^iQ_{i, -}(y)$ (resp.
$y^jQ_{-,j}(x)$) counts quadrant walks ending at abscissa $i$ (resp. at
ordinate $j$), and $Q_{i,j}$ is the length \gf\ of walks ending at
$(i,j)$.

For a model {of walks with \emph{small backward steps} confined to the orthant $\mathbb{N}^d$ in arbitrary dimension~$d$, the functional equation reads}:
\begin{equation}\label{eqfunc:small}
K(\Bx)Q(\Bx)=1+t \sum_{\emptyset \not = I \subset \llbracket 1, d\rrbracket} \left((-1)^{|I|} Q_I(\Bx) \sum_{\substack{\Bs \in \cS: \\ s_i =-1 \,\, \forall i \in I}} \Bx^{\Bs}\right),
\end{equation}
where $\Bx=(x_1, \ldots, x_d)$, $K(\Bx)=1-tS(\Bx)$, and $Q_I(\Bx)$ is the specialization of $Q(\Bx)$ where each $x_i, i\in I$, is set to $0$ (for instance, if $I=\{2,3\}$ then $Q_I(x)=Q(x_1, 0,0, x_4, \ldots, x_d)$). The proof is an inclusion-exclusion argument generalizing {the proof of}~\eqref{eqfunc-3D}.

%%%%%%%%%%%%%%%%%%%%%%%%%%%%%%%%%%%%%%%%%%%%%%%%%%%%%%%%%%%%%%
\section{The orbit of \texorpdfstring{$\boldsymbol{(x_1, \ldots, x_d)}$}{(x1,...,xd)}}
\label{sec:orbit}
%%%%%%%%%%%%%%%%%%%%%%%%%%%%%%%%%%%%%%%%%%%%%%%%%%%%%%%%%%%%%%

In Section~\ref{sec:basic}, we have {shown on one example} how to associate a
group to a 2D model with small steps. We now describe, for a general step set
$\cS$ in arbitrary dimension $d$, how to define the counterpart of this group,
or more precisely of its orbit. To avoid trivial cases, we only consider
models that have {both} positive and negative steps in each direction.

%=====================================================
\subsection{Definition and first examples}
%=====================================================
We denote by $\GK$ the field $\cs(x_1, \ldots, x_d)$, and by $\overline
\GK$ an algebraic closure of $\GK$. We first define two relations $\approx$ and
$\sim$ on elements of $\left(\overline \GK\setminus \{0\}\right)^d$;
recall that $S(\Bx)$ denotes the step polynomial of~$\cS$, defined
by~\eqref{S-def}. 

\begin{Definition}
\label{def:equiv} 
Let $\Bu$ and $\Bv$ be two {distinct} $d$-tuples in  $\left(\overline
  \GK\setminus \{0\}\right)^d$, and let $1\le i \le d$. Then $\Bu$ and $\Bv$ are \emm
$i$-adjacent,, denoted $\Bu \stackrel i \approx \Bv$, if  $S(\Bu)=S(\Bv)$  and $\Bu$
and $\Bv$ differ only by their $i$th coordinate. They are \emm
adjacent,, denoted $\Bu  \approx \Bv$, if they are $i$-adjacent for
some $i$.

Clearly, the relation $\approx$ is symmetric. We denote by $\sim$ its
reflexive and transitive closure. The \emm orbit, of $\Bu$ is its equivalence
class for this relation.

The \emm {$\Bu$-length}, of an element $\Bv$ in the orbit of $\Bu$ is the smallest $\ell$ such that there exists $\Bu^{(0)}=\Bu,
  \Bu^{(1)}, \ldots, \Bu^{(\ell)}= \Bv$ with $\Bu^{(0)}\approx
  \Bu^{(1)} \approx \cdots \approx \Bu^{(\ell)}$.
\end{Definition}
Note that the value of $S$ is constant over the orbit of $\Bu$. We will often
refer to the orbit of $\Bx=(x_1, \ldots,x_d)$ (or $(x,y)$ in two dimensions)
as \emph{the} orbit of the model $\cS$, {and to the \emm length, of an element
of this orbit as its $\Bx$-length}. We use the word \emm orbit, even though we
have not defined any underlying group: this terminology comes from the case of
small steps, as justified by Proposition~\ref{prop:small-group} below. Before
we proceed, let us check that the structure of the orbit does not depend on
the choice of the algebraic closure of $\GK$.

\begin{Lemma}\label{lem:isom}
  Let $ \bGK$ and $\hGK$ be two algebraic closures of
  $\GK$ and $\tau :  \bGK \rightarrow \hGK$ a field automorphism
  fixing $\GK$. For any $\Bu=(u_1, \ldots, u_d)\in
  \left(\bGK\setminus\{0\}\right)^d$, we denote by $\tau(\Bu)$ the
  element of {$\left(\hGK\setminus\{0\}\right)^d$}
 obtained by applying
  $\tau$ to $\Bu$ component-wise. Then $\tau$ {preserves  adjacencies,
  and sends the orbit of $\Bu$
  onto the orbit of   $\tau(\Bu)$.}
\end{Lemma}
\begin{proof} (sketch)
  Clearly, if $S(\Bv)=S(\Bw)$ then  $S(\tau(\Bv))=S(\tau(\Bw))$, because
  $S$ has rational coefficients. And if $\Bv$ and $\Bv'$ differ by
  their $i$th coordinate, the same holds for their images
  by~$\tau$. {This shows that adjacencies are preserved.}
  The {isomorphism of orbits} then follows by induction on the length.
\end{proof}

The next proposition tells that two models that are
  equivalent up to a symmetry of the  hypercube have isomorphic
  orbits. Since these  symmetries are generated by a
reflection and adjacent transpositions, it suffices to examine these
two cases.

\begin{Proposition}\label{prop:sym}
  Let $\cS \subset \zs ^d$ be a model with step polynomial $S(x_1,
  \ldots, x_d)$, and let $\tilde \cS$ be the model obtained by
  swapping the first two coordinates, with step
  polynomial $S(x_2, x_1, x_3 ,\ldots, x_d)$. Then the orbits of $\cS$
  and $\tilde \cS$ are \emm isomorphic, (there is a bijection from one to
  the other that preserves adjacencies).

The same holds if $\tilde S$ is obtained from $\cS$ by a reflection in
the hyperplane $x_1=0$; that is, if its step  polynomial is
$S(1/x_1, x_2, \ldots, x_d)$.
\end{Proposition}
\begin{proof}
To lighten notation, we prove this result in two dimensions. The
proof is similar in higher dimensions.

In the first case, let us construct the orbit of $\cS$ in
the field $\bGK$ of iterated Puiseux series in $x$ and $y$
(Puiseux series in $x$ whose coefficients are Puiseux series in
$y$). We shall construct the orbit of $\tilde \cS$ in the field $\hGK$ of iterated Puiseux series in $y$ and $x$ (note the inversion).
If $u\in  \bGK$,  let $\delta(u) \in  \hGK$  be obtained from
$u$ by
swapping $x$ and $y$. We claim that, if $(u,v)$ is in the
orbit of~$\cS$, then the
pair $(\delta(v), \delta(u))$ is in the orbit of $\tilde \cS$, and
vice-versa. {First, if $(u,v)=(x,y)$, then $(\delta(v),
\delta(u))=(x,y)$. Then, if $(u,v)$ is 2-adjacent to $(u,w)$ in the
orbit of $\cS$, then $(\delta(v),
\delta(u))$ is 1-adjacent to $(\delta(w),
\delta(u))$ in the orbit of $\tilde \cS$, because
\[
\tilde S(\delta(w),\delta(u))= S(\delta(u), \delta(w))= S(\delta(u),
\delta(v))=\tilde S(\delta(v),\delta(u)).
\]
(The second equality comes from the 2-adjacency of $(u,v)$ and $(u,w)$
for $\cS$.) One proves similarly that 1-adjacencies for $\cS$ become
2-adjacencies for $\tilde \cS$.
The isomorphism between the orbits of
$\cS$ and $\tilde \cS$ then follows by induction on the length.}

The proof is similar in the second case, upon constructing the orbit
of $\tilde \cS$ in the field $\hGK$ of iterated Puiseux series in $\bar x$
and $y$.  Denoting by $\delta$ the transformation from $\bGK$
to $\hGK$ that sends $x$ to $\bx$, a pair $(u,v)$ is in the orbit
of $\cS$ if and only if $(1/\delta(u), {\delta(v)})$ is in the orbit of $\tilde \cS$.
\end{proof}
 
We will now examine examples. One important observation is the
following.
\begin{Lemma}\label{lem:ind}
   If the coordinates of $\Bu$  are algebraically independent over $\qs$, then
the same holds for any $\Bv$ {in the orbit of $\Bu$}.
Moreover, the number
of elements $\Bv$ that are $i$-adjacent to $\Bu$ is
{$M_i+m_i-1$}, where
$M_i$ (resp. $-m_i$) is the largest (resp. smallest) move in the $i$th
direction among the steps of~$\cS$.  In particular, for small step models
($M_i=m_i=1$ {for all $i$}) there is one adjacent element in every direction.
\end{Lemma}

\begin{proof} {Let $\Bv=(u_1, \ldots, u_{i-1},v, u_{i+1}, \ldots, u_d)$ be
$i$-adjacent to $\Bu=(u_1, \ldots, u_d)$, and let us prove that the
coordinates of $\Bv$ are independent over~$\qs$.} Assume that there exists a
non-trivial polynomial $P(\Ba)$ with rational coefficients such that
$P(\Bv)=0$. Since the $u_i$'s are algebraically independent, $P(\Ba)$ must
depend on $a_i$. Hence $v$ is algebraic over $\qs(u_1, \ldots, u_{i-1},
u_{i+1}, \ldots, u_d)$. The same holds for $S(\Bv)$, and hence for $S(\Bu)$.
Since $S(\Ba)$ actually depends on $a_i$, this means that~$u_i$ is algebraic
over $\qs(u_1, \ldots, u_{i-1}, u_{i+1}, \ldots, u_d)$, which contradicts the
algebraic independence of the $u_j$'s. {The first statement of the lemma
follows, by induction on the length.}

Then, by expanding $S(\Bv)$ in powers of
$v$, we have
\[
{S(\Bv):=} P_{M_i}(u_1, \ldots, u_{i-1},u_{i+1}, \ldots,u_d) v^{M_i} +\cdots
+P_{-m_i}(u_1, \ldots, u_{i-1},u_{i+1}, \ldots,u_d) v^{-m_i}= S(\Bu).
\]
As the coordinates of $\Bu$ are algebraically independent, this
equation has $M_i+m_i$ solutions in~$v$. One of them is the trivial
solution $v=u_i$. Each root $v$ gives rise to an element $\Bv$ whose coordinates
are algebraically independent. In particular, $S_{a_i}(\Bv)\not = 0$,
which means that $v$ is not a multiple root of 
$S(\Bv)-S(\Bu)$. Hence this polynomial (in $v$) has distinct
roots. Removing the trivial root $v=u_i$ gives  $m_i+M_i-1$ distinct elements
$\Bv$ that are $i$-adjacent to $\Bu$.
\end{proof}

 \medskip \noindent{\bf Example $\boldsymbol {(d=1)}$.}
In dimension 1 the orbit of $x$  consists of all
solutions $x'$ of the equation $S(x)=S(x')$. It is thus  finite.\qee

\medskip \noindent{\bf Example: small steps with $\boldsymbol {d=2}$.} Let us
get back to the example of Section~\ref{sec:basic}. Then it can be checked
that two elements are adjacent if and only of one is obtained from the other
by applying $\Phi$ or $\Psi$. This will be generalized to all small step
models {(in arbitrary dimension)} in the proposition below.

Note however that in dimension 2, and beyond, the orbit may be infinite. This
happens for 56 of the 79 small step quadrant models~\cite{BoMi10}, for
instance when $\cS=\{\uparrow, \rightarrow, \swarrow, \leftarrow\}$, {in which
case $S(x,y) = y + x + \bx \by + \bx$.}

{For models with small steps, the orbit of $\Bx$ is indeed its
  orbit under the action of a certain group, as in the example of Section~\ref{sec:basic}.}
\begin{Proposition}\label{prop:small-group}
Assume that $\cS$ consists of small steps, that is, $\cS\subset \{-1,
0,1\}^d$.
Define $d$ bi-rational transformations $\Phi_1, \ldots, \Phi_d$ by: 
\[
\Phi_i(a_1, \ldots, a_d)= 
\left(a_1, \ldots, a_{i-1}, \frac 1 {a_i} \ \frac{ S_{i}^-(\Ba)}{S_i^+(\Ba)}, a_{i+1}, \ldots, a_d\right),
\]
where $\Ba=(a_1, \ldots, a_d)$ and $S_{i}^-(\Ba)$ (resp. $S_i^+(\Ba)$)
is the coefficient of $1/a_i$ (resp. $a_i$) in $S(\Ba)$.
Then the $\Phi_i$'s are involutions. If the  $a_j$'s are
algebraically independent over {$\qs$},
 then $\Ba$ and $\Phi_i(\Ba)$ are $i$-adjacent.

{Conversely,} let $\Bx=(x_1, \ldots, x_d)$ and {let} $\Bu=(u_1, \ldots, u_d)$
be in the orbit of $\Bx$. An element $\Bv$ of $\left(
{\bGK}\setminus\{0\}\right)^d$ is {$i$-adjacent} to $\Bu$ if and only if
$\Bv=\Phi_i(\Bu)$. Consequently, the orbit of $\Bx$ is indeed its orbit under
the action of a group, namely the group generated by the involutions $\Phi_i$.

{Finally, the length of two adjacent elements in the orbit of $\Bx$ differ by
$\pm1$.} 
\end{Proposition}

\begin{proof} To prove that $\Phi_i$ is an involution, we first observe that
  $S_i^+(\Ba)$ and $S_i^-(\Ba)$ are independent of~$a_i$. Hence,
  denoting $\Ba'=\Phi_i(\Ba)$, the $i$th coordinate of  $\Phi_i(\Ba')$ is
  \[
  \frac 1 {a'_i} \frac{ S_{i}^-(\Ba')}{S_i^+(\Ba')}
  = a_i \frac { S_{i}^+(\Ba)}{S_i^-(\Ba)}\frac{
    S_{i}^-(\Ba)}{S_i^+(\Ba)}=a_i.
  \]
  If the $a_j$'s are algebraically independent over {$\qs$}, then  $\Ba$ and $\Ba'$
  {are distinct}, differ  in their {$i$th}
  coordinate only, and, upon writing
\[
S(\Bx)= \frac 1 {x_i} S_i^-(\Bx) +S_i^0(\Bx)+ x_i S_i^+(\Bx),
\]
we can check that $S(\Ba)= S(\Ba')$. {Hence $\Ba$ and $\Phi_i(\Ba)$ are $i$-adjacent.}

  Now {let $\Bu$ be in}  the orbit of $\Bx$.
Write $S(\Bx)$ as above.
Note that $S_i^-(\Bx)$, $S_i^0(\Bx)$ and $ S_i^+(\Bx)$ are unchanged
if we only modify the $i$th coordinate of~$\Bx$. So if
{$\Bv\stackrel i \approx\Bu$},
{the fact that $S(\Bu)=S(\Bv)$ gives}
\[
 \frac 1 {u_i} S_i^-(\Bu) + u_i S_i^+(\Bu)
=
 \frac 1 {v_i} S_i^-(\Bu) + v_i S_i^+(\Bu), \quad \hbox{that is, } \quad
 S_i^-(\Bu) = u_i v_i S_i^+(\Bu).
\]
{By the above lemma,} the coordinates of $\Bu$ are algebraically
independent, hence
$S_i^+(\Bu)\not = 0$ and $\Bv$ must be $\Phi_i(\Bu)$. Conversely, we
have proved above that $\Phi_i(\Bu)$ is $i$-adjacent to $\Bu$. 
This concludes the {description of the orbit of $\Bx$}.

The proof of the final result was communicated to us by Andrew Elvey Price and
Michael Wallner, whom we thank for their great help. Clearly, if $\Bu$ and
$\Bv$ are two adjacent elements in the orbit of $\Bx$, their lengths differ by
$0, +1$ or $-1$. We want to exclude the value $0$, which amounts to saying
that in the graph whose vertices are the elements of the orbit, with edges
between adjacent elements, there is no odd cycle. Equivalently, this graph is
bipartite. In order to prove this, we define a sign $\vareps(\Bu) \in \{-1,
+1\}$ on elements $\Bu$ of the orbit of $\Bx$, which changes when an
involution $\Phi_i$ is applied.
{The sign is defined by
\[
\vareps(\Bu)= \left(\prod_{i=1}^d x_i\right) \det M(\Bu), \qquad
\hbox{where} \qquad M(\Bu)=\left( \frac 1 {u_i}
  \frac{\partial u_i}{\partial x_j}\right)_{1\le i,j \le d}.
\]
It is readily checked that $\vareps(\Bx)=1$, and 
this implies $\vareps(\Bu)=(-1)^{\hbox
{\scriptsize{length}}(\Bu)}$.}
Let us then take
$\Bv=\Phi_i(\Bu)$, and prove that $\vareps(\Bv)=-\vareps(\Bu)$. The
matrix $M(\Bv)$ only differs from $M(\Bu)$ in the $i$th row. Let us
denote
\[
\Phi_i(\Ba)= 
\left(a_1, \ldots, a_{i-1},\frac 1 {a_i} R_i(\Ba), a_{i+1}, \ldots, a_d\right),
\] 
where $R_i(\Ba)=S_i^-(\Ba)/S_i{^+}(\Ba)$ only depends on the variables
$a_1, \ldots, a_{i-1}, a_{i+1}, \ldots, a_d$. Then for $1\le j \le d$,
the $(i,j)$ entry of $M(\Bv)$ is
\begin{align*}
  \frac 1{v_i}  \frac{\partial v_i}{\partial x_j} &= \frac{u_i}{R_i(\Bu)}
\left( -\frac 1{u_i^2} R_i(\Bu) \frac{\partial u_i}{\partial x_j} +
  \sum_{k\not = i} \frac{\partial R_i}{\partial a_k}(\Bu)
  \frac{\partial u_k}{\partial x_j}\right)
\\
&= -\frac 1 {u_i} \frac{\partial u_i}{\partial x_j} + \frac{u_i}{R_i(\Bu)}
  \sum_{k\not = i} \frac{\partial R_i}{\partial a_k}(\Bu)
  \frac{\partial u_k}{\partial x_j}.
\end{align*}
Upon subtracting from the $i$th row of $M(\Bv)$ its $k$th row,
multiplied by ${u_i u_k} \partial R_i/\partial a_k(\Bu)/R_i(\Bu)$,
for $1\le k \not =i \le d$, we see that $\det M(\Bv)$ is also the
determinant of the matrix obtained
from $M(\Bu)$ by changing the sign of all elements of the $i$th
row, which concludes the proof.
\end{proof}

\medskip \noindent{\bf Example D (continued): large steps with $\boldsymbol{d=2}$.} 
Let us take $\cS=\{\bar 20,\bar 11, 02,1\bar 1\}$, so that
\[
S(x,y)= \bx^2+\bx y +y^2+x\by.
\]
We will incrementally construct the orbit of $(x,y)$. This example
should provide the intuition for the algorithm given in the next
subsection.

We start from  $(x,y)$  and want to determine
which elements $(X,y)$  are 1-adjacent to it; that is, to find the solutions to
$S(X,y)=S(x,y)$ with $X\not = x$. We have
\[
S(X,y)-S(x,y)=\frac {(X-x) \left(x^2X^2-y(1+xy)X-xy\right)}{x^2 y X^2}.
\]
Hence the two elements that are 1-adjacent to
 $(x,y)$ are $(x_1,y)$ and $(x_2,y)$, where $x_1$ and $x_2$ are the
 two roots of $P_1(X,x,y):=x^2X^2-y(1+xy)X-xy$ {(when solved for
   $X$)}. The  $x_i$'s can be taken
 as Laurent series 
 in $x$ with coefficients in $\qs[y,\by]$:
\[
x_1= y{x}^{-2}+{y}^{2}{x}^{-1} -x_2\qquad \hbox{and} \qquad 
x_2=-x+y{x}^{2}-{y}^{2}{x}^{3}+\left( {{y}^{3}+\by}\right){x}^{4}+ O({x}^5).
\]
Similarly, we find that the two elements that are 2-adjacent to
$(x,y)$ are $(x,y_1)$ and $(x,y_2)$, where $y_1$ and $y_2$ are the
roots of  $Q_1(Y,x,y):=xyY^2+y(1+xy)Y-x^2$. But $Q_1(Y,x,y)$ {coincides
  with $P_1(1/Y,x,y)$} (up to a factor of $Y^2$),
 thus we take $y_1=\bar x_1 := 1/x_1$ and $y_2=\bar
x_2 := 1/x_2$. We have now obtained five elements in the orbit of $(x,y)$ (one
can follow the construction on Figure~\ref{fig:orbit-quadrangulations}).

Now we want to find the elements $(x_1,Y)$ that are 2-adjacent to
$(x_1,y)$. In principle, we should thus solve $S(x_1,Y)=S(x_1,y)(=S(x,y))$, but
we prefer not to handle equations with algebraic coefficients (like
$x_1$). So instead, we consider the \emm {polynomial} system,
\[
P_1(X,x,y)=0, \qquad S(X,Y)=S(x,y),
\]
whose solutions $(X,Y)$ are the pairs $(x_i,Y)$ belonging to the orbit. Upon
eliminating $X$ between these two equations, we find that $Y$ is
necessarily  either $y$, or $\bx$, or one of the series $\bx_i$. Upon
checking that $S(x_1, \bx_1)\not = S(x,y)$, we conclude that the two
elements that are 2-adjacent to $(x_1,y)$ are $(x_1,\bx)$ and $(x_1,
\bx_2)$. Symmetrically, $(x_2,\bx)$ and $(x_2, \bx_1)$ are 2-adjacent
to $(x_2,y)$. We now have 9 elements in the orbit.

In order to find the elements that are 1-adjacent to $(x, \bx_i)$, for
$i=1,2$, we study  similarly the {polynomial} system
\[
Q_1(Y,x,y)=0, \qquad S(X,Y)=S(x,y)
\]
and conclude that $(\by, \bx_1)$ and $(x_2, \bx_1)$ are 1-adjacent to
$(x, \bx_1)$ while $(\by, \bx_2)$ and $(x_1, \bx_2)$ are 1-adjacent to
$(x, \bx_2)$. We have reached 11 elements.

At this stage, we still need one element that would be 1-adjacent to
$(x_1, \bx)$ and $(x_2, \bx)$, and one element that would be
2-adjacent to $(\by, \bx_1)$ and $(\by, \bx_2)$. We address the first
problem by solving $S(X,\bx)=S(x,y)$, and find that $(\by,\bx)$ in fact
solves both problems. The orbit is now complete, and contains 12 elements.

\begin{figure}[htb]
  \centering
  \scalebox{0.9}{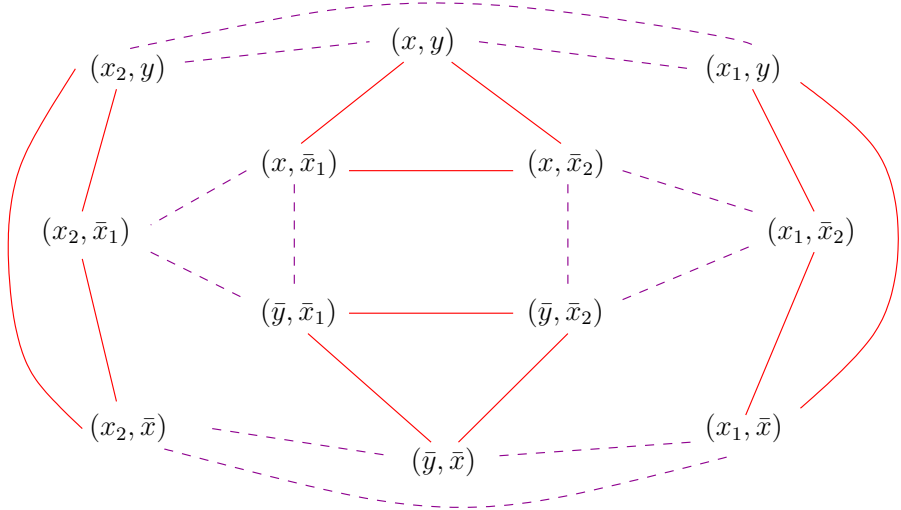}

  \caption{The orbit of $\cS=\{\bar 20,\bar 11, 02,1\bar 1\}$.  The
    values $x_1$ and $x_2$ are the roots of
    $P_1(X,x,y)=x^2X^2-y(1+xy)X-xy$. The values $\bxun$ and $\bxde$ are
    their reciprocals.
The  dashed (resp. {solid}) edges join 1-adjacent (resp. 2-adjacent) elements.}
  \label{fig:orbit-quadrangulations}
\end{figure}

%==========================================
\subsection{An algorithm that detects finite orbits (case $\boldsymbol{d=2}$)}
\label{sec:algo}
%================================================

Given a model in dimension $d=2$ we now describe a (semi-)algorithm that stops
if and only if the orbit is finite. This algorithm constructs incrementally
two sets $\PP$ and $\cQ$ of irreducible polynomials in $X$ and $Y$,
respectively, with coefficients in $\qs(x,y)$. It starts with $\mathcal
P=\{X-x\}$ and $\cQ=\{Y-y\}$, and both polynomials are declared non-treated.
At each stage, the algorithm chooses a non-treated polynomial in $\PP \cup
\cQ$, say $Q\in \cQ$, and constructs a new polynomial $P'(X,x,y)$, which is
the resultant in $Y$ of $ Q(Y,x,y) $ and { the numerator of the Laurent
polynomial $S(x,y)-S(X,Y)$ (namely $(xX)^{m_1}(yY)^{m_2}(S(x,y)-S(X,Y))$,
where $-m_1$ is the smallest move in the $x$-direction and similarly for
$m_2$)}. Then the algorithm adds to $\mathcal P$ every irreducible factor of
$P'$, and the {new} factors are declared non-treated. The algorithm treats
symmetrically polynomials of $\PP$. These stages are repeated as long as there
are non-treated polynomials.

We recall that $\overline\GK$ denotes an algebraic closure of $\GK:=\qs(x,y)$.

\begin{Proposition}\label{prop:algo}
  The following two properties hold at each stage of the algorithm:
  \begin{enumerate}
    \item[$(i)$] the set $\PP$ contains no element of $\qs[X]$;
      moreover, for  $P \in \mathcal P$ and  $x' \in \overline \GK$ such that $P(x',x,y)=0$, there exists $Q\in \cQ$ and $y'\in \bGK$   such that $Q(y',x,y)=0$ and $(x',y')$ is in the orbit of $(x,y)$,
    \item [$(ii)$] symmetrically, the set $\cQ$ contains no element of
      $\qs[Y]$; moreover, for   $Q \in \mathcal Q$ and $y' \in \bGK$ such that $Q(y',x,y)=0$, there exists  $P\in \cQ$ and $x'\in \bGK$   such that $P(x',x,y)=0$ and $(x',y')$ is in the orbit of $(x,y)$.
  \end{enumerate}
  The algorithm stops if and only if the orbit of $(x,y)$ is finite. In this case, the converse of $(i)$ and $(ii)$ holds, that is:
  \begin{enumerate}
    \item [$(iii)$]for every $(x',y')$ in the orbit of $(x,y)$, the minimal polynomials of $x'$ and $y'$ over $\qs(x,y)$ belong respectively to $\mathcal P$ and $\cQ$.
  \end{enumerate}

\end{Proposition}

Note that the sets $\PP$ and $\cQ$ do not determine completely the orbit: one
still has to decide, for each $x'$ that solves a polynomial of $\PP$, which
$y'$ (taken from the roots of the polynomials of~$\cQ$) go with it in the
orbit, as was done in Example D above.

\begin{proof}
Let us first prove $(i)$ and $(ii)$, by induction on the number of stages
performed by the algorithm. Both properties obviously hold at the
initialization step, where $\mathcal P= \{X-x\}$ and $\cQ=\{Y-y\}$.

Now assume that they hold at some stage, and that we treat a
polynomial $Q\in \cQ$  as described at the beginning of
Section~\ref{sec:algo}. Let us prove that the extended collections  of
polynomials still satisfy  $(i)$ and $(ii)$. Clearly $(ii)$ still
holds, since we have not extended $\cQ$. So let us check~$(i)$. It
suffices to check it for the  factors of $P'$ that we have added to
$\mathcal P$. So let us take one of these factors, and let $x'$ be one
of its roots. Then $x'$ is a root of $P'(X,x,y)$.  The properties of
the resultant imply that there exists $y'$ such that $Q(y',x,y)=0$ and
\begin{equation}\label{S-equal}
x'^{m_1} y'^{m_2} \left(S(x,y)-S(x',y')\right)=0,
\end{equation}
 where $-m_1$ (resp. $-m_2$) is the smallest move along the $x$-axis
 (resp. the $y$-axis).
By Property~$(ii)$ applied to $Q$ and $y'$, there exists an element  $x'' \in
\overline \GK$  such that $(x'',y')$ is in the orbit. By
Lemma~\ref{lem:ind}, $x''$ and $y'$ are algebraically independent over
$\qs$, and in particular $y'$ is not an algebraic {number}. If $x'=0$,
then~\eqref{S-equal} tells us that the coefficient of $x^{-m_1}$ in
$S(x,y)$, evaluated at $y=y'$, vanishes, which would make $y'$
algebraic, a contradiction. Thus $x'\not =0$, {$y'\not = 0$,} and
$S(x,y)=S(x',y')$. Hence 
$S(x',y')=S(x'',y')$, which shows that $(x',y')$ is adjacent to
$(x'',y)$, and thus is in the
orbit of $(x,y)$. In particular, $x'$ and $y'$ are algebraically
independent {over $\qs$}, thus $x'\not \in {\overline \qs}$, which means that its minimal
polynomial $P$ is not in $\qs[X]$.

\medskip Now assume that the algorithm stops; that is, that there are no more
non-treated polynomials. Let us prove $(iii)$ by induction of the length of
$(x',y')$. If $\ell=0$, then $(x',y')=(x,y)$ and we have precisely initialized
$\PP$ and $\cQ$ with the minimal polynomials of $x$ and $y$. Now assume that
$(iii)$ holds for length $\ell-1$, and that $(x',y')$ has length $\ell$.
Without loss of generality, we can assume that $(x',y')\approx (x'',y')$,
where $(x'',y')$ has length $\ell-1$. By the induction hypothesis, the minimal
polynomial $Q$ of $y'$ belongs to $\cQ$, so we only need to consider~$x'$. The
polynomials (in $Y$) $Q(Y,x,y)$ and $X^{m_1}
Y^{m_2}\left(S(x,y)-S(X,Y)\right)$ have a common root (namely $y'$) when
$X=x'$. Hence their resultant $P'(X,x,y)$ must have $x'$ as a root. This
implies that one of the factors of $P'$ is the minimal polynomial of $x'$, and
this factor is added to $\PP$ when the algorithm treats the polynomial $Q$
(unless it was already in $\PP$).

We have thus established $(iii)$, assuming the algorithm stops. In this case
$\PP$ and $\cQ$ are finite so $(iii)$ implies that the orbit is finite.

Conversely, assume that the orbit is finite. By $(i)$, every $P\in\PP$ must be
the minimal polynomial of some $x' \in \bGK$ such that $(x',y')$ is in the
orbit for some $y'$. Hence $\PP$ cannot grow indefinitely. A similar argument
applies to $\cQ$, and the algorithm has to stop. 
\end{proof}

%===============================================================
\subsection{Infinite orbits and the excursion exponent}
\label{sec:infinite-orbit}
%==================================================================

We now describe an approach, of wide applicability, to prove that a model has
an infinite orbit. It generalizes a fixed point argument applied to quadrant
walks with small steps in~\cite[Thm.~3]{BoMi10} {(see also~\cite{Du} for an
application to 3D walks with small steps)}. It also constructs a group of
transformations which generates part of the orbit of $\Bx$. In the
2-dimensional case, it establishes a connection with the asymptotic proof of
non-D-finiteness developed in~\cite{BoRaSa14}. One outcome will be the
following convenient criterion for 2-dimensional models.

  \begin{Theorem} \label{thm:theta}
Let $\cS\subset \zs^2$ be a {step set} that is not contained in a
half-plane, and {contains an element of}~$\ns^2$. Then the step
polynomial $S(x,y)$ has a unique \emm critical
point, $(a,b)$ in $\rs_{>0}^2$ (that is, a solution of $S_x(a,b)=S_y(a,b)=0$), which satisfies $S_{xx}(a,b)>0$ and
$S_{yy}(a,b)>0$. Define
\[
c= \frac{
  S_{xy}(a,b)}{\sqrt{S_{xx}(a,b)S_{yy}(a,b)}}.
\] Then $c\in [-1,1]$ can be written as $\cos \theta$.
 If
$\theta$ is not a rational multiple of $\pi$, then the orbit of $\cS$
is infinite, and  the series $Q(x,y;t)$ is not D-finite.
  \end{Theorem}
Note that this result is algorithmic: the quantities $a,b, c$ are algebraic
over $\qs$ and one can compute their minimal polynomials. Saying that $\theta$
is a rational multiple of $\pi$ amounts to saying that the solutions of
$z+1/z=2c$ are roots of unity, so that their minimal polynomials are
cyclotomic. This can be checked algorithmically. In Section~\ref{sec:m21} we
apply this theorem systematically to the 13 110 models having steps in $\{-2,
-1, 0, 1\}^2$ and at least one large step. Combined with the algorithm that
detects finite orbits, it determines the size of the orbit for all but 16
models. (These 16 models turn out to have an infinite orbit, see
Section~\ref{sec:embarassing}).

 The above theorem also shows that the calculations performed in~\cite{BoMi10}
to prove that 51 small step models have an infinite group are equivalent to
those performed in~\cite{BoRaSa14} to prove that these~51 models have a
non-D-finite \gf.

%======================================================
\subsubsection{A group acting on the orbit}
\label{sec:group}
%======================================================
We begin with the part of the above theorem that deals with the size
of the orbit. In fact, we have a more general result that holds  for  models in
$d$ dimensions. So let  $\cS \in \zs^d$, and assume that there exists
a point $\Ba:=(a_1,\ldots, a_d)$ such that $S_{x_1}(\Ba)
= \partial S/\partial {x_1}(a_1,
\ldots, a_d)=0$. {If $I(X,\Bx)$ denotes the Laurent polynomial 
\[ I(X,\Bx) = \frac{S(X,x_2,\dots,x_d)-S(x_1,\dots,x_d)}{X-x_1} \]
(after {normalizing the rational function}) then $I(a_1,\Ba)= S_{x_1}(\Ba)=0$. Assume now that
$S_{x_1 x_1}(\Ba)\not =0$, so that $I_X(a_1,\Ba)=S_{x_1x_1}(\Ba)/2 \neq 0$.} By the implicit function theorem (in its
analytic form), there exists a unique analytic function
$\X_1(x_1, \ldots, x_d)$ defined in a neighborhood of $\Ba$,
satisfying $\X_1(\Ba)={a_1}$ and 
\begin{equation}\label{X1-char}
I(\X_1(\Bx), \Bx)=\frac{S(\X_1(\Bx), x_2, \ldots, x_d)-S(\Bx)}{\X_1(\Bx)-x_1}=0.
\end{equation}
The expansion of $\X_1(\Bx)$ around $\Ba$ can be computed
inductively. Writing $\Bx=\Ba+\Bu$, we have
\begin{equation}\label{X1-exp}
\X_1(\Bx)= a_1 - u_1 - \frac 2 {S_{x_1x_1}(\Ba)} \sum_{i=2}^d S_{x_1x_i}(\Ba)
u_i + \cdots,
\end{equation}
the missing terms being of degree at least 2 in the $u_i$'s. We define the
transformation $\Phi_1$ by $\Phi_1(\Bx)= (\X_1(\Bx), x_2, \ldots,
x_d)$. Clearly,  {$\Phi_1(\Bx)$ is 1-adjacent to $\Bx$} and thus lies in the orbit of $\Bx$ (which we construct
 in an algebraic closure of $\qs(\Bx)$ containing power series in
the $u_i$'s). {Since $\Phi_1(\Ba) = \Ba$}, we can
iterate $\Phi_1$. In particular,
\[
\Phi_1\circ \Phi_1(\Bx)= ( \X_1(\X_1(\Bx), x_2, \ldots,x_d), x_2,
\ldots, x_d)
\]
satisfies
\[
S( \Phi_1\circ \Phi_1(\Bx))= S(\Phi_1(\Bx))=S(\Bx),
\]
by~\eqref{X1-char}. Hence  either $\Phi_1\circ \Phi_1$ is the identity, or 
\[
I(\Phi_1\circ \Phi_1(\Bx)) = \frac{S(\Phi_1\circ \Phi_1(\Bx))-S(\Bx)}{\X_1(\X_1(\Bx), x_2,
  \ldots,x_d)-x_1}=0,
\]
which means that the function $\tilde \X_1:\Bx\mapsto \X_1(\X_1(\Bx), x_2,
\ldots,x_d)$ satisfies the same conditions as~$\X_1$. By uniqueness of $\X_1$,
this would imply that $\tilde \X_1=\X_1$: but this is impossible as
$\X_1(\Bx)$ has linear part $a_1-u_1+ \cdots$ while $\tilde \X_1$ has linear
part $a_1+u_1+ \cdots$ (by~\eqref{X1-exp}). Hence $\Phi_1$ is an involution.

Assume now that $\Ba$ is a critical point of $S$, that is, $S_{x_i}(\Ba)=0$
for $i=1, \ldots, d$. {Assume moreover that $S_{x_ix_i}(\Ba)\not = 0$ for all
$i$.} We then define similarly the transformations $\Phi_i$ for $i=1, \ldots,
d$. Still writing $\Bx=\Ba+\Bu$, each $\Phi_i$ leaves the constant term of
$\Bx$ unchanged, so we can compose them and they form a group $G$. For any
$\Theta$ in this group, $\Theta(\Bx)$ lies in the orbit of $\Bx$. If the orbit
of $\Bx$ is finite, $G$ is finite as well, and every $\Theta\in G$ has finite
order. The expansion of $\Theta$ around $\Ba$ reads:
\[
\Theta(\Ba+\Bu)= \Ba+ \Bu\, J(\Ba) + \hbox{quadratic terms in the } u_i,
\]
hence the Jacobian matrix $J(\Ba)$ must have finite order. This means that its
eigenvalues are roots of unity, which, once again, can be checked
algorithmically.

We now restrict the discussion to the 2-dimensional case, in order to
lighten notation.  
We denote $\Phi:=\Phi_1$ and $\Psi:=\Phi_2$,
$\Bx=(x,y)$, $\Ba=(a,b)$ and $\Bu=(u,v)$.  For
$\Theta:=\Psi\circ \Phi$, we have  
\[
J(a,b):=\left(
  \begin{array}{cc}
    -1& -\eta \\
\nu & \eta \nu-1 
  \end{array}
\right)\]
where
\[
\eta= \frac{2 S_{xy}(a,b)}{S_{xx}(a,b)} \qquad \hbox{and } \qquad 
\nu= \frac{2 S_{xy}(a,b)}{S_{yy}(a,b)}.
\]
 The eigenvalues of $J$ are  the roots of
\[
\lambda^2-(\eta\nu-2)\lambda+1
\]
and, as the orbit is finite, they must equal $e^{\pm 2i \theta}$ for $\theta$
a rational multiple of $\pi$. That is, 
\[
\lambda^2-(\eta\nu-2)\lambda+1=
(\lambda-e^{2i\theta})(\lambda-e^{-2i\theta}).
\]
Extracting the coefficient of $\lambda$ gives the following
proposition.
\begin{Proposition}\label{prop:ab}
Consider a two-dimensional model $\cS$, and a critical point $(a,b)$
of $S(x,y)$ such that $S_{xx}(a,b)S_{yy}(a,b) \not = 0$. Then one can
define involutions $\Phi$ and $\Psi$ as described above. If the orbit
is finite, then $\Theta:=\Psi\circ \Phi$ has finite order. In
particular,  there exists a rational multiple of~$\pi$, denoted
$\theta$, such that
\[
\frac{ S_{xy}(a,b)^2}{S_{xx}(a,b)S_{yy}(a,b)} =\cos^2 \theta .
\]
\end{Proposition}
We can now prove the part of Theorem~\ref{thm:theta} that deals with the orbit
size. Since $\cS$ is not contained in a half-plane, there exists a unique \emm
positive, critical point $(a,b)$ (an argument is given in the proof
of~\cite[Thm.~4]{BoRaSa14}). The derivatives $S_{xx}$ and $S_{yy}$ are
positive at this point (because every monomial $x^i y^j$ gives a non-negative
contribution, and one of them at least gives a positive contribution), and
thus the above proposition applies. \qed

%==================================================
\subsubsection{The excursion exponent}
%==================================================
We will now show that in the 2-dimensional case, the above criterion is
closely related to an asymptotic result that has been used as a criterion for
the non-D-finiteness of $Q(x,y;t)$ in~\cite{BoRaSa14}. This result originally
applies to \emm strongly aperiodic models, only, and it will take us a bit of
work to obtain a version that is valid for periodic models as well. Given a
model $\cS$, we denote by $\Lambda$ the lattice of $\zs^2$ spanned by its
steps. Then $\cS$ is \emm strongly aperiodic, if for any point $x\in\Lambda$,
the lattice $\Lambda_x$ spanned by the points $x+s$ for $s\in \cS$, which is
clearly a sublattice of $\Lambda$, coincides with $\Lambda$. For instance,
Kreweras' model $\{\nearrow, \leftarrow, \downarrow\}$ is \emm not, strongly
aperiodic: one has $\Lambda=\zs^2$, but for $x=(1,0)$, the lattice $\Lambda_x$
only contains points $(i,j)$ such that $i+j$ is a multiple of $3$.

Given a model $\cS$, and a point $(i,j)$ in $\zs^2$, we denote by $w(i,j;n)$
the number of $n$-step walks going from $(0,0)$ to $(i,j)$ consisting of steps
taken in $\cS$ \emm without the quadrant condition,. We call any walk starting
and ending at the same point an \emm excursion,.

\begin{Proposition}\label{prop:lattice}
 Let $\cS\subset \zs^2$ be a model that is not contained in a
 half-plane, and denote by $\Lambda$ the lattice of $\zs^2$ generated
 by $\cS$. Then there exists an integer
$p$, called the \emm period, of $\cS$, such that for any $(i,j)\in \Lambda$, there exists $r\in \llbracket0, p-1\rrbracket$ with
$
w(i,j;n) =0$ if  $n \not \equiv r \!\mod p$ and $w(i,j;n) >0$ if
$n=mp+r$ and $m$ is large enough.

The model $\cS$ is strongly aperiodic if and only if $p=1$.
\end{Proposition}
\begin{proof}
  Several ingredients of the proof are borrowed  from
  Spitzer~\cite[Sec.~I.5]{spitzer}, who deals with recurrent random
  walks and only considers the case $\Lambda=\zs^2$. The fact that all
  points {$(i,j)$} of $\Lambda$ can be reached {from
    $(0,0)$} is closely related to Farkas'
  Lemma~\cite[Sec.~7.3]{schrijver}. 

Let $\cN= \{n \ge 0: w(0,0;n)\not = 0\}$. Since one can concatenate two walks
starting and ending at the origin, $\cN$ is an additive semi-group of $\ns$.
Our first objective is to prove that it is not reduced to $\{0\}$, that is,
that there exist non-empty excursions.

Let $s$ be a non-zero vector of $\cS$. Since $\cS$ is not contained in
a half-plane, there exists another non-zero vector of $\cS$, say $s'$,
such that the  wedge formed by the pair $s,s'$ forms  an angle
$\phi\in (0, \pi)$. 
Let us choose  $s'$
so as to  maximize $\phi$ in this interval (Figure~\ref{fig:wedge}). Since $s$ and $s'$ form a basis of~$\rs^2$, any other vector $s''$ of $\cS$ can be written as $s''=\alpha s+\beta s'$ for  a unique pair  $(\alpha,\beta) \in \qs^2$. Since  $\cS$ is not contained in a half-plane,
 there must exist a  vector $s''$ in~$\cS$ such
that $\alpha$ is negative. By  maximality
of $\phi$, this vector is such  that $\beta \le 0$.  Writing $\alpha=-a/d$ and
$\beta=-b/d$ with $a, d$ positive integers and $b$ a non-negative integer, we conclude that
$as+bs'+ds''=0$, which shows that the walk starting at the origin and
formed of $a$ copies of
$s$, $b$ copies of $s'$ and $d$ copies of $s''$ ends at the origin as
well. Thus there exist non-empty excursions. Moreover, we have
\[
-s= (a-1)s+bs'+ds''.
\]
Since $a$ is a positive integer, {and $s$ is an arbitrary element of $\cS$,} this proves that the set of endpoints of walks starting at the origin
is not only a semi-group of $\zs^2$ (again, by a concatenation
argument), but in fact the entire lattice~$\Lambda$. 

\begin{figure}
  \centering
  \scalebox{0.9}{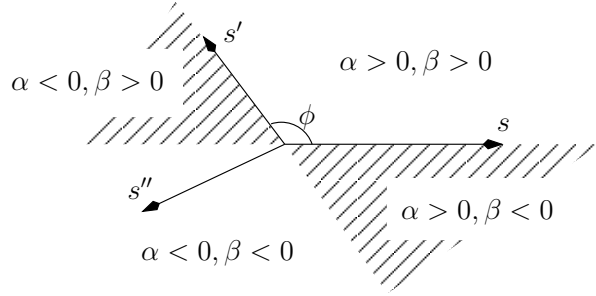}
  \caption{On the existence of excursions.}
\label{fig:wedge}
\end{figure}

We have established that $\cN\not = \{0\}$. Let $p$ be the greatest common
divisor of the elements of~$\cN$. The structure of semi-groups of $\ns$ are
well-understood: $\cN \subset p \ns$, and $pm \in \cN$ for any large enough
$m$. We have thus proved the first statement of the proposition for
$(i,j)=(0,0)$, {with $r=0$}.

Now let $(i,j)\in \Lambda$. We have proved above that there exists a walk
going from $(0,0)$ to $(i,j)$. Assume that there are two such walks $w$ and
$w'$, and choose a walk $w''$ from $(i,j)$ to $(0,0)$. Then both $ww''$ and
$w'w''$ are excursions, hence they must have length $0$ modulo~$p$.
Consequently, $w$ and $w'$ must have the same length modulo $p$, say $r$.
Finally, by concatenating a large excursion to a walk ending at $(i,j)$, we
see that for $m$ large enough, there is a walk of length $pm+r$ from the
origin to $(i,j)$.

The equivalence between strong aperiodicity and $p=1$ can be proved by
mimicking the corresponding part of the proof of Proposition P1
in~\cite[Sec.~I.5]{spitzer}. 
\end{proof}

In the following theorem, we  assume that each step $s$ of $\cS$ is weighted
by a positive weight~$\om_s$. This means that the ``number''
$q(i,j;n)$ is actually the sum of the weights of all quadrant walks
from $(0,0)$ to $(i,j)$, the weight of a walk being the product of
the weights of its steps. In this context, the step polynomial is
\[
S(x,y)=\sum_{s=(s_1, s_2)\in \cS} \om_s x^{s_1} y^{s_2}.
\]
\begin{Definition}
   Given a model $\cS\subset \zs^2$, a point $(i,j)\in \ns^2$ is  \emm reachable from infinity, if there exists
a quadrant walk that starts from 
a point $(k, \ell) \in (i,j)+ \zs_{>0}^2$ and ends at $(i,j)$. 
\end{Definition} Note
that in this case, $(k, \ell)$ itself is reachable from
infinity. Moreover, upon
concatenating several copies of the walk, we can find a starting
point $(k', \ell')$ with arbitrarily large coordinates,
and a quadrant walk from this point to $(i,j)$. Finally,
Proposition~\ref{prop:lattice} implies that if $\cS$ is not contained in a
half-plane, then any point with large enough coordinates is reachable
from infinity.

We can now complete the asymptotic result of Denisov and
Wachtel~\cite{denisov-wachtel} with a statement that holds in the periodic
case.

\begin{Theorem}\label{thm:exponent}
Let $\cS\subset \zs^2$ be a model that is not contained in a
half-plane and {contains an element of}~$\ns^2$. Then the step polynomial $S(x,y)$ has a unique critical
point $(a,b)$ in $\rs_{>0}^2$, which satisfies $S_{xx}(a,b)>0$ and
$S_{yy}(a,b)>0$. Define
\[
\mu= S(a,b), \qquad   \qquad c= \frac{
  S_{xy}(a,b)}{\sqrt{S_{xx}(a,b)S_{yy}(a,b)}}\qquad \hbox{and } \qquad
\alpha=-1 -\pi/\arccos(-c).\]
Assume first that $\cS$ is strongly aperiodic. Then if  $(i,j)$ is reachable from infinity, there exists a positive constant $\kappa $
such that, as $n$ goes to infinity,
\[
q(i,j;n)\sim \kappa\,  \mu^n n^{\alpha}.
\]
If   $\cS$ is not strongly aperiodic and has period
$p>1$, 
define 
\[
\bcS= \{ s_1+\cdots + s_p, \ (s_1, \ldots, s_p)\in \cS ^p \},
\]
and let $\overline \Lambda$ be the lattice spanned by the vectors
of $\bcS$. Then if $(i,j)\in \overline \Lambda$  is  reachable from
infinity {for $\cS$},  there exist positive constants $\kappa _1$ and $\kappa _2$  such that for $n=pm$ and $m$ large enough,
\begin{equation}\label{bounds}
\kappa _1 \, \mu^n n^{\alpha} \le q({i,j};n)\le \kappa_2\, \mu^n n^{\alpha}.
\end{equation}
We call $\alpha$ the \emm excursion exponent., 
\end{Theorem}
\noindent{\bf Remarks}\\
{\bf 1.} It is very likely that an asymptotic estimate holds as well in the
periodic case (see~\cite[p.~3/4]{duraj-wachtel}), but the proof does not seem
to be written down, and we will content ourselves with the above
bounds. \\
{\bf 2.} The reachability condition, which is somewhat
implicit in~\cite{denisov-wachtel}, is important. Consider for
instance the (strongly aperiodic) model $\cS=\{10,01, 1  \bone, \bone 1, \bar 3 2, 2 \bar
3\}$. Then for $n>0$,
\[
q(0,0;n)= 0 \quad \hbox{and} \quad q(1,0;n)=1,
\]
while 
\[
q(1,1;n) \sim \kappa \, 6^n n^\alpha
\]
with $\alpha =-1-\pi/\arccos(7/8)$. The reason for these different
asymptotic behaviours is that the points $(0,0)$ and $(1,0)$ are not reachable
from infinity, while $(1,1)$ is.
Similarly, any asymptotic result for quadrant walks starting from a
given point $(k,\ell)$ should require that there exists a quadrant
walk  that starts from $(k, \ell)$ and ends in $(k, \ell)
+\zs_{>0}^2$ (we say that $(k, \ell)$ \emm reaches infinity,). Given that we have assumed that $\cS$ contains a point
of $\ns^2$ and is not included in a half-plane, this condition holds
here for any $(k, \ell)$.
  \begin{proof} [{Proof of Theorem~\ref{thm:exponent}}]
In the aperiodic case, the proof can be copied verbatim from the proof
of Theorem~4 in~\cite{BoRaSa14}. One considers  an
underlying random walk  and  normalizes it  into a walk whose
projections on the $x$- and $y$-axes are centered, reduced, and of
covariance $0$. The key result is then a local limit theorem of
Denisov and Wachtel 
 that applies to such walks~\cite[Thm.~6]{denisov-wachtel} (note that
 one should assume in that theorem that $V(x)>0$ and $V'(y)>0$, which
 holds if $x$ reaches infinity and $y$ is reached from infinity). 

We  thus focus on the periodic case. The idea is to consider $p$
consecutive steps of a walk as a single generalized step to obtain a
strongly aperiodic walk. More precisely, let us define $\bcS$ as above,
and define the weight  of a step $s$ of $\bcS$ to be
\[
\bar \om_s= \sum_{\atopfix{(s_1, \ldots, s_p)\in \cS^p}{ s_1+\cdots + s_p=s}}
\om_{s_1} \cdots \om_{s_p}.
\]
We denote with bars all quantities that deal with the model
$\bcS$. For instance, $\bar w(i,j;n)$ is the (weighted)
number of walks going from $(0,0)$ to $(i,j)$  in $n$ steps taken
from $\bcS$. By the definition of~$p$, we have $w(0,0;pn)=\bar w(0,0;n) >0$
for $n$ large enough, hence $\bcS$ is strongly aperiodic (on the lattice
$\overline \Lambda$ that it generates). Observe that
\[
\bar S(x,y)=S(x,y)^p, \qquad (\bar a, \bar b)=(a,b), \qquad \bar \mu=
\mu^p, \qquad   \bar c=c \qquad \hbox{and} \qquad {\bar \alpha= \alpha}.
\]
Note that if $(i,j) \in \overline \Lambda$ is reachable from infinity in the
model $\cS$, then it is also reachable from infinity in the model $\bcS$.
Since we will consider both models $\cS$ and $\bcS$ at the same time, we will
often refer to a walk with steps in $\cS$ as an \emm $\cS$-walk.,

\medskip
\noindent {\bf Upper bound.}
   A quadrant walk from $(0,0)$ to $(i,j)$
consisting of $n=pm$ steps of $\cS$ can be seen as a quadrant walk
from $(0,0)$ to $(i,j)$  consisting of $m$ steps of $\bcS$ (the
converse is not true in general: for instance, taking a step $(1,0)$
in $\bcS$ may correspond to a sequence $(-1, 0),( 2,0)$ of steps of
$\cS$ and involve crossing the $y$-axis). Hence
\[
q({i,j};pm) \le \bar q ({i,j};m).
\]
Since $\bcS$ is strongly aperiodic,  and $(i,j)$ reachable from
infinity in $\bcS$, the right hand-side is asymptotic to
$\kappa\, {(\mu^p)^m} m^\alpha$ for some positive $\kappa$, which gives the
desired upper bound on $q({i,j};n)$. 

\medskip
\noindent{\bf Lower bound.} Since $(0,0)$ reaches infinity, and
$(i,j)$ is reachable from infinity, we can pick two quadrant walks
$w_1$ and $w_2$ satisfying the following conditions:
\begin{itemize}
\item $w_1$ goes from $(0,0)$ to a point $x=(i_1,j_1)$, whose
  coordinates are larger than $pM$,  where $M$ is the
maximal norm of a step of $\cS$. Moreover, $w_1$ has length $pm_1$;
\item $w_2$ goes from some point $y=(i_2,j_2)$ to $(i,j)$,
  and the coordinates of $y$ are large enough for $y-x$ to be
  reachable from infinity in the model $\bcS$ {(in particular,
    $i_2\ge i_1$ and $j_2\ge j_1$)}. Moreover, $w_2$ has length $pm_2$.
\end{itemize}
  Now take a quadrant walk $w$ from $(0,0)$ to $y-x$ consisting of $m$ elements of $\bcS$: if we replace every step
$\sigma=s_1+\cdots +s_p$ (with each $s_k \in \cS$), by the sequence $s_1, \ldots, s_p$, the
resulting walk $\tilde w$ may exit the
quadrant. But it will remain in the translated quadrant $[-pM,
\infty)^2$. Thus, if we translate $\tilde w$ so that it starts at $x$, it
will remain in the quadrant $\ns^2$, and end at $y$. Adding $w_1$ as a prefix
and $w_2$ as a suffix  gives a quadrant walk of length $n=p(m_1+m_2+m)$
ending at $(i,j)$.  Consequently,
\[
q(i,j;n) \ge c\, \bar q (i_2-i_1, j_2-j_1;m)
\]
for some positive constant $c$ that depends on the weights of  $w_1$ and $w_2$. Since $\bcS$ is strongly aperiodic, and $y-x$ is reachable from
infinity in this model, the right-hand side is asymptotic to some
$\kappa\, 
{(\mu^{p})^m}
m^\alpha$, which gives the desired lower bound on $q(i,j;n)$.
  \end{proof}

We can now conclude the proof of Theorem~\ref{thm:theta}. \begin{proof}[Proof
of Theorem~\ref{thm:theta}] We have already established the part that deals
with the orbit size, so we focus on the nature of the series $Q(x,y;t)$. We
assign weight $\om_s=1$ to every step of $\cS$. Let {$(i,j) \in \overline
\Lambda$, with $i$ and~$j$ large enough} for $(i,j)$ to be reachable from
infinity. Then the bounds~\eqref{bounds} on $q(i,j;n)$ hold (whether the model
is periodic or not) with $\alpha$ irrational. The \gf\ $\sum_{n} q(i,j;n) t^n$
is the coefficient of $x^i y^j$ in $Q(x,y;t)$, and it is D-finite if
$Q(x,y;t)$ is D-finite. In this case it must be a
G-function~\cite[Sec.~2]{BoKa09}. But the properties of these functions are
incompatible with the existence of such bounds~\cite[Thm.~2]{BoKa09}, and thus
$Q(x,y;t)$ cannot be D-finite. Indeed, it follows from the
Katz-Chudnovsky-Andr\'e theorem~\cite{Andre00,FiRi14} on the local structure
of G-functions, combined with classical transfer theorems, that $q(i,j;n)$
needs to be asymptotically equivalent to a sum of terms of the form $\kappa
\rho^n n^a (\log n)^b$ with only \emph{rational} exponents~$a$, and our
exponent $\alpha$ must be one of these $a$'s. 
\end{proof}

%==================================================================
\subsubsection{Examples}
\label{sec:ex}
%==================================================================
We now illustrate the above results with {five} examples.

\smallskip
\noindent{\bf Example D (continued): a model with rational exponent and
  finite orbit.} Let us take $S(x,y)= \bx^2 +\bx
y+y ^2+x\by$. The unique positive critical pair is $(a,b)= (3^{1/4},
3^{-1/4})$. We have seen that the orbit of $\cS$ is finite {(Figure~\ref{fig:orbit-quadrangulations})}, and
indeed,
\[c:= \frac{
  S_{12}(a,b)}{\sqrt{S_{11}(a,b)S_{22}(a,b)}}= -\frac 1 2 = \cos \frac
{2\pi} 3.
\]
{With the notation of Theorem~\ref{thm:theta}}, we have $\theta=2\pi/3$, and by
Theorem~\ref{thm:exponent} the
excursion exponent is $\alpha=-4$.
The involutions $\Phi$ and $\Psi$ defined in
Section~\ref{sec:group} satisfy
\[
\Phi(a+u,b+v)= (a-u+\sqrt 3 v + \cdots, b+v)\qquad 
\Psi(a+u,b+v)= (a+u,b+u/\sqrt 3 -v + \cdots), \qquad 
\]
so that
\[
\Theta(a+u,b+v)= (a-u+\sqrt 3 v+\cdots ,b-u/\sqrt 3 +\cdots).
\]
The matrix $J$ is 
\[
J=\left(
  \begin{array}{cc}
    -1 & \sqrt 3\\
-1/\sqrt 3 & 0
  \end{array}\right),
\]
its eigenvalues are $e^{\pm 2 i \pi/3}$, and $J^3$ is the identity
matrix. In fact, it can be checked that $\Theta^3=\id$. This is
reflected in Figure~\ref{fig:orbit-quadrangulations} by the existence of bicoloured hexagons. \qee

\medskip
\noindent{\bf Example: a model with irrational exponent and
  infinite orbit.} Now take $S(x,y)=\bx^2+y+x\by$. The unique
positive critical pair is $(a,b)=(2^{2/5}, 2^{1/5})$. 
We have 
\[c:= \frac{
  S_{12}(a,b)}{\sqrt{S_{11}(a,b)S_{22}(a,b)}}= -\frac 1 {\sqrt 6}.
\]
Let us prove that this is not the cosine of a rational multiple
$\theta$ of
$\pi$. With $z=e^{i\theta}$, this would mean that $z+1/z=-\sqrt{2/3}$,
so that the minimal polynomial of $z$ (and $1/z$) would be
$z^4+4z^2/3+1$. This is not a cyclotomic polynomial, hence $c$ is not
of the requested form. We  conclude from Theorem~\ref{thm:theta} that the orbit is infinite, and
the series $Q(x,y;t)$ not D-finite. The excursion exponent is
$\alpha=-1-\pi/\arccos(1/\sqrt6)\sim
{-3.73 }\ldots$, and it is an
irrational number. 

The involutions $\Phi$ and $\Psi$ satisfy
\[
\Phi(a+u,b+v)= (a-u+ 2^{6/5} v/3 + \cdots,b+v) \qquad 
\Psi(a+u,b+v)= (a+u,b+u/2^{1/5} -v + \cdots), \qquad 
\]
so that
\[
\Theta(a+u,b+v)= (a-u +2^{6/5} v/3+\cdots ,b-u/2^{1/5}-v/3 +\cdots).
\]
The matrix $J$ is 
\[
J=\left(
  \begin{array}{cc}
    -1 &2^{6/5} /3\\
-1/2^{1/5} & -1/3
  \end{array}\right).
\]
Its eigenvalues are the roots of $\lambda^2+4\lambda/3+1$, and thus
are not roots of unity.  In particular, the group generated by $\Phi$
and $\Psi$ is infinite. \qee

\medskip

{The same argument proves that the walks of
  Figure~\ref{fig:slope2} have an irrational excursion exponent
  $\alpha=-1-\pi/\arccos(1/\sqrt 5)$, and thus a
  non-D-finite \gf.}

\medskip We will now consider {three} models that have a rational excursion
exponent, but still an infinite orbit. We will prove this using the approach
of Section~\ref{sec:group}, either by taking for $(a,b)$ the positive critical
point and pushing further the expansion of $\Theta$, or by considering another
critical point.

\medskip
\noindent{\bf Example: a model with rational exponent but
  infinite orbit.}  Take $S(x,y)=x+y+\bx +\by+x\by^2+\bx^2y$. This
is model \#13 in Table~\ref{tab:embarassing} {(Section~\ref{sec:embarassing})}. The unique
positive critical pair is $(a,b)=(\sqrt 2, \sqrt 2)$. 
We have 
\[c:= \frac{
  S_{12}(a,b)}{\sqrt{S_{11}(a,b)S_{22}(a,b)}}= -\frac 1 2 = \cos \frac
{2\pi} 3.
\]
{With the notation of Theorem~\ref{thm:theta}},
we have $\theta=2\pi/3$. The excursion exponent is $\alpha=-4$.

If we start from the positive critical point $(a,b)=(\sqrt 2, \sqrt 2)$ to define the involutions
$\Phi$ and $\Psi$, we find
\[
\Phi(a+u,b+v)=( a-u+ v+\cdots,b+v) \qquad 
\Psi(a+u,b+v)= (a+u,b +u-v + \cdots), \qquad 
\]
so that
\[
\Theta(a+u,b+v)= (a-u+v +\cdots ,b-u+\cdots).
\]
The matrix $J$ is 
\[
J=\left(
  \begin{array}{cc}
    -1 &1\\
 -1&0
  \end{array}\right).
\]
Its eigenvalues are $e^{\pm2i\pi/3}$, and $J^3=\id$, so
we cannot use the criterion of Theorem~\ref{thm:theta} to prove that
the orbit is infinite. But let us push further the expansion of
$\Theta$. We have
\[
\Phi(a+u,b+v)=\left(a-u+v+\frac5 8\,a{u}^{2}-\frac 5 8\,a uv-\frac 1 8\,a{v}
^{2}-{\frac {25\,}{32}}{u}^{3}+{\frac {7\,}{8}}{u}^{2}v+\frac 1{16}\,u{v}^{2}
+\frac 1{32}\,{v}^{3}+ \cdots, b+v\right),
\]
where the missing terms are of order 4 or more. A symmetric formula
holds for $\Psi$. Hence
\begin{multline*}
  \Theta(a+u,b+v) = \left( a-u+v+\frac5 8\,a{u}^{2}-\frac 5 8\,a uv-\frac 1 8\,a{v}
^{2}-{\frac {25\,}{32}}{u}^{3}+{\frac {7\,}{8}}{u}^{2}v+\frac 1{16}\,u{v}^{2}
+\frac 1{32}\,{v}^{3}+ \cdots, \right. \\
\left. b -u+\frac 1 2\,a{u}^{2}+\frac 1 4\,auv-\frac 1 4\,a{v}^
{2}-\frac 1 2\,{u}^{3}-\frac 3 8\,{u}^{2}v+{\frac {7\,}{16}}{v}^{3}+\cdots
\right).
\end{multline*}
We have already seen that $\Theta^3$ comes close to being the identity
-- at least, it is the identity at first order. But in fact,
\[
 \Theta^3(a+u,b+v) = \left(a+u + \frac 1 8 (u-2v)(u^2-uv+v^2)+\cdots,
b+v 
{-}  \frac 1 8 (v-2u)(u^2-uv+v^2)+\cdots
\right),
\]
so that 
\[
 \Theta^{3k}(a+u,b+v) = \left(a+u + \frac k 8 (u-2v)(u^2-uv+v^2)+\cdots,
b+v {-}  \frac k 8 (v-2u)(u^2-uv+v^2)+\cdots
\right).
\]
Thus $\Theta$ has infinite order, and by Proposition~\ref{prop:ab} the orbit is infinite. The nature of
$Q(x,y;t)$ remains unknown. 

An alternative way to prove infiniteness of the orbit for this model
is to start from another critical point and use a first order argument
rather than the above longer expansion. Let us take
$(a,b)=(e^{5i\pi/6}, e^{i\pi/6})$. Then the involutions $\Phi$ and
$\Psi$ satisfy
\begin{multline*}
{\Phi(a+u,b+v)= \left(a-u - \frac {2-6i \sqrt 3}7v+\cdots,
    b+v\right),}
\\
{\Psi(a+u,b+v)=\left(a+u, b-v- \frac  {2+6i\sqrt 3}7 u+\cdots\right),
}
\end{multline*}
so that 
\[
{\Theta(a+u,b+v)=\left( a-u -\frac {2-6i\sqrt 3}7 v+\cdots, b
  +\frac  {2+6i\sqrt 3}7 u+\frac 9 7 v+\cdots\right).
}\]
The characteristic polynomial of the corresponding matrix $J$ is
{$\lambda^2-2\lambda/7+1$}, and its roots are not roots of
unity. By Proposition~\ref{prop:ab}, the orbit is infinite.
\qee

\medskip

In the next example, the excursion exponent is again rational, and only the
first method above (expanding $\Theta$ to higher order) works to prove
infiniteness of the orbit.

\medskip
\noindent{\bf Example: one more model with rational exponent but
  infinite orbit.} Take $S(x,y)=xy+x\by^2+\bx^2y$. This is model \#2
in Table~\ref{tab:embarassing}. The positive critical point is
$(a,b)=(1,1)$, and
$c= -1/2= \cos(2\pi/3)$. The excursion exponent is again $-4$. Note
that there is no quadrant excursion from $(0,0)$ to $(0,0)$, because
this point is not reachable from infinity. But the asymptotic
bounds~\eqref{bounds} apply for instance for $(i,j)=(1,1)$ (with
period $p=3$).

We can prove that the orbit is infinite by
expanding $\Theta:=\Psi\circ \Phi$ up to order 3:
\begin{multline*}
  \Theta(1+u,1+v)= \left( 1-u+v+\frac 4 3\,{u}^{2}-\frac 4 3\,uv-\frac 2
  3\,{v}^{2}-{\frac {16\,}{9}}{u}^{3}+2\,{u
}^{2}v+\frac 2 3\,u{v}^{2}-\frac 1 9\,{v}^{3}+\cdots,\right. \\
\left. 1-u+\frac 2 3\,{u}^{2}+\frac 4 3\,uv-\frac 4 3\,{v}^{2}
+\frac 1 9\,{u}^{3}-3\,{u}^{2}v+\frac 1 3\,u{v}^{2}+{\frac {22\,}{9}}{v}^{3}+\cdots
\right)
\end{multline*}
which gives 
\[
\Theta^3(1+u,1+v)= \left( 1+u + \frac 2 3 (u-2v)(u^2-uv+v^2)+\cdots,
1+v {-}  \frac 2 3 (v-2u)(u^2-uv+v^2)+\cdots
\right).
\]
We conclude as above that all series $\Theta^{3k}(1+u,1+v)$ are
distinct, and that the orbit is infinite.

Starting from another critical pair $(a,b)$ does not make the argument
shorter: for all possible choices, the transformation $\Theta^3$ is the
identity at linear order. \qee

\medskip

We conclude with a third model with a rational exponent but an infinite orbit.
This one is symmetric in both coordinate axes. Recall that highly symmetric
models \emm with small steps, behave nicely in any dimension: they have a
finite orbit, a D-finite \gf, and explicit asymptotic enumeration is
known~\cite{MelczerMishna2016}. But the large step, highly symmetric model of
the next example has an infinite orbit. This cannot be proved starting from
the positive critical point, because the corresponding involutions $\Phi$ and
$\Psi$ \emm do, generate a finite group. But taking another critical point
works.

\medskip
\noindent{\bf Example: a highly symmetric model with an infinite
  orbit.}
Take 
\[
S(x,y)=(x+\bx)(y+\by)+(x^2+\bx^2)(y^2+\by^2).
\]
The positive critical point is $(a,b)=(1,1)$, and 
 $c=0$. The excursion exponent is $-3$. The
transformations $\Phi$ and $\Psi$ {defined from $a=b=1$} are respectively $(x,y)\mapsto
(\bx,y)$ and $(x,y)\mapsto (x, \by)$ and they \emm do, generate a
finite group, {of order $4$}. But let us consider instead the critical point $(a,b)=(i,i)$. Then 
\[
\Theta(a+u,b+b)=(a-u+v/2+ \cdots, b+u/2-v+\cdots),
\]
and the Jacobian matrix $J$ has characteristic polynomial
$\lambda^2+7\lambda/4+1$, which is not cyclotomic. Hence the orbit is infinite.

 We do not know about the nature of the associated \gf, but {the first 70\,000
terms of the series $Q(0,0)$ (modulo a prime)} did not allow us to guess any
recurrence relation for its coefficients. \qee

%%%%%%%%%%%%%%%%%%%%%%%%%%%%%%%%%%%%%%%%%%%%%%%%%%%%%%%%%%%%%%
\section{Section-free functional equations}
\label{sec:free}
%%%%%%%%%%%%%%%%%%%%%%%%%%%%%%%%%%%%%%%%%%%%%%%%%%%%%%%%%%%%%%

In this section we consider step sets $\cS \subset \zs^d$ such that the orbit
of $\Bx=(x_1, \ldots, x_d)$ is finite. For every element $\Bx'$ of this orbit
we can replace $\Bx$ by $\Bx'$ in the main functional equation defining
$Q(\Bx)$, as {we did} in~\eqref{6-eqs}. The resulting equation will be called
an \emm orbit equation,\footnote{In other papers, like~\cite{BoBoKaMe16}, the
\emm orbit equation, is what we call here the \emm section-free, equation. We
hope that this change in the terminology will not cause any trouble.}. As the
left-hand side of the original functional equation is $K(\Bx)Q(\Bx)$, where
$K(\Bx)=1-tS(\Bx)$ is the kernel, the orbit equation associated with $\Bx'$
has left-hand side $K(\Bx)Q(\Bx')$, because the kernel takes the same value
for all elements in the orbit. On the right-hand side of the orbit equations
are several specializations of the \gf \ $Q$, which we call \emm sections,.
Due to the construction of the orbit, every section occurs at least in two
orbit equations.

The next step in our approach is to form a linear combination of the orbit
equations that is free from sections, if one exists, as was the case
for~\eqref{alt-tandem}. Once the main functional equation is written, and the
(finite) orbit determined, section-free equations can be found by solving a
linear system with coefficients in the algebraic closure of $\qs(x_1, \ldots,
x_d)$. {In all cases that we have examined, we find that a section-free
equation exists (and sometimes several). However,} we have not been able to
find a generic form for section-free equations. Let us examine two simple
examples; the first one shows that there can be multiple section-free
combinations.

\medskip
\noindent{\bf Example A (continued).} 
We return to the one-dimensional step set $\cS=\{\bar 1,2\}$. The step polynomial is $S(x)= \bx+x^2$, and the elements $x'$ of the orbit of $x$ are the solutions of $S(x)=S(x')$. Hence the orbit  is $\{x, x_1, x_2\}$, with
\[
x_{1,2}=\frac {-{x}^{2}\pm\sqrt {x \left( {x}^{3}+4 \right) }}{2x}.
\]
Substituting the three orbit elements into {the functional equation}~\eqref{eqf-1D-bar} gives three
orbit equations, each involving  only one section  (namely, $Q(0)$). There are several section-free linear combinations of the orbit equations. One of them is
\begin{equation}\label{exA:eq1}
K(x)\left( x Q(x)-\xp Q(\xp )\right)= x-\xp ,
\end{equation}
another one is
\begin{equation}\label{exA:eq2}
K(x)\left( x Q(x)-\xm Q(\xm )\right)= x-\xm ,
\end{equation}
and in fact any section-free equation is a linear combination of these two. \qee

\medskip
\noindent{\bf Example B (continued).} 
We now reverse the steps of the previous example and consider $\cS=\{\bar 2,1\}$. The orbit of $x$ consists of $x$, $\xp$ and $\xm$ with
\[
x_{1,2}= {\frac {1\pm\sqrt {4\,{x}^{3}+1}}{2{x}^{2}}}.
\]
Substituting the orbit elements into~\eqref{eqf-1D} gives three orbit equations containing two sections, $Q_0$ and $Q_1$. There is, up to a multiplicative factor, a \emm unique, section-free linear combination of these three equations:
\begin{multline*}
  K(x) \left(
\frac{x^2}{(x-\xp )(x-\xm )}Q(x)
+\frac{\xp ^2}{(\xp -x)(\xp -\xm )}Q(\xp )
+\frac{\xm ^2}{(\xm -x)(\xm -\xp )}Q(\xm )\right)
\\= \frac{x^2}{(x-\xp )(x-\xm )}
+\frac{\xp ^2}{(\xp -x)(\xp -\xm )}
+\frac{\xm ^2}{(\xm -x)(\xm -\xp )}= 1.
\end{multline*}
\qee

\medskip
The above two examples  are instances of a more  general result that applies
to any 1-dimensional model.

\begin{Proposition}\label{prop:section-free1D}
  Assume $d=1$. Let $-m$ (resp. $M$) be the smallest (resp. largest)
  element in~$\cS$, and assume that $m\ge 0$ and $M> 0$. Then the
  orbit of $x$ has cardinality  $m+M$. The vector space of section-free equations consists of all linear combinations of 
\[
    K(x) \sum_{i=0}^m \frac{u_i^m}{\prod_{j\not = i} (u_i-u_j)}Q(u_i) =1,
  \]
  where $u_0, u_1, \ldots, u_m$ are any  distinct elements in the orbit of $x$.
\end{Proposition}
This proposition will proved in Section~\ref{sec:1D}. The number of ways of
choosing the $u_i$'s is $\binom{m+M}{m+1}$. These section-free equations are
not always linearly independent (in Example A, the third equation of this
type, which involves $x_1$ and $x_2$, is the difference of~\eqref{exA:eq1}
and~\eqref{exA:eq2}). However, if the largest step is~$1$ (that is, $M=1$),
then Proposition~\ref{prop:section-free1D} tells that there is a unique
section-free equation (up to a multiplicative factor). This was observed in
Example B, and seems to generalize to dimension 2.

\begin{Conjecture}\label{conj:section-free}
When $d=2$ and the orbit is finite, there always exist non-trivial section-free linear combinations of the orbit equations. Moreover, if there is no large forward step,
then there is a unique section-free combination, up to a multiplicative factor.
\end{Conjecture}

\noindent{\bf Example.} In some $x/y$-symmetric quadrant models, like  Kreweras' model
$\cS=\{\nearrow, \leftarrow, \downarrow\}$, the orbit of $(x,y)$
contains $(y,x)$, and we want to clarify what we mean  with the
uniqueness of the section-free equation. The  functional equation reads
\[
  K(x,y)Q(x,y)=\\ 1- t\bx Q(0,y) -t\by Q(x,0).
\]
The orbit of $(x,y)$ consists of 6 pairs:
\[
(x,y), \quad  (\bx \by, y), \quad  (\bx\by, x),  \quad (y,x), \quad  (y,
\bx\by), \quad  (x, \bx\by).
\]
A linear combination of the 6 orbit equations, with indeterminate
weights $\alpha_1, \ldots, \alpha_{6}$, involves 6 sections: three
specializations of $Q(x,0)$, and three of $Q(0,y)$. If we
require the contribution of each to vanish, we find (up to a
multiplicative factor) a unique solution for the $\alpha_i$'s, and
thus a unique section-free equation:
\begin{equation}\label{sec-free-K}
K(x,y) \big(xyQ(x,y) -\bx Q(\bx \by, y) +\by Q(\bx\by, x) -xyQ(y,x)+ \bx Q(y,
\bx\by)-\by Q(x, \bx\by)\big)=0.
\end{equation}
Note that the right-hand side (the so-called \emm orbit sum,)
vanishes. The $x/y$-symmetry makes this equation trivial.

However, it makes sense to exploit the symmetry of the model in the functional equation, and to write:
\[
K(x,y)Q(x,y)= 1- t\bx Q(y,0) -t\by Q(x,0).
\]
Now a linear combination of the 6 orbit equations involves only 3
sections. If we want the contribution of each to vanish, we find a
vector space of dimension 3 of solutions, generated by all equations
of the form 
\[
K(x,y) \left( Q(x',y')-Q(y',x')\right)=0,
\]
for $(x',y')$ in the orbit. Again,  these equations are trivial. \qee

\medskip
We now prove that {Conjecture~\ref{conj:section-free}}
 holds  in the case of small steps --- and in fact, in arbitrary dimension.
\begin{Proposition}\label{prop:sec-free-small}
  If $\cS\subset \{-1,0,1\}^d$ has positive and negative steps in
  every direction, and the associated orbit is finite, then there is a
  unique section-free linear combination of orbit equations, up to a
  multiplicative factor. {It reads
\begin{equation}\label{sec-free}
\sum_{\Bu} (-1)^{\ell(\Bu)}  K(\Bu) Q(\Bu) \prod _{i=1}^d u_i =
\sum_{\Bu} (-1)^{\ell(\Bu)} \prod _{i=1}^d u_i,
\end{equation}
where the sum runs over all elements $\Bu=(u_1, \ldots, u_d)$ of the orbit and $\ell(\Bu)$ is
the length of $\Bu$.}
\end{Proposition}
\begin{proof}
We consider the result of multiplying the functional
equation~\eqref{eqfunc:small} by the product of all variables
$\prod_i x_i$: 
\begin{equation}
\label{eqfunc:small-mult}
K(\Bx)Q(\Bx) \prod_i x_i  =
\prod_i x_i+t \sum_{\emptyset \not = I \subset \llbracket 1, d\rrbracket} 
\left( (-1)^{|I|} Q_I(\Bx) \sum_{\Bs \in \cS: s_i =-1 \forall i \in I} \Bx^{\Bs}\prod_i x_i\right).
\end{equation}
Note that, since the last sum is over all $\Bs$ such that $s_i=-1$ for $i\in
I$, the monomial $\Bx^{\Bs}\prod_i x_i$ does not involve any of the $x_i$'s
for $i\in I$. The same holds for $Q_I(\Bx)$. We now call any version
of~\eqref{eqfunc:small-mult} instantiated at an orbit element an \emm orbit
equation,.

Take $I=\{i\}$, with $1\le i\le d$. For $\Bu$ in the orbit of $\Bx$, the
section $Q_I(\Bu)$ occurs in exactly two orbit equations: the equation
obtained from $\Bu$, and the one obtained from $\Bv:=\Phi_i(\Bu)$,
with~$\Phi_i$ defined as in Proposition~\ref{prop:small-group}. Moreover, the
coefficient of $Q_I(\Bu)$ is the same in both equations (it does not depend on
the $i$th coordinate of $\Bu$). Hence in a section-free linear combination of
orbit equations, the weights of the equations associated with $\Bu$ and
$\Phi_i(\Bu)$ must be opposite. {By transitivity, there cannot be more than
one section-free equation. Moreover, in the small step case, the lengths of
two adjacent elements differ by $\pm 1$ (Proposition~\ref{prop:small-group}),
and thus the only possible section-free equation is~\eqref{sec-free}.}

{So let us form the linear combination of orbit equations having the same
left-hand side as~\eqref{sec-free}.} For $\Bu$ in the orbit and $I\subset
\llbracket 1, d\rrbracket$, the section $Q_I(\Bu)$ occurs (with the same
weight) in all orbit equations obtained from elements $\Bv$ that only differ
from $\Bu$ at positions of $I$. We can define on these elements an involution
that changes the parity of the length (for instance $\Phi_{\min (I)}$). This
implies that the coefficient of $Q_I(\Bu)$ in the signed sum vanishes, and
that we have indeed constructed a section-free equation. 
\end{proof}

Of course, all the examples of this paper support
Conjecture~\ref{conj:section-free}. The next example shows that the number of
sections occurring in the orbit equations can be larger than the number of
orbit equations, which makes the existence of section-free equations more
surprising.

\medskip \noindent{\bf Example F: a model with small forward steps.}
Take $\cS=\{10, \bar 1 0, \bar 2 1, 0 \bar 1\}$.  Then the orbit of $(x,y)$ consists of the following pairs:
\begin{equation}\label{orbit:hard1}
\begin{array}{ccc}
(x, y) &(x_1 , y)  & (x_2 , y) \\
(x, x^2\by)  &(-\bxun , x^2\by)  &(-\bxde  , x^2\by)  \\
(x_1 , x_1 ^2\by)  &(-\bx, x_1 ^2\by) &(-\bxde  , x_1 ^2\by) \\
(x_2 , x_2 ^2\by)  &(-\bx, x_2 ^2\by) &(-\bxun , x_2 ^2\by) 
\end{array}
\end{equation}
where
\[
x_{1,2}= \frac{x+y\pm\sqrt{(x+y)^2+4x^3y}}{2x^2}
\]
and $\bx_i=1/x_i$. The structure of this orbit is the first
shown in Figure~\ref{fig:orbits} {(Section~\ref{sec:m21})}. The functional equation reads
\begin{equation}\label{eq:hard1}
K(x,y)Q(x,y)=1-t\bx (1+\bx y)Q_{0,-}(y)-t\bx y  Q_{1,-}(y)-t\by Q(x,0).
\end{equation}
The 12 orbit equations involve in total $6+4+4=14$ distinct sections: 6
specializations of $Q(x,0)$, 4 specializations of $Q_{0,-}(y)$ and  4
specializations of $Q_{1,-}(y)$. Hence in order to find a section-free
equation, we need to solve a linear system with 14 equations but only
12 unknowns. Still,  we find a solution (and only one, up to a
multiplicative factor).  The weight of the orbit equation associated with the pair $(x',y')$ is
\[
\pm x'^2(x'_1 -x'_2 )\sqrt {yy'},
\]
where $(x'_i,y')\approx (x',y')$ for $i=1,2$, and $x'_1\not = x'_2$. More precisely, the weights associated with the above 12 orbit elements are
\begin{equation}\label{weights:hard1}
\begin{array}{rrr}
x^2(x_1-x_2)y &x_1 ^2(x_2-x) y  & -x_2^2 (x_1-x) y \\
x^2( \bxde -\bxun) x  &-\bxun^2(x+\bxde )x   &\bxde ^2(x+\bxun)x   \\
x_1^2(\bx-\bxde )x_1  &\bx^2(x_1+\bxde ) x_1 &-\bxde ^2(x_1+\bx) x_1  \\
-x_2 ^2(\bx-\bxun) x_2   &-\bx^2(x_2+\bxun)x_2  &\bxun^2(x_2+\bx) x_2 .
\end{array}
\end{equation}\qee

\medskip \noindent{\bf Example D (continued): a model with  large
  forward and backward step.} 
Let us take $\cS=\{ \bar 20,\bar 11, 02,1\bar 1\}$. 
Recall that the orbit of $(x,y)$ is shown in
Figure~\ref{fig:orbit-quadrangulations}, with
\[x_{1,2}= {\frac {x{y}^{2}+y\pm\sqrt {y \left(
        {x}^{2}{y}^{3}+4\,{x}^{3}+2\,x{y}^{2}+y \right) }}{2{x}^{2}}}.
\]
The functional equation for this model is given by~\eqref{eqfunc:quadrangulations}, and the 12 orbit equations involve $4+4+4=12$ sections. The vector space of section-free linear combinations has dimension 2; it is generated by two linear combinations of 9 orbit equations:
\begin{multline*}
x^2y\,Q ( x,y ) 
-{\frac {{x_1}^{2}y \left( x-x_2 \right) Q ( x_1,y ) }{
     x_1-x_2 }}
+{\frac {{x_2}^{2}y \left( x-x_1 \right) Q ( x_2,y ) }{
  x_1-x_2  }}
-{{x^2\bxun \, Q  ( x,\bxun  ) }}
\\
+{\frac {{x_2}^{2} \left( xy-1 \right) Q
 ( x_2,\bxun  ) }{x_1 \left( x_2
\,y-1 \right) }}
-{\frac { \left( x-x_2 \right) Q ( \by,\bxun  ) 
}{yx_1 \left( x_2\,y-1 \right) }}
+{\frac {{x_1}^{2} \left( x-x_2 \right) Q ( x_1,\bx 
 ) }{x \left( x_1-x_2 \right) }}
\\
-{\frac {{x_2}^{2}  \left( x_1\,y-1
 \right) \left( x-x_2 \right) Q ( x_2,\bx )
}{{x} \left( x_1-x_2 \right)  \left( x_2\,y-1 \right) }}
+{\frac { \left( x-x_2 \right) Q ( \by,
\bx  ) }{xy \left( x_2\,y-1 \right) }}
={\frac {
\left(1- xy\right)  \left( 1-x_1y \right)   \left( x-x_1   
 \right)  \left( x-x_2 \right) }{x{y}x_1K(x,y)}},
\end{multline*}
and the same equation with $x_1$ and $x_2$ exchanged. We refer to~\cite{BoFuRa17} for the solution of a family of models with arbitrarily large steps which generalizes this one.

%%%%%%%%%%%%%%%%%%%%%%%%%%%%%%%%%%%%%%%%%%%%%%%%%%%%%%%%%%%%%%
\section{Extracting the main \gf}
\label{sec:extract}
%%%%%%%%%%%%%%%%%%%%%%%%%%%%%%%%%%%%%%%%%%%%%%%%%%%%%%%%%%%%%%

We now assume that, for a step set $\cS$ with a finite orbit, we have obtained
one (or several) section-free functional equations. Can we extract from these
equations the main \gf\ $Q(x_1, \ldots, x_d)$, as we did in
Section~\ref{sec:basic}? Not systematically, as we already learnt from some
small step models.

\medskip
\noindent{\bf Example C: Gessel's walks (continued).} 
The orbit of $(x,y)$ consists of 8 elements. The steps are small,
hence the unique section-free equation is the {alternating sum~\eqref{sec-free}.}
Remarkably, its right-hand side vanishes:
\begin{multline*}
  xyQ(x, y)-\bx Q (\bar x \bar y, y) +xQ (\bar x \bar y, x^2y)-xyQ (\bar x, x^2y)\\
+\bx\by Q (\bar x, \bar y)-xQ (xy, \bar y)+\bx Q(xy, {\bar x}^2 \bar
y)-\bx\by Q (x, \bar y {\bar x}^2)=0.
\end{multline*}
This homogeneous equation does not characterize $Q(x,y)$. For
instance, 1,  $x$,  $xy$, and  $y-x^2$ are solutions. The space of
solutions is actually infinite dimensional, as it clearly contains all
monomials $x^iy^i$. \qee

\medskip
Among the 23 quadrant models with small steps that have a finite
orbit, exactly 4 have a section-free equation that does not
characterize $Q(x,y)$:
Gessel's model, as just shown, and the three Kreweras like
models: $\cS=\{\nearrow, \leftarrow, \downarrow\}$, its reverse
$\overline\cS=\{\swarrow, \rightarrow, \uparrow\}$ and the union $\cS
\cup \overline\cS$~\cite{BoMi10}. For those three, the orbit of $(x,y)$ contains
$(y,x)$, and the section-free equation is~\eqref{sec-free-K}.
Clearly, any symmetric series in $x$ and $y$ satisfies this
equation. 

For these four models, the \emm orbit sum,, that is, the right-hand side of
the section-free equation, vanishes. However, there exist as well (weighted)
models with a non-vanishing orbit sum, for which the section-free equation
does not characterize $Q(x,y)$. Let us recall an example taken
from~\cite[Sec.~8.2]{BoBoKaMe16}.

\medskip
\noindent{\bf Example.} 
Take $\cS=\{\bar 1 \bar 1 , \bar 1 1 , \bar 1 0, \bar 1 0, 1 0,1 1\}$ (note the repeated West step). The step polynomial is
\[
S(x,y)=  (1+y)\left(\bx(1+\by)+x\right).
\]
The orbit of $(x,y)$ contains 6 elements, and the unique section-free equation reads:
 \begin{align*}
   xyQ(x,y) &- \bx(1+y)Q(\bx(1+\by),y )
 + \frac{x(1+y)}{(1+y)^2+x^2y^2}\, Q\left(\bx(1+\by),
   \frac{x^2y}{(1+y)^2+x^2y^2}\right) 
 \\
 &- \frac{xy(1+y+x^2y)}{(1+y)^2+x^2y^2}\, Q\left(\bx(1+y)+xy,
   \frac{x^2y}{(1+y)^2+x^2y^2}\right)\\
 &+\frac{\bx\by(1+y+x^2y)}{1+x^2} Q\left(\bx(1+y)+xy,
   \frac{\by}{1+x^2}\right)
 -\frac{x\by}{1+x^2}\, Q\left(x,
   \frac{\by}{1+x^2}\right)
 \\
 &=\frac{\left(1+y(1-x^2)\right) \left(1-y^2(1+x^2)\right) \left( 1-x^2 +y(1+x^2)\right)}{xy(1+x^2)K(x,y)\left((1+y)^2+x^2y^2\right)}.
 \end{align*}
The right-hand side is non-zero, but this equation does not define $Q(x,y)$ uniquely in the ring $\qs[x,y][[t]]$. In fact, the associated  homogeneous equation (in $Q(x,y)$) seems to have an infinite dimensional space of solutions. It includes at least the following polynomials in $x$ and $y$: 
\[
x, \quad  2xy+{x}^{3}y , \quad {x}^{2}y+{x}^{2}+y+2 , \quad {x}^{3}{y}^{2}-{x}^{3}y+{x}^{3}+2x{y}^{2}. 
\]
\qee\medskip

We now consider examples where the series $Q(\Bx)$ is {indeed} characterized
by a section-free equation, but for which the extraction is not as simple as
in Section~\ref{sec:basic}. Our first example is one-dimensional.

\medskip
\medskip\noindent{\bf Example A (continued).} 
It can be seen that~\eqref{exA:eq1} (or~\eqref{exA:eq2}) characterizes $Q(x)$,
but how can we extract it effectively? Here is {one solution}.

Take the first of these two linear combinations, written as
\[
Q(x)- \bx x_1  Q(x_1)= \frac{1- \bx x_1}{K(x)}
\]
with $K(x)=1-t(\bx+x^2)$, and choose for the algebraic closure of
$\qs(x)$ the set of Puiseux series in $\bx$ (not in $x$!). Then
\[
x_1={\frac {\sqrt {4\,{\bx}^{3}+1}-1}{2\bx}}=\bx^2-\bx^5+O(\bx^8)
\]
is a \fps\ in $\bx$. Now both sides of the above section-free equation are
series in $t$ whose coefficients are Laurent series in $\bx$.
Extracting the non-negative part in $x$  gives:
\[
Q(x)= [x^{\ge}]  \frac{1- \bx x_1}{K(x)},
\]
where the right-hand side is first expanded in $t$, then in $\bx$. This will
be generalized to arbitrary one-dimensional models in Section~\ref{sec:1D}
(Proposition~\ref{prop:dim1-alg}). \qee

\medskip
In our next example, one simply has to extract the positive part of a rational
series to obtain $Q(x,y)$, but justifying why is a bit delicate.

\medskip\noindent{\bf Example F (continued).} 
Let $\cS=\{10, \bar 1 0, \bar 2 1, 0 \bar 1\}$. The functional equation is given by~\eqref{eq:hard1}, the orbit by~\eqref{orbit:hard1} and the weights in the section-free linear combination by~\eqref{weights:hard1}. Let us divide  this linear combination by $x^2y(x_1-x_2)K(x,y)$,  so as to isolate $Q(x,y)$. The resulting equation reads
\begin{equation}\label{eqA}
Q(x,y)+ x\bxun\bxde\by\,  Q(x,x^2\by) 
+ A_1 +A_2 +A_3+A_4+A_5=R(x,y)
\end{equation}
with
\begin{align}
A_1&= \bx^2 \, \frac{x_1^2(x_2-x) Q(x_1,y)-x_2^2(x_1-x) Q(x_2,y)}{x_1-x_2} \notag \\
A_2&=-\bx^2\by \, \frac{ \bxun^2 (x+\bxde ) x Q(-\bxun,x^2\by)-\bxde ^2 (x+\bxun) x Q(-\bxde ,x^2\by)}{x_1-x_2} \notag\\
A_3&=\bx^2 \by \, \frac{x_1^3(\bx-\bxde )Q(x_1,x_1^2y)-x_2^3(\bx-\bxun)Q(x_2,x_2^2y)}{x_1-x_2} \notag\\
A_4&=\bx^2 \by\, \frac{\bx^2  (x_1+\bxde)x_1 Q(-\bx, x_1^2\by) -\bx^2 (x_2+\bxun) x_2 Q(-\bx, x_2^2\by)}{x_1-x_2} \notag\\
A_5&=-\bx^2\by\, \frac{\bxde^2 (x_1+\bx )x_1Q(-\bxde , x_1^2\by)-\bxun^2 (x_2+\bx)x_2Q(-\bxun, x_2^2\by)}{x_1-x_2} \label{A5}
\end{align}
and
\[
R(x,y)= {\frac { \left( {x}^{2}+1 \right) \left( x+y \right)  \left( y-x \right)
\left( {x}^{2}y-2\,x-y \right) \left( {x}^{3}-x-2\,y \right) }{{x}^{7}{y}^{3} \left(1-t(x+\bx+\bx^2 y +\by) \right)}}.
\]
Each term in~\eqref{eqA} is written as a power series in $t$ whose
coefficients are Laurent polynomials in $x$, $y$, $x_1$ and $x_2$, symmetric
in $x_1$ and $x_2$ (because the numerators of the series $A_i$ are \emm
anti-symmetric, in $x_1$ and $x_2$). Observe that the symmetric functions of
$x_1$ and $x_2$ are Laurent polynomials in $x$ and $y$, and more precisely,
polynomials in $\bx\qs[\bx,y,\by]$ (we say that they are \emm $x$-negative,):
\begin{equation}
\label{neg-x}
x_1 +x_2 = \bx(1+\bx y) \quad \hbox{and} \quad x_1 x_2 =-\bx y.
\end{equation}
 The symmetric functions of their reciprocals are Laurent polynomials
 in $x$ and $y$, and more precisely, polynomials in $\qs[x, \bx,\by]$
 (we say that they are \emm $y$-non-positive,):
\begin{equation}\label{neg-y}
\bxun +\bxde  = -\bx-\by \quad \hbox{and} \quad \bxun \bxde  =-x \by.
\end{equation} 
Hence every term of~\eqref{eqA} is a series in $t$ whose coefficients are Laurent polynomials in $x$ and $y$. We claim that extracting from the left-hand side of~\eqref{eqA} the non-negative part in $x$ and $y$  gives $Q(x,y)$. First, the second term of~\eqref{eqA} is $y$-negative, and hence does not contribute.  Then
\[
A_1=\bx\, \frac{x_1^2(\bx x_2-1) Q(x_1,y)-x_2^2(\bx x_1-1) Q(x_2,y)}{x_1-x_2},
\]
and is $x$-negative by~\eqref{neg-x}. Using $x x_1 x_2=-y$, we see that the same holds for
\[
A_3=\bx\, \by \, \frac{x_1^3(\bx^2+x_1\by)Q(x_1,x_1^2y)-x_2^3(\bx^2+x_2\by)Q(x_2,x_2^2y)}{x_1-x_2},
\]
and for
\[
A_4=\bx^2 \by\, \frac{\bx x_1^2 (\bx-\by) Q(-\bx, x_1^2\by) -\bx x_2^2 (\bx-\by) Q(-\bx, x_2^2\by)}{x_1-x_2}.
\] 
We are left with two terms. One is
\[
A_2= -\bx\,\by^2 \, \frac{ \bxun^2 (x+\bxde ) x Q(-\bxun,x^2\by)-\bxde ^2 (x+\bxun) x Q(-\bxde ,x^2\by)}{\bxun-\bxde },
\]
which is $y$-negative by~\eqref{neg-y}. The other is $A_5$, which looks more challenging because the variables in the series $Q$ mix positive and negative powers of the $x_i$'s.  Its analysis requires the following lemma.
\begin{Lemma}
\label{lem:extr}
For $a\ge 0$, the expression 
\begin{equation}
E_a:=\frac{x_1^{a+1}-x_2^{a+1}}{x_1-x_2}
\label{eq:Ea}
\end{equation}
is a polynomial in $\bx$ and $y$. Every monomial $\bx^e y^f$ that
occurs in it satisfies $f\le e$.
\end{Lemma}
\begin{proof}
  By induction on $a\ge 0$, using $E_{-1}=0$, $E_0=1$, $E_a=(x_1+x_2)E_{a-1}-x_1x_2E_{a-2}$ and~\eqref{neg-x}.
\end{proof}
Let us return to the expression~\eqref{A5} of $A_5$. Since $Q(x,y)$ is
a series in $t$ with polynomial coefficients in $x$ and $y$, it
suffices to prove that, for $i,j\ge 0$, the term obtained by replacing
$Q(x,y)$ by $x^i y^j$, namely
\[
\pm\bx^2\by\, \frac{\bx_2^2 (x_1+\bx )x_1 \bxde ^i  x_1^{2j}\by^j-\bx_1^2 (x_2+\bx)x_2\bxun^i x_2^{2j}\by^j }{x_1-x_2},
\]
has no non-negative part in $x$ and $y$. By splitting the sum and using $xx_1 x_2=-y$, it suffices to prove this for
\begin{equation}
\label{sum1}
\bx^2 \by^{j+1}\frac{ \bxde ^{2+i} x_1^{2+2j} -\bxun^{2+i} x_2^{2+2j}}{x_1-x_2} =(-1)^i x^{i} \by^{i+j+3} \frac{  x_1^{4+i+2j} - x_2^{4+i+2j}}{x_1-x_2}
\end{equation}
and for
\begin{equation}
\label{sum2}
\bx^3\by^{j+1}\, \frac{ \bxde ^{i+2} x_1^{1+2j} -\bxun^{i+2} x_2^{1+2j}}{x_1-x_2}
=
(-1)^{i+1} x^{i-1}\by^{i+j+3}\, \frac{ x_1^{3+i+2j} - x_2^{3+i+2j}}{x_1-x_2}.
\end{equation}
By Lemma~\ref{lem:extr}, any monomial $x^ay^b$ that occurs in~\eqref{sum1} satisfies
\[
a=i-e, \qquad b=f-i-j-3,
\]
with $f\le e$. 
Saying that $a$ and $b$ are both non-negative means that
$e\le i$ and $f \ge i+j+3$, so that
\[
e+j+3\le f \le e,
\]
which is impossible for $j \ge 0$. A similar argument proves
that~\eqref{sum2} contains no monomial that would be non-negative in
$x$ and in $y$.  So  the non-negative part of the left-hand side
of~\eqref{eqA} is indeed $Q(x,y)$. This tricky extraction deserves a
proposition. 
\begin{Proposition}\label{prop:F}
  The \gf\ $Q(x,y)$ of quadrant walks with steps in $\cS=\{10, \bar 1 0, 0 \bar
  1, \bar 2 1\}$ is the non-negative part (in $x$ and $y$) of the
  rational series
\[
R(x,y)= {\frac { \left( {x}^{2}+1 \right) \left( x+y \right)  \left( y-x \right)
\left( {x}^{2}y-2\,x-y \right) \left( {x}^{3}-x-2\,y \right)
}{{x}^{7}{y}^{3} \left(1-t(x+\bx+\bx^2 y +\by) \right)}},
\]
seen as a power series in $t$ with coefficients in $\qs[x,\bx,y,\by]$.
\end{Proposition}

{From this, one can derive interesting results for the specialization $Q(0,0)$
counting excursions.}

  \begin{Corollary}\label{cor:explicit}
For $\cS=\{10, \bar 1 0, 0 \bar
  1, \bar 2 1\}$,   the sequence $e_n:=q(0,0;2n)$ counting  excursions 
 satisfies a linear recurrence relation of order $2$:
\[
( n+3
 )  ( n+2 )  ( n+1 ) e_n =
 12 ( 2\,n-1 )  ( 2\,n-3 )  ( n-1
 ) e _{ n-2 } +4 ( 2\,n-1 )  ( n+2
 )  ( n+1 ) e_ { n-1 },
\]
with $e_0=e_1=1$. It is not hypergeometric.

The associated \gf\ admits an
expression in terms of hypergeometric series: 
\begin{equation*}
Q(0,0) = 
\frac{3}{4t}
+
\frac{{9t-2}}{2t^2} 
\int
\frac{(1+4t)^{3/2}}{(9t-2)^2}
\left(
\twoFone{-\frac32}{\frac32}{2}{\frac{16\,t}{1+4t}}
+ 
2\times \twoFone{-\frac12}{\frac32}{3}{\frac{16\,t}{1+4t}}
\right).
\end{equation*}
 \end{Corollary}

\begin{proof} (sketched)
{The recurrence relation is easily guessed from the first few values of $e_n$.
It can be proved using computer algebra and the approach
of~\cite{BoChHoKaPe17}. The idea is to write $Q(0,0)$ as the constant
coefficient (w.r.t. $x$ and $y$) of the rational function {$R(x,y)$}, then to
apply creative telescoping techniques. This proves that $Q(0,0)$ satisfies an
explicit linear differential equation of order 4, from which the validity of
the above linear recurrence relation for $e_n$ is easily deduced. The fact
that the sequence $(e_n)$ is not hypergeometric follows from Petkov\v sek's
algorithm~\cite{AB}.}

 The use of $_2F_1$ solving algorithms~\cite{BoChHoPe11,ImHo17,BoChHoKaPe17}
then provides a closed-form expression of $Q(0,0)$.
\end{proof}

%%%%%%%%%%%%%%%%%%%%%%%%%%%%%%%%%%%%%%%%%%%%%%%%%%%%%%%%%%%%%%
\section{The one-dimensional case revisited}
\label{sec:1D}
%%%%%%%%%%%%%%%%%%%%%%%%%%%%%%%%%%%%%%%%%%%%%%%%%%%%%%%%%%%%%%

So far we have only studied sporadic models. We now consider a family of
models, namely general one-dimensional models. We take $\cS\subset \zs$ and
denote by $-m$ (resp. $M$) the smallest (resp. largest) step of $\cS$; to
avoid trivial cases we assume $m\ge 0$ and $M>0$. Finally, we allow step
weights taken in some algebraically closed field $\GF$ of characteristic zero.
The indeterminates $t$ and $x$ are algebraically independent over $\GF$. The
step polynomial is then
\[
S(x)=\sum_{i\in \cS} w_i x^i,
\]
where $w_i$ is the weight of the step $i$. The weight of a walk is the product of the weights of its steps.

Let us first recall the standard solution, originally obtained by
Gessel~\cite{gessel-factorization} (see
also~\cite[Ex.~3]{bousquet-petkovsek-recurrences}
and~\cite{banderier-flajolet}). It involves auxiliary series $X_i$, which are
fractional series in the length variable $t$, algebraic over $\GF(t)$.

\begin{Proposition}\label{prop:dim1-standard}
The kernel $K(x)=1-tS(x)$, when solved for $x$,  admits $m+M$ roots, which are Puiseux series in $t$ with coefficients in $\GF$. Exactly $m$ of these roots, denoted $X_1, \ldots, X_m$, are finite at $t=0$ (and in fact, vanish at $t=0$). Let us denote by $X_{m+1}, \ldots, X_{m+M}$ the other ones.

The \gf\ $Q(x;t)\equiv Q(x)$ is 
\begin{equation}
\label{F-dim1-standard}
Q(x)= \frac  {  \prod_{i=1} ^m (1-\bx X_i) }  {K(x)} = -\frac 1  {tw_M}  \prod_{i=m+1} ^{m+M} \frac 1 {x-X_i}.
\end{equation}
\end{Proposition}
We recall the proof given in~\cite[Ex.~3]{bousquet-petkovsek-recurrences}
or~\cite{banderier-flajolet}, for comparison with the approach of this paper.
Roughly speaking, the standard solution is obtained by \emm canceling the
kernel, by appropriate specializations of $x$, while the approach of this
paper is more algebraic and consists in playing with certain invariance
properties of the kernel.

 \begin{proof} 
  The  statements of the proposition dealing with the roots of the
  kernel come from the fact that the equation $K(x)=0$, once written
  as a polynomial equation in $x$ (that is, as $x^m K(x)=0$), has
  degree $m+M$ in $x$, reducing to $m$  when $t=0$ (see~\cite[Prop.~6.1.8]{stanley-vol2}).

  Let us write $Q(x)=\sum_{i\ge 0 } x^i Q_i,$ where $Q_i$ counts walks ending at abscissa $i$. The functional equation reads
  \begin{equation}\label{eq-func-1D}
  K(x)Q(x)=1- \sum_{k=-m}^{-1}x^k G_k,
  \end{equation}
  where 
  \[
  G_k=t \sum_{i\in \cS, i\le k} w_i Q_{k-i}.
  \]
  So we have $m$ unknown series $G_{-1}, \ldots, G_{-m}$ (or
  equivalently, $Q_0, \ldots, Q_{m-1}$) on the
  right-hand side of the functional equation. When we replace $x$ by
  $X_i$ in~\eqref{eq-func-1D}, for $1\le i\le m$, both the left and
  right-hand sides vanish (we only use the ``small'' roots $X_1,
  \ldots, X_m$, because the substitution by a root involving
  negative powers of $t$ may be undefined). But the right-hand side is a polynomial in
  $\bx$, of degree $m$ and constant term $1$. Hence it must be equal
  to $\prod_{i=1}^m (1-\bx X_i)$, and this gives the first expression
  of $Q(x)$. The second one follows by factoring  $K(x)$  as 
  \begin{equation}\label{K-fact-1D}
  K(x)= -t w_M \prod_{i=1}^m (1-\bx X_i) \prod_{i=m+1}^{m+M} (x-X_i).
  \end{equation}
  (The factor $-tw_M$ is obtained by extracting the coefficient of $x^M$ in $K(x)$.)
\end{proof}

We now present the expression provided by the method of this paper. Rather
than algebraic series in $t$ (the $X_i$'s), it involves algebraic series in
$\bx$ (denoted by $x_i$), and then the extraction of a non-negative part.
Admittedly, it is not as attractive as the standard solution. In particular,
it does not make the algebraicity of $Q(x)$ clear, unless the largest step is
1. But we show later how to recover the standard solution from it. One
surprising feature of this solution is that, as foreseen in Example A, it
involves expansions in $\bx$ rather than $x$.

\begin{Proposition}\label{prop:dim1-alg} 
The equation  $S(X)=S(x)$  (when solved for $X$) admits $m+M$ roots, which can be taken in the field of Puiseux series in $\bx:=1/x$ with coefficients in $\GF$. Exactly $m$ of these roots, denoted $x_1, \ldots, x_m$, contain no positive power of $x$ (and, in fact, have no constant term either).

The \gf\ $Q(x;t)\equiv Q(x)$ is
\begin{equation}\label{1D-gen}
Q(x)= [x^{\ge}] \frac{\prod_{j=1}^m (1-\bx x_j)}{K(x)},
\end{equation}
where the right-hand side is expanded first in $t$, then in $\bx$.

If the largest step of $\cS$ is $M=1$ the right-hand side of~\eqref{1D-gen} is rational, and
\begin{equation}\label{1D-M1}
Q(x)= [x^{\ge}]  \frac{ S'(x)}{w_1 K(x)}.
\end{equation}
\end{Proposition}

We will use the following lemma, {which is a simple application
  of the Lagrange interpolation formula}~\cite[Lemma~13]{mbm-chapuy-preville}.

\begin{Lemma}\label{lem:Lagrange}
Let $u_0, u_1, \dots , u_m$ be $m+1$  variables. Then
\[
\sum_{i=0}^m \frac{{u_i}^d}{\prod_{j\neq i}{(u_i-u_j)}} =
\begin{cases}
  1 & \hbox{if } d=m,\\
  0 & \hbox{if } 0\le d <m.
  \end{cases}
\]
\end{Lemma}

\begin{proof}[Proof of Propositions~\ref{prop:section-free1D} and~\ref{prop:dim1-alg}]
We first establish the section-free equation of
Proposition~\ref{prop:section-free1D}. The equation $S(X)=S(x)$ has
$m+M$ solutions (counted with multiplicity), including $X=x$, which
form the orbit of $x$. These
solutions  are in fact distinct: a solution of $S'(X)=0$ belongs to
the ground field $\GF$, and cannot satisfy $S(X)=S(x)$ since $x$ is an indeterminate.

Let
$u_0, \ldots, u_m$ be $m+1$ {distinct} orbit elements. For $0\le i \le m$, the
functional equation~\eqref{eq-func-1D} specializes into
\[
K(x)Q(u_i)=1- \sum_{k=-m}^{-1}u_i^k G_k.
\]
Note that $K(u_i)=K(x)$ since $K(x)=1-tS(x)$. We can eliminate the $m$ series $G_k$ by taking an appropriate linear combination of our $m+1$
equations, namely:
\begin{align}
  K(x) \sum_{i=0}^m \frac{u_ i^m}{\prod_{j\not = i} (u_i-u_j)}Q(u_i)
&=\sum_{i=0}^m \frac{u_i^m}{\prod_{j\not = i} (u_i-u_j)}
-  \sum_{k=-m}^{-1}G_k \sum_{i=0}^m \frac{u_i^{k+m}}{\prod_{j\not =
    i} (u_i-u_j)}\nonumber\\
&=
1\label{1D-lc-bis} 
\end{align}
by Lemma~\ref{lem:Lagrange}.

We have thus exhibited $\binom{m+M}{m+1}$ section-free equations, each
involving $m+1$ orbit equations, but we still need to prove that they generate
all section-free equations. So let us take a generic section-free equation,
say
\[
  \sum_{i=0}^{m+M-1} \alpha_i K(x) Q(u_i)= \sum_{i=0}^{m+M-1} \alpha_i
\left(1-\sum_{k=-m}^{-1}u_i^k G_k\right)
=\sum_{i=0}^{m+M-1} \alpha_i ,
\]
where $u_0, u_1, \ldots, u_{m+M-1}$ are now all orbit elements.
By subtracting a number of versions of~\eqref{1D-lc-bis} (with well
chosen $u_i$'s and well chosen weights), we can assume that this
equation only involves (at most) $m$ of the $u_i$'s, say $u_1, \ldots,
u_m$. Then saying that this equation is section-free means that for
all $k$ in $\llbracket -m, -1\rrbracket$,
\[
\sum_{i=1}^m \alpha_i u_i^k=0.
\]
But  the determinant of this system
is not zero (since the $u_i$'s are distinct), and thus all
$\alpha_i$'s must be zero.

\medskip
We now go on with the proof of Proposition~\ref{prop:dim1-alg}.
The equation $S(X)=S(x)$, written as a polynomial in $\bx$ and $X$, reads
\[
\bx^M \sum_{i\in \cS} w_i X^{m+i}=X^m \sum_{i\in \cS}w_i \bx^{M-i}.
\]
The number of solutions $X$ that are fractional power series in $\bx$  is the degree in $X$ of the above polynomial, once evaluated at $\bx=0$ (see again~\cite[Prop.~6.1.8]{stanley-vol2}), hence $m$. From now on we denote these roots by $x_1, \ldots, x_m$, and it is clear that $x$ is not among them so we denote $x_0=x$.

\medskip
We now write the section-free equation~\eqref{1D-lc-bis} with
$u_i=x_i$, and isolate $Q(x_0)=Q(x)$:
\begin{equation}\label{F-iso}
Q(x)+ \prod_{j=1}^m (1-\bx x_j) \sum_{i=1}^m
\frac{x_i^m}{\prod_{0\le j\not = i \le m} (x_i-x_j)}Q(x_i)
=
\frac{\prod_{j=1}^m (1-\bx x_j)}{K(x)}.
\end{equation}
{Comparing with~\eqref{1D-gen} shows {that} we have to prove
  that the second term in the left-hand side, once expanded as a
  series in $t$, only contains  negative
  powers of $x$.}
In the coefficient of $Q(x_i)$, the term $(1-\bx x_i)$ coming from the
numerator gets simplified with the term $(x_i-x_0)=(x_i-x)=-x(1-\bx
x_i)$ coming from the denominator. Hence the least common denominator
of the coefficients of all $Q(x_i)$ is the Vandermonde determinant in
$x_1, \ldots, x_m$. We can thus rewrite the second term as follows: 
\begin{multline}
  \prod_{j=1}^m (1-\bx x_j) \sum_{i=1}^m
\frac{x_i^m}{\prod_{0\le j\not = i \le m} (x_i-x_j)}Q(x_i)
=- \bx \sum_{i=1}^m \left(x_i^m Q(x_i)
\prod_{1\le j\not = i \le m} \frac{ 1-\bx x_j}
{x_i- x_j}\right)\label{vdm}
\\ 
=\frac \bx{\prod_{1\le k <\ell \le m} (x_k-x_\ell)}\sum_{i=1}^m\left( (-1)^i x_i^m Q(x_i) \prod_{1\le
  j\not = i \le m} (1-\bx x_j) \prod_{
\substack{1\le k < \ell\le m \\ k,
  \ell\not = i}} (x_k-x_\ell)\right).
\end{multline}
The sum over $i$ is easily checked to be an antisymmetric expression
in $x_1, \ldots, x_m$. More precisely, if we exchange in this sum $x_a$ and
$x_{a+1}$, the summands involving $Q(x_a)$ and $Q(x_{a+1})$ are
exchanged, and their  signs change (because of the factor $(-1)^i$), and for $i \not \in\{a, a+1\}$ the
sign of the summand involving $Q(x_i)$ changes (because of the factor
$(x_a-x_{a+1})$ occurring in the rightmost product).  Thus,  dividing
the sum over $i$ by the Vandermonde {determinant} in $x_1, \ldots, x_m$ gives a
series in $t$ with \emm polynomial, coefficients in $\bx, x_1, \ldots,
x_m$.
Hence, once  expanded in $t$ and~$\bx$, the right-hand side of~\eqref{vdm} contains
only negative powers of $x$ (because the $x_i$'s contain no positive
power of $x$ and there is a factor $\bx$). We now return
to~\eqref{F-iso}, which we expand in powers of $t$ and $\bx$. 
The expression~\eqref{1D-gen} of $Q(x)$ follows. 

\medskip
Now assume $M=1$. Then $x_1, \ldots, x_m$ are \emm all, roots of $S(X)=S(x)$ except $X=x$. That is,
\[
\frac {S(X)-S(x)}{X-x}= w_1 \prod_{j=1}^m (1-x_j/X).
\]
Taking the limit as $X\rightarrow x$ gives
\begin{equation}\label{S-prime}
S'(x)=w_1\prod_{j=1}^m (1-\bx x_j).
\end{equation}
Substituting into~\eqref{1D-gen} gives~\eqref{1D-M1}.
\end{proof}

\noindent{\bf Why  Proposition~\ref{prop:dim1-alg}  implies Proposition~\ref{prop:dim1-standard}.}
We now derive from~\eqref{1D-gen} the standard
expression~\eqref{F-dim1-standard}. We start from the
factorization~\eqref{K-fact-1D} of the kernel. It gives the following
partial fraction decomposition {in $x$}:
\[
\frac{1}{K(x)}= -\frac 1  {tw_M} \sum_{i=1}^m \frac{\bx X_i^m }{(1-\bx X_i)
\prod_{j\not = i}(X_i-X_j)}
+\frac 1  {tw_M}  \sum_{i=m+1}^{m+M}\frac {X_i^{m-1}} {(1-x/X_i)\prod_{j\not = i}(X_i-X_j)}.
\]
The expansion in $\bx$ of the term $A(\bx):=\prod_{j=1}^m (1-\bx x_j)$ only involves non-positive powers of $x$, hence~\eqref{1D-gen} implies
\begin{equation}\label{QA}
Q(x)= \frac 1  {tw_M} [x^{\ge}]A(\bx) \sum_{i=m+1}^{m+M}\frac {X_i^{m-1}}
{(1-x/X_i)\prod_{j\not = i}(X_i-X_j)}.
\end{equation}
Recall that, as $x_1, \ldots, x_m$ themselves,  $A(\bx)$ is a
fractional power series in $\bx$ with  coefficients in~$\GF$, say
$A(\bx)=\sum_{n\ge 0} a_n \bx^{n/p}$, for a positive integer $p$. In
fact we can take $p=1$. Indeed, by~\cite[Prop.~6.1.6]{stanley-vol2},
for $1\le i\le m$, every conjugate of $x_i$  over the field
$\GF((\bx))$ of Laurent series in $\bx$ is one of the $x_j$'s, with
$1\le j\le m$; hence $\prod_{i=1}^m (u-x_i)$ is a product of minimal
polynomials over $\GF((\bx))$, and thus only involves integer powers
of $\bx$ in its expansion.

{Now let us return to~\eqref{QA}, and focus on the term $A(
\bx)/(1-x/X_i)$.} Recall that for $i>m$,   $X_i$ is  a
Puiseux series in $t$, infinite at $t=0$. Thus  $1/X_i$ is a
fractional power series in $t$, vanishing at $t=0$, and hence
$A(1/X_i)$ is also a fractional series in~$t$. Moreover, {in the ring
of fractional series in $t$ with coefficients in $\GF[[\bx]]$, we have}
\begin{align*}
  [x^{\ge}]\frac{A(\bx)}{1-x/X_i}& = [x^{\ge}]\left( \sum_{m\ge 0}
  \frac {x^m}{X_i^m} \sum_{n\ge 0} a_n
  \bx^{n}\right) \\
&=\sum_{n\ge 0} a_n\sum_{m\ge n} \frac {x^{m-n}}{X_i^m} 
  \\ &=\sum_{n\ge 0} a_n \frac{1}{X_i^n(1-x/X_i)}
\\ &=\frac{A(1/X_i)}{1-x/X_i}.
\end{align*}
Returning to~\eqref{QA}, this gives:
\[
Q(x)= \frac 1  {tw_M} \sum_{i=m+1}^{m+M}\frac {X_i^{m-1} A(1/X_i)}
{(1-x/X_i)\prod_{j\not = i}(X_i-X_j)}.
\]
Thus it remains to determine $A(1/X_i)$ when $X_i$ is one of the roots
of the kernel that diverges at $t=0$. That is, we have to know the
values of $x_1, \ldots, x_m$ when $x$ is $X_i$. Recall the definition
of these $x_j$: they are power series in (a rational power of) $\bx$,
satisfying $S(x_j)=S(x)$. Specializing this at $\bx=1/X_i$ shows that  when $x=X_i$, the series $x_1, \ldots, x_m$ are power series in (a fractional power of) $t$, satisfying $S(X_i)=S(x_j)$. But $X_i$ cancels the kernel $1-tS$, hence the $x_j$ are also roots of the kernel, and since they must be finite at $t=0$, they are $X_1, \ldots, X_m$. This holds for any $X_i$ with $i>m$.

Hence,
\begin{align*}
  Q(x)&= \frac 1  {tw_M} \sum_{i=m+1}^{m+M}\frac {X_i^{m-1}}
{(1-x/X_i)\prod_{j\not = i}(X_i-X_j)} \prod_{j=1}^m (1-X_j/X_i)
\\ &=
\displaystyle \frac 1 {tw_M}  \sum_{i=m+1}^{m+M}\frac {1}
{(X_i-x)\prod_ {j>m, j\not = i}(X_i-X_j)}
\end{align*}
{where we recognize the partial fraction expansion of 
\[
-\frac 1 {tw_M} \prod_{i=m+1}^{m+M}\frac{1}{x-X_i}.
\]
This gives} the second expression in~\eqref{F-dim1-standard}.

%%%%%%%%%%%%%%%%%%%%%%%%%%%%%%%%%%%%%%%%%%%%%%%%%%%%%%%%%%%%%%
\section{Two-dimensional Hadamard walks}
\label{sec:hadamard}
%%%%%%%%%%%%%%%%%%%%%%%%%%%%%%%%%%%%%%%%%%%%%%%%%%%%%%%%%%%%%%

Following~\cite{BoBoKaMe16}, we say that a  2-dimensional  model $\cS$ is \emm Hadamard, if its step polynomial can be written as:
\begin{equation}\label{S-Hadamard}
S(x,y)=U(x)+V(x)T(y),
\end{equation}
for some Laurent polynomials $U$, $V$ and $T$. {Some examples
  are shown in Figure~\ref{fig:hadamard}.}
When $T(y)=y+\by$, the model has small variations along the
$y$-axis and is symmetric with respect to the $x$-axis. It was proved
 in~\cite{bousquet-versailles,bousquet-petkovsek-knight} that the
 associated \gf\ $Q(x,y)$ is always D-finite. This holds in fact for \emm all,
 two-dimensional Hadamard models, whatever $T(y)$ is. We provide two
 proofs, one based on a simple projection argument, the other on the
 method of this paper. 

\begin{figure}[t!]
  \centering
   \includegraphics[scale=0.8]{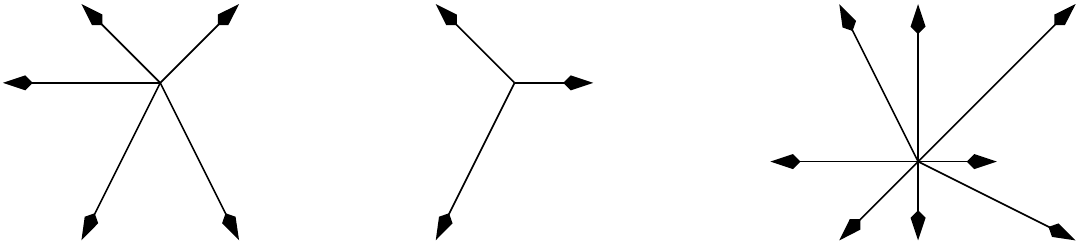}
  \caption{Some Hadamard models. The series $Q(x,y)$ is D-finite for
    all of them, and given as the non-negative part of a rational
    function for those that have small forward steps (the leftmost two).}
  \label{fig:hadamard}
\end{figure}

 \begin{Proposition}\label{prop:Hadamard-product}
   Consider a Hadamard model with step polynomial
   given by~\eqref{S-Hadamard}, and let $\cU$, $\cV$ and $\cT$ be the subsets
   of $\zs$ with generating polynomials $U(x)$, $V(x)$ and
   $T(y)$ respectively. Let $C_1(x,v;t)$ be the \gf\ of walks on $\ns$, starting
   from $0$ and taking steps in {the multiset} $\cU \cup \cV$
   {(steps in $\cU \cap   \cV$ occur twice)}, counted by the length
   (variable $t$), the position of the endpoint ($x$), and 
{the number of steps in $\cV$ ($v$)}.
Let $C_2(y;v)$ be the \gf\ of walks on $\ns$, starting from
   $0$ and taking steps in $\cT$, counted by the length ($v$) and the endpoint ($y$). Then
   $C_1(x,v;t)$ and $C_2(y;v)$ are algebraic, and the \gf\ of quadrant
   walks with steps in $\cS$ is
\[
Q(x,y;t) =\left. C_1(x,v;t) \odot_v C_2(y;v)\right|_{v=1},
\]
where $\odot_v$ denotes the Hadamard product in $v$, defined by $\sum a_n v^n\odot_v
  \sum b_n v^n = \sum a_n b_n v^n$. In particular, $Q(x,y;t)$ is D-finite.
 \end{Proposition}
 \begin{proof}
 The proof is the same as
   in~\cite[Sec.~5]{BoBoKaMe16}, but generalized (in a harmless
   fashion) to walks with arbitrary steps.  It goes by projecting quadrant walks along the $x$-axis, and
   ``decorating''  steps {of $\cV$} in this 1D walk with steps of a
   ``vertical'' walk with steps in $\cT$; we omit the details. The
   Hadamard product of algebraic (and in fact, of D-finite) series is
   known to be D-finite~\cite{lipshitz-diag}.
 \end{proof}

 The approach of this paper works systematically in the Hadamard case, and
provides the solution as the positive part of an algebraic (sometimes
rational) series, often more explicitly than the above solution. In the case
of small steps, 16 of the 19 models solvable by the method of this paper (the
leftmost branch in Figure~\ref{fig:class-2D}) are Hadamard. The three
remaining ones are shown below.

\smallskip
\begin{center}
  \begin{tabular}{ccc}
     $ \diagr{N,SE,W}$  & \hskip 4mm $  \diagr{SE,N,E,S,W,NW}$
  & \hskip 4mm$\diagr{E,W,NW,SE} $
  \end{tabular}
\end{center}
\smallskip

Consider a Hadamard model $\cS$. Let $-m$ (resp. $M$) be the valuation (resp.
degree) of $S(x,y)$ in $x$, and write similarly $-m'$ and $M'$ for the
valuation and degree in $y$. In other words, $-m$ (resp. $M$) is the smallest
(resp. largest) move in the $x$-direction, and similarly for $m'$ and $M'$. We
assume $m, m'\ge 0$ and $M, M' >0$. The solution given below has strong
analogies with the 1-dimensional case of Proposition~\ref{prop:dim1-alg}.

\begin{Proposition}\label{prop:hadamard} 
The equation $S(x,y)=S(X,y)$, solved for $X$, admits $m+M$ solutions
(including $x$ itself), which can be seen as Puiseux series in $\bx$
with coefficients in an algebraic closure of $\qs(y)$ (below we 
take Puiseux series in {$\by$}). We denote them
by $x_0(y), \ldots, x_{m+M-1}(y)$, with $x_0(y)=x$. Exactly $m$ of
them, say  $x_1(y), \ldots, x_m(y)$, do not involve positive powers of
$x$. 

The equation $S(x,y)=S(x,Y)$, now solved for $Y$, reads $T(y)=T(Y)$. It admits $m'+M'$ solutions (including $y$ itself), which can be seen as Puiseux series in $\by$ with coefficients in  $\cs$. We denote them by $y_0, \ldots, y_{m'+M'-1}$, with $y_0=y$. Exactly $m'$ of them, say $y_1, \ldots, y_{m'}$, do not involve positive powers of $y$.

The orbit of $(x,y)$ consists of all pairs $(x_i,y_j)$, for $i\in\llbracket 0, m+M-1\rrbracket$ and $j\in\llbracket 0, m'+M'-1\rrbracket$. 

The series $Q(x,y)$ reads
\begin{equation} \label{Q-Hadamard1}
Q(x,y)= [x^\ge y^{\ge}]\frac{\prod_{i=1}^m(1-\bx x_i(y))\prod_{j=1}^{m'}(1-\by y_j)}{K(x,y)},
\end{equation}
where the right-hand side is expanded first in powers of $t$, then
$\bx$, and finally $ \by$. The extraction of the non-negative part in
$x$ can be done explicitly, {and yields}:
\begin{equation} \label{Q-Hadamard2}
Q(x,y)= [ y^{\ge}]\frac{\prod_{i=1}^m(1-\bx
  X_i(y))\prod_{j=1}^{m'}(1-\by y_j)}{K(x,y)} = -[ y^{\ge}]
\frac{\prod_{j=1}^{m'}(1-\by y_j)}{tB_M(y) \prod_{i=m+1}^{m+M}(x-X_i)},
\end{equation}
where $X_1(y), \ldots, X_m(y)$ are the roots (in $x$) of $1-tS(x,y)$,
seen as Puiseux series in $t$ {with coefficients in the
  algebraic closure of $\qs(y)$}, that are finite at $t=0$, and $X_{m+1},
\ldots, X_{m+M}$ are the other ones. The polynomial $B_M(y)$ is the
coefficient of $x^M$ in $S(x,y)$.

If $M=1$, then the derivative of $S(x,y)$ with respect to $x$ factors as
\[
S_x(x,y)= B_1(y) \prod_{i=1}^m (1-\bx x_i(y)).
\]
Similarly, if
$M'=1$, then 
\[
T_y(y)= \prod_{j=1}^{m'} (1-\by y_j).
\]
This simplifies the above expressions. In particular, when all forward steps are small ($M=M'=1$), we can write $Q(x,y)$ as the non-negative part of a simple rational function:
\begin{equation}\label{M=Mp=1}
Q(x,y) =   [x^{\ge}y^{\ge}]
\frac{S_x(x,y)T_y(y)}{B_1(y)K(x,y)}.
\end{equation}
\end{Proposition}
\begin{proof}
  The statements dealing with roots are in essence one-dimensional, and follow from Proposition~\ref{prop:dim1-alg} since we allowed weights in the previous section.

  We next want to build the orbit of $(x,y)$. By definition of the $x_i$'s and $y_j$'s we have $(x_i,y) \approx (x,y) \approx (x,y_j)$ for all $i$ and $j$. Now
  \begin{align*}
  S(x_i,y_j)&=U(x_i)+V(x_i)T(y_j) 
  \\
  &= U(x_i)+V(x_i) T(y) \ \  \hbox{ by definition of } y_j
  \\
  &= S(x_i,y). 
  \end{align*}
  Thus $(x_i,y_j)\approx (x_i,y)$, and all pairs $(x_i, y_j)$ are in
  the orbit. In this collection, every element $(x_i,y_j)$ is
  1-adjacent to $m+M-1$ other elements, and 2-adjacent to $m'+M'-1$
  other elements, hence the orbit is complete (Lemma~\ref{lem:ind}).

  The functional equation has the following general form (see~\eqref{eq:quadrant}):
  \[
  K(x,y)Q(x,y)=1- \sum_{k=1}^m \bx^k R_k(y)- \sum_{\ell=1}^{m'}
{ \by^\ell }S_\ell(x),
  \]
  for some series $R_k(y)$ and $S_\ell(x)$. A similar equation holds with $(x,y)$ replaced by any element $(x_i,y_j)$ of the orbit. The fact that the orbit is a Cartesian product allows us to construct a section-free equation by mimicking the argument that led to~\eqref{1D-lc-bis}:
  \[
  K(x,y) \left( \sum_{i=0}^m \sum_{j=0}^{m'}\frac{x_i^m
    y_j^{m'} \, Q(x_i, y_j)}{\prod_{0\le k \not= i \le m} (x_i-x_k)
  \prod_{0\le \ell \not= j \le m'} (y_j-y_\ell)} \right) =1.
  \]
  Equivalently, after isolating  $Q(x_0,y_0)=Q(x,y)$:
\begin{multline*}
  Q(x,y)- \by \sum_{j=1}^{m'} y_j^{m'}Q(x,y_j) \prod_{1\le \ell \not =
    j \le m'} \frac{1-\by y_\ell}{y_j-y_\ell}
-\bx \sum_{i=1}^m x_i^m Q(x_i,y) \prod_{1\le k\not = i \le m} \frac{1-\bx x_k}{x_i-x_k}
\\
+\bx \by \sum_{i=1}^m\sum_{j=1}^{m'} x_i^m y_j^{m'}Q(x_i,y_j) 
\prod_{1\le k\not = i \le m} \frac{1-\bx x_k}{x_i-x_k}
\prod_{1\le \ell\not = j\le m'} \frac{1-\by
  y_\ell}{y_j-y_\ell}
\\
= \frac{\prod_{i=1}^m (1-\bx x_i)\prod_{j=1}^{m'} (1-\by y_j)}{K(x,y)}.
\end{multline*}
We now expand the coefficient of $t^n$ in this identity in powers of
$\bx$ (with coefficients in the field of Puiseux series in 
{$\by$}), and extract the non-negative powers of
$x$.  The coefficients of the first two terms on the first line
{(those involving $Q(x,y)$ and $Q(x,y_j)$)} are clearly
non-negative in $x$. By recycling our analysis of~\eqref{vdm}, we see
that the coefficient of $t^n$ in the third term {(involving
  $Q(x_i,y)$)} 
is a \emm polynomial, in $y$, $\bx$, $x_1, \ldots, x_m$, multiplied by $\bx$, and thus only involves negative powers of $x$ and  does not contribute. A similar argument shows that the second line does not contribute either. We are thus left with
\[  
Q(x,y)- \by \sum_{j=1}^{m'} y_j^{m'}Q(x,y_j) \prod_{\ell \not = j \in \llbracket1, m'\rrbracket} \frac{1-\by y_\ell}{y_j-y_\ell}
=
[x^\ge] \frac{\prod_{i=1}^m (1-\bx x_i)\prod_{j=1}^{m'} (1-\by y_j)}{K(x,y)}.
\]
 The symmetry argument applied earlier to~\eqref{vdm} shows that the
 sum over $j$ is a series in $t$ whose coefficients are polynomials in
 $x$, $\by, y_1, \ldots, y_m$. Hence a final expansion  in powers of
 $\by$, followed by the extraction of  non-negative powers of $y$
 gives 
 the
 {first expression~\eqref{Q-Hadamard1} of $Q(x,y)$}. The second
 one, {that is~\eqref{Q-Hadamard2}}, follows by combining the one-dimensional results of Propositions~\ref{prop:dim1-standard} and~\ref{prop:dim1-alg}. Indeed, {Proposition~\ref{prop:dim1-alg}} shows that
\[
[x^\ge ]\frac{\prod_{i=1}^m(1-\bx x_i(y))}{K(x,y)}
\]
counts walks with steps in $\cS$ confined to the half-plane $\{(i,j): i\ge 0\}$, and {Proposition~\ref{prop:dim1-standard}} gives an alternative expression for this series.

The rest of the proof follows the same lines as the end of the proof of
Proposition~\ref{prop:dim1-alg} (see in particular~\eqref{S-prime}).
\end{proof}

\medskip\noindent
{\bf Example: a Hadamard model with small forward steps.} 
Take $\cS=\{10,\bar11, \bar1\bar2\}$. The step polynomial is 
\[
S(x,y)= x+\bx (y+\by^2)=U(x)+V(x)T(y)
\]
with $U(x)=x$, $V(x)=\bx$ and $T(y)=y+\by^2$, so this is a Hadamard model. Moreover, the forward steps are small, so that the simple formula~\eqref{M=Mp=1} holds:
\[
Q(x,y)
= [x^{\ge} y^{\ge}]\frac{(1-\bx^2 (y+\by^2))(1-2\by^3)}{1-t(x+\bx (y+\by^2))}.
\]
The number of walks of length $n$ ending at $(i,j)$ is non-zero if and only if $n=i+2j+6m$ for some $m$,  in which case 
\begin{equation}\label{hadamard-sol} 
q(i,j;n)=\frac{(i+1)(j+1) n!}{m! (2m+j+1)! (3m+i+j+1)!}.
\end{equation}
\qee

\medskip
\noindent{\bf Example: a Hadamard model with a large forward step.} 
Let us now  reverse  the above  steps. The step polynomial becomes
\[
S(x,y)=\bx +x(\by+y^ 2)
\] and
is of course still Hadamard. With the notation of Proposition~\ref{prop:hadamard},
$m=m'=1$, 
\[
x_1(y)=\frac{\bx}{\by+y^2} \qquad \hbox{and} \qquad y_{1}=
 \frac{-1+  \sqrt{4\by^3+1}}{2\by}.
\]
Indeed $y_1$ is a power series in $\by$, while {its conjugate root $y_2$} contains a term
$-y$ in its expansion. {The two expressions of  Proposition}~\ref{prop:hadamard} read:
\[
Q(x,y)= 
[x^\ge y^\ge ]\frac{(1-\bx x_1(y))(1-\by y_1)}{K(x,y)} =- [ y^\ge ]
\frac{\bx(1-\by y_1)}{t(\by + y^2) (1-\bx X_2)},
\]
with 
\[
X_2= \frac{1+\sqrt{1-4t^2(\by+y^2)}}{2t(\by+y^2)}.
\]
As before,  we expand the right-hand side first in $t$, then $\bx$, then  $\by$.

%%%%%%%%%%%%%%%%%%%%%%%%%%%%%%%%%%%%%%%%%%%%%%%%%%%%%%%%%%%%%%
\section{Quadrant walks with steps in \texorpdfstring{$\boldsymbol{\{-2,-1,0, 1\}^2}$}{\{-2,-1,0,1\}} }
\label{sec:m21}
%%%%%%%%%%%%%%%%%%%%%%%%%%%%%%%%%%%%%%%%%%%%%%%%%%%%%%%%%%%%%%
In this section, we explore systematically all models obtained by taking $\cS$
in $\{-2,-1,0, 1\}^2 \setminus{(0,0)}$, with the (ultimate) objective of
reaching a classification similar to that of quadrant walks with small steps
(Figure~\ref{fig:class-2D}). Our results are summarized in
Figure~\ref{classificationm2}. In Section~\ref{sec:final} we discuss the
classification of orbits (not of \gfs!) for models in $\{-1,0, 1,2\}^2
\setminus{(0,0)}$.

\begin{figure}[htb]
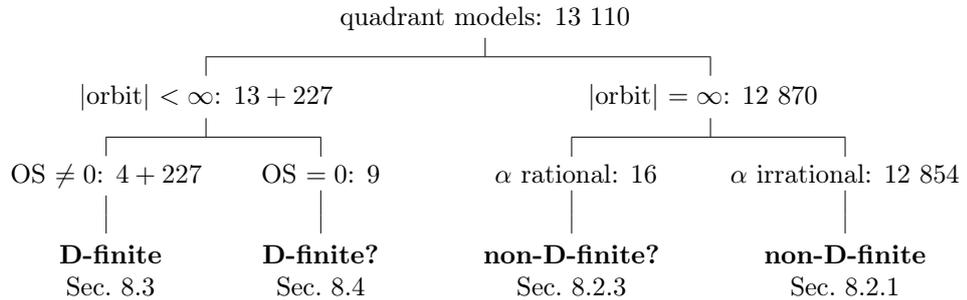

\begin{center}\edgeheight=7pt\nodeskip=2.23em\leavevmode
 \tree{quadrant models: 13 110}
{
\tree{|orbit| $<\infty$: $13 + {227}$}
{
\tree{OS $\not = 0$: ${4 + {227}}$}{\tree{
\begin{tabular}{c} \textbf{D-finite}\\   
Sec.~\ref{sec:works}\end{tabular}}{}}
\tree{OS $=0$: {9}}{\tree{
\begin{tabular}{c} 
\textbf{D-finite?}\\
Sec.~\ref{sec:interesting}
\end{tabular}
}{}}}
\tree{\hskip -2mm|orbit| $=\infty$:
{12 870}}{
\tree{ $\alpha$ rational: 16}{\tree{
\begin{tabular}{c}\textbf{non-D-finite?}\\
Sec.~\ref{sec:embarassing}\end{tabular}
}{}}
\tree{$\alpha$ irrational: 12 854}
{\tree{
\begin{tabular}{c}\textbf{non-D-finite}\\
Sec.~\ref{subsec:exponent}
\end{tabular}
}{}}
}}

\end{center}
\caption{Partial classification of quadrant walks with steps in $\{-2,-1,0,1\}^2$, when at least one large backward step is allowed. The approach of this paper solves the 231 models on the leftmost branch, 
{including} 227 Hadamard models.}
\label{classificationm2}
\end{figure}

%==========================================
\subsection{The number of relevant models}
%==========================================
We first proceed as in~\cite[Sec.~2]{BoMi10} in order to count, among the
$2^{15}$ possible models {(Figure~\ref{fig:15})}, those that are really
distinct and relevant. Clearly, we do not want to consider separately two
models that only differ by an $x/y$-symmetry, as such models are isomorphic.
Moreover, for certain models, forcing walks to lie in some half-plane
automatically forces them to {remain in} the first quadrant. This happens, for
instance, for $\cS=\{\nearrow, \uparrow, \swarrow\}$ and the right half-plane.
Half-plane models are essentially 1-dimensional and thus have an algebraic
\gf, which can be determined in an automatic fashion
(Proposition~\ref{prop:dim1-standard}).

\begin{figure}[htb]
  \centering
 \includegraphics{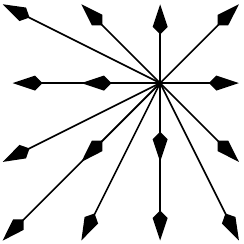}
  \caption{The 15 allowed steps.}
  \label{fig:15}
\end{figure}

Using the same arguments as in~\cite[Sec.~2]{BoMi10}, we first determine the
number of step sets $\cS$ that contain at least an $x$-positive, an
$x$-negative, a $y$-positive and a $y$-negative step. More precisely, we count
such sets by their cardinality. An inclusion-exclusion argument gives their
generating polynomial as:
\[
P_1(z)=
(1+z)^{15}-2(1+z)^{11}-2(1+z)^7+2(1+z)^3+(1+z)^8+2(1+z)^5+(1+z)^3-2(1+z)^2-2(1+z)+1.
\]
The term $(1+z)^{15}$ counts all step sets, while $(1+z)^{11}$
  counts those that contain no $x$-positive step, $(1+z)^7$ those that
  contain no $x$-negative step, $(1+z)^3$ those that contain no
  $x$-positive nor $x$-negative step, and so on. We refer
  to~\cite[Sec.~2]{BoMi10} for a detailed argument.
Then, we must exclude sets in which no step belongs to $\ns^2$. This leaves fewer step sets, counted by:
\[
P_2(z)=P_1(z) - \big((1+z)^{12} -2(1+z)^{10} +(1+z)^8\big).
\]
We also do not wish to consider step sets such that all walks confined to the
right half plane $x \ge 0$ are automatically quadrant walks. As in the case of
small steps, this means that all steps $(i,j)$ of $\cS$ satisfy $j\ge i$. That
is, we have an upper diagonal model. The generating polynomial of such sets,
satisfying the above conditions (steps in all directions, at least one step in
$\ns^2$) is
\[
z\left( (1+z)^8-(1+z) ^5\right),
\]
where the factor $z$ accounts for the step $(1,1)$, which is necessarily in
such a set. Symmetrically, we need to exclude lower diagonal models, and avoid
excluding twice the models that are both upper and lower diagonal. We are left
with a collection of step sets counted by
\[
P_3(z)=P_2(z)-2z \left( (1+z)^8-(1+z) ^5\right) + z(2z+z^2).
\]

Finally, if two models differ only by a diagonal symmetry, we do not
want to consider them both. We thus have to count separately the
models counted by $P_3$ that have an $x/y$ symmetry. {Mimicking 
 the above argument, and including the symmetry constraint,} gives:
\[
P_1^\sym(z)= (1+z)^3(1+z^2)^6 -(1+z)^2(1+z^2)^3 -(1+z)(1+z^2) +1,
\]
\[
P_2^\sym(z)= P_1^\sym(z)- \big( (1+z) ^2(1+z^2)^5-(1+z)^2(1+z^2)^3\big),
\]
and 
\[P_3^\sym= P_2^\sym- z(2z+z^2).
\]
We have thus restricted the collection of models that we have to study to 13 189 models, with generating polynomial
\begin{multline*}
  \frac 1 2 \left(P_3(z)+P_3^\sym(z)\right)=
z^{15}+9 z^{14}+57 z^{13}+236 z^{12}+691 z^{11}+1481 z^{10}+2374
z^9+2872 z^8\\
+2610 z^7+1749 z^6+826 z^5+248 z^4+35 z^3.
\end{multline*}
Among these, we know from~\cite{BoMi10} that those with small steps are counted by
\[
7\,{z}^{3}+23\,{z}^{4}+27\,{z}^{5}+16\,{z}^{6}+5\,{z}^{7}+{z}^{8},
\]
and we are thus left with 13 110 models with at least one large backward step,
counted by 
\[
z^{15}+9 z^{14}+57 z^{13}+236 z^{12}+691 z^{11}+1481 z^{10}+2374 z^9+2871 z^8+2605 z^7+1733 z^6+799 z^5+225 z^4+28 z^3.
\]
Note that no model in our  collection is included in a
half-plane. {This will allow us to apply Theorem~\ref{thm:theta} systematically.}

%====================================================
\subsection{The size of the orbit}
\label{sec:size}
%======================================================

%========================================================
\subsubsection{The excursion exponent}
\label{subsec:exponent}
%========================================================
Consider a model $\cS$ in our collection. Recall that
if the quantity $c$ defined in  Theorem~\ref{thm:theta} cannot be written as $\cos \theta$ with $\theta\in \pi
\qs$, then the orbit of $\cS$ is infinite and $Q(x,y;t)$ is not D-finite.
In order to decide if $c$ is of the requested form, we apply the
following procedure, borrowed from~\cite{BoRaSa14}.
\begin{enumerate}\itemsep=1em
\item  Compute a polynomial $P(C)$ that admits  $c$ as a root. This is done by eliminating the variables $x,y$ and $u$ from the polynomial system comprised of (the numerators of):
\[
  S_x(x,y), \quad S_y(x,y), \quad C^2-\frac{S_{xy}(x,y)}{S_{xx}(x,y)^2S_{yy}(x,y)^2}, \quad 1-uxy .
\]
The final equation 
forces $x,y \neq 0$.  This is done via a Gr{\"o}bner basis computation.
\item Identify the irreducible factor $I(C)$ of $P(C)$ which admits
  $c$ as a root.  To do this it is sufficient to determine the
  critical pair $(a,b)$, and thus $c$, to sufficient numerical precision.
\item Decide whether $c$ can be written as $\cos \theta$, with $\theta
  \in \pi \qs$. Equivalently, decide if the solutions of $2c=z+1/z$
  are roots of unity. 
To do this it is sufficient to 
{examine whether} the
  polynomial $R(z) := z^{\deg I} I(\frac{z+1/z}{2})$ has
  cyclotomic factors.

\end{enumerate}
The polynomials $R(z)$ which are constructed by running this algorithm on the
13\,110 step sets {in our collection} are all irreducible and have degree less
than 72. Thus, as the degree of the {$k$th} cyclotomic polynomial is \[
\phi(k) > \frac{k}{e^{\gamma}\log\log k + \frac{3}{\log \log k}},\] where
$\gamma \approx 0.577$ is Euler's constant~\cite[Thm.~8.8.7]{bach-shallit},
to prove that the excursion exponent is irrational it is sufficient to show
that $R({z})$ is not divisible by any of the first 349 cyclotomic polynomials;
{constructing cyclotomic polynomials is a routine task in computer
algebra~\cite{ArMo11}}.

After performing this filtering step we conclude that 12\,854 models have an
irrational excursion exponent, and thus an infinite orbit and a non-D-finite
\gf. They form the rightmost branch in Figure~\ref{classificationm2}.

%========================================================
\subsubsection{Detecting finite orbits}
%========================================================
We are thus left with 256 step sets, each of which having a rational exponent
$\alpha$. Among them we find 227 Hadamard models.
Proposition~\ref{prop:hadamard} tells us that they have a finite orbit, of
cardinality 6 or 9 depending on the sizes of the steps
{(Figure~\ref{fig:hadamardOrbits})}. For each of them the excursion exponent
is found to be $\alpha=-3$.

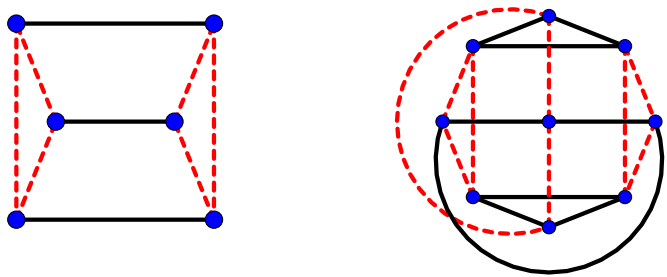
\begin{figure}[htb]
\centering
\hspace{0.9in}
\begin{minipage}{0.45\linewidth}
\begin{tikzpicture}[line cap=round,line join=round,>=triangle 45,x=1cm,y=1cm, scale=1.3]
\clip(-1.5,-1.5) rectangle (1.5,1.5);
\draw [line width=1.6pt,color=red, dashed] (-1,1)-- (-1,-1);
\draw [line width=1.6pt,color=red, dashed] (-1,-1)-- (-0.6,0);
\draw [line width=1.6pt,color=red, dashed] (-0.6,0)-- (-1,1);
\draw [line width=1.6pt] (-0.6,0)-- (0.6,0);
\draw [line width=1.6pt,color=red, dashed] (0.6,0)-- (1,1);
\draw [line width=1.6pt,color=red, dashed] (1,1)-- (1,-1);
\draw [line width=1.6pt,color=red, dashed] (1,-1)-- (0.6,0);
\draw [line width=1.6pt] (-1,-1)-- (1,-1);
\draw [line width=1.6pt] (1,1)-- (-1,1);
\begin{scriptsize}\draw [fill=blue] (-1,1) circle (2.5pt);
\draw [fill=blue] (1,1) circle (2.5pt);
\draw [fill=blue] (1,-1) circle (2.5pt);
\draw [fill=blue] (-1,-1) circle (2.5pt);
\draw [fill=blue] (-0.6,0) circle (2.5pt);
\draw [fill=blue] (0.6,0) circle (2.5pt);
\end{scriptsize}

\end{tikzpicture}
\end{minipage}
\hspace{-0.5in}
\begin{minipage}{0.45\linewidth}
\begin{tikzpicture}[line cap=round,line join=round,>=triangle 45,x=1cm,y=1cm]
\clip(-2.1,-2.1) rectangle (2.1,2.1);
\draw [line width=1.6pt] (-1,1)-- (0,1.4); 
\draw [line width=1.6pt] (0,1.4)-- (1,1);
\draw [line width=1.6pt] (1,1)-- (-1,1);
\draw [line width=1.6pt] (-1,-1)-- (0,-1.4);
\draw [line width=1.6pt] (0,-1.4)-- (1,-1);
\draw [line width=1.6pt] (1,-1)-- (-1,-1);
\draw [line width=1.6pt,color=red, dashed] (-1,-1)-- (-1.4,0);
\draw [line width=1.6pt,color=red, dashed] (-1.4,0)-- (-1,1);
\draw [line width=1.6pt,color=red, dashed] (-1,1)-- (-1,-1);
\draw [line width=1.6pt,color=red, dashed] (1,-1)-- (1,1);
\draw [line width=1.6pt,color=red, dashed] (1,1)-- (1.4,0);
\draw [line width=1.6pt,color=red, dashed] (1.4,0)-- (1,-1);
\draw [line width=1.6pt,color=red, dashed] (0,-1.4)-- (0,0);
\draw [line width=1.6pt] (-1.4,0)-- (0,0);
\draw [line width=1.6pt,color=red, dashed] (0,1.4)-- (0,0);
\draw [line width=1.6pt] (1.4,0)-- (0,0);
\draw [shift={(-0.51,0)},line width=1.6pt,color=red, dashed]  plot[domain=1.2214519287784171:5.061733378401169,variable=\t]({1*1.49*cos(\t r)+0*1.49*sin(\t r)},{0*1.49*cos(\t r)+1*1.49*sin(\t r)});
\draw [shift={(0,-0.51)},line width=1.6pt]  plot[domain=-3.4909370516062723:0.34934439801647943,variable=\t]({1*1.49*cos(\t r)+0*1.49*sin(\t r)},{0*1.49*cos(\t r)+1*1.49*sin(\t r)});
\begin{scriptsize}
\draw [fill=blue] (0,1.4) circle (2.5pt);
\draw [fill=blue] (0,-1.4) circle (2.5pt);
\draw [fill=blue] (1.4,0) circle (2.5pt);
\draw [fill=blue] (-1.4,0) circle (2.5pt);
\draw [fill=blue] (-1,1) circle (2.5pt);
\draw [fill=blue] (0,0) circle (2.5pt);
\draw [fill=blue] (1,1) circle (2.5pt);
\draw [fill=blue] (1,-1) circle (2.5pt);
\draw [fill=blue] (-1,-1) circle (2.5pt);
\end{scriptsize}\end{tikzpicture}
\end{minipage}

\caption{The possible orbits for two-dimensional Hadamard models with
  long steps in $\{-2,-1,0,1\}^2$, depending on whether there are
  steps with $-2$ in only one coordinate (left) or both coordinates
  (right). {The convention for dashed and solid edges is the
    same as in Figure~\ref{fig:orbit-quadrangulations}.}
}
\label{fig:hadamardOrbits}
\end{figure}

There are 29 models remaining.
We apply to them   the semi-algorithm of Section~\ref{sec:algo}, 
{which detects} 13 more  models with a finite orbit,  of cardinality
12 or 18. They are listed in Table~\ref{tab:finite}.
 Three distinct orbit structures arise, shown in
Figure~\ref{fig:orbits}.

\newcolumntype{C}{ >{\centering\arraybackslash} m{2cm} }
\begin{table}[htb]
  \centering
\begin{tabular}{c@{}C@{}cc|c@{}C@{}cc|c@{}C@{}cc}
$g$&steps & orbit &  $\alpha$&$g$& steps & orbit &  $\alpha$
& $g$&steps & orbit & $\alpha$  \\
&&&&&&&\\
1&$\diag{-21,-10,0-1,10}$&$O_{12}$ &$-4$ &
2&$\diag{-2-1,-1-2,01,10}$  &$\tilde O_{12}$  & $-5/2$ 
&1& $\diag{-2-1,-10,01,10}$ & $O_{12}$ & $-5/2$
\\
2&$\diag{-21,-1-1,01,1-1}$&$O_{12}$& $-4$ & 
3& $\diag{-20,-1-1,0-2,11}$  &$\tilde O_{12}$ &  $-5/2$
&2&$\diag{-2-1,-11,0-1,11}$ & $ O_{12}$  &  $-5/2$
\\
2&  $\diag{-21,-1-1,-10,01,1-1}$&$O_{12}$& $-4$ & 
2&$\diag{-2-1,-1-2,-1-1,01,10} $  &$\tilde O_{12}$  &  $-5/2$
&2&$\diag{-2-1,-10,-11,0-1,11}$ &  $ O_{12}$  & $-5/2$ 
\\
2&$\diag{-20,-21,1-1,10}$&$O_{18}$& -4 &
3& $ \diag{-20,-1-1,-10,0-2,0-1,11} $  &$\tilde O_{12}$  & $-5/2$
&2& $\diag{-2-1,-20,10,11}$ &  $ O_{18}$  & $-7/3$
\\
&&&& 4& $  \diag{-2-1,-20,-1-2,-1-1,0-2,01,10,11}$ &$\tilde O_{12}$  & $-5/2$
\\
\end{tabular}
\vskip 4mm
   \caption{The 13 non-Hadamard models with a finite orbit. Our method
   solves the ones on the left, proving that their \gf\ is D-finite
   (and transcendental). We conjecture that the 9 others are D-finite
   too, two of them being  possibly algebraic {(the second and
     third in the last column)}. {We also give the excursion exponent
     $\alpha$, and the genus $g$ of
     the curve $K(x,y)$, which is $0$ or $1$ for small step models.}
  }
    \label{tab:finite}
  \end{table}

\begin{figure}[htb]
\centering 
\begin{minipage}{0.28\linewidth}
\begin{tikzpicture}[line cap=round,line join=round,>=triangle 45,x=1.5cm,y=1.3cm, scale=0.8]
\clip(-1.6,-1.1) rectangle (1.6,2.1);
\draw [line width=1.5pt,dashed,color=red] (-0.5,1)-- (0.5,1);
\draw [line width=1.5pt,dashed,color=red] (0.5,1)-- (0,0);
\draw [line width=1.5pt,dashed,color=red] (0,0)-- (-0.5,1);
\draw [line width=1.5pt] (-0.5,1)-- (-1,1.5);
\draw [line width=1.5pt] (0.5,1)-- (1,1.5);
\draw [line width=1.5pt] (0,0)-- (0,-0.5);
\draw [line width=1.5pt,dashed,color=red] (-1,1.5)-- (-1,2);
\draw [line width=1.5pt,dashed,color=red] (-1,2)-- (-1.5,1);
\draw [line width=1.5pt,dashed,color=red] (-1.5,1)-- (-1,1.5);
\draw [line width=1.5pt,dashed,color=red] (1,1.5)-- (1,2);
\draw [line width=1.5pt,dashed,color=red] (1,2)-- (1.5,1);
\draw [line width=1.5pt,dashed,color=red] (1.5,1)-- (1,1.5);
\draw [line width=1.5pt,dashed,color=red] (0,-0.5)-- (-0.5,-1);
\draw [line width=1.5pt,dashed,color=red] (-0.5,-1)-- (0.5,-1);
\draw [line width=1.5pt,dashed,color=red] (0.5,-1)-- (0,-0.5);
\draw [line width=1.5pt] (-1.5,1)-- (-0.5,-1);
\draw [line width=1.5pt] (0.5,-1)-- (1.5,1);
\draw [line width=1.5pt] (1,2)-- (-1,2);
\begin{scriptsize}
\draw [fill=blue] (-0.5,1) circle (2.5pt);
\draw [fill=blue] (0.5,1) circle (2.5pt);
\draw [fill=blue] (0,0) circle (2.5pt);
\draw [fill=blue] (-1,1.5) circle (2.5pt);
\draw [fill=blue] (1,1.5) circle (2.5pt);
\draw [fill=blue] (0,-0.5) circle (2.5pt);
\draw [fill=blue] (-1,2) circle (2.5pt);
\draw [fill=blue] (-1.5,1) circle (2.5pt);
\draw [fill=blue] (1,2) circle (2.5pt);
\draw [fill=blue] (1.5,1) circle (2.5pt);
\draw [fill=blue] (-0.5,-1) circle (2.5pt);
\draw [fill=blue] (0.5,-1) circle (2.5pt);
\end{scriptsize}
\end{tikzpicture}
\end{minipage}
\qquad
\begin{minipage}{0.25\linewidth}
\begin{tikzpicture}[line cap=round,line join=round,>=triangle 45,x=0.65cm,y=0.65cm, scale=0.6]
\clip(-4.2,-4.2) rectangle (4.2,4.2);
\draw [line width=1.6pt] (-4,4)-- (-4,-4);
\draw [line width=1.6pt,color=red, dashed] (-4,-4)-- (4,-4);
\draw [line width=1.6pt] (4,-4)-- (4,4);
\draw [line width=1.6pt,color=red, dashed] (4,4)-- (-4,4);
\draw [line width=1.6pt,color=red, dashed] (-1,1)-- (-1,-1);
\draw [line width=1.6pt] (-1,-1)-- (1,-1);
\draw [line width=1.6pt,color=red, dashed] (1,-1)-- (1,1);
\draw [line width=1.6pt] (1,1)-- (-1,1);
\draw [line width=1.6pt] (-1,1)-- (0,3);
\draw [line width=1.6pt] (0,3)-- (1,1);
\draw [line width=1.6pt,color=red, dashed] (-1,1)-- (-3,0);
\draw [line width=1.6pt,color=red, dashed] (-3,0)-- (-1,-1);
\draw [line width=1.6pt,color=red, dashed] (1,1)-- (3,0);
\draw [line width=1.6pt,color=red, dashed] (3,0)-- (1,-1);
\draw [line width=1.6pt] (-1,-1)-- (0,-3);
\draw [line width=1.6pt] (0,-3)-- (1,-1);
\draw [line width=1.6pt] (-4,4)-- (-3,0);
\draw [line width=1.6pt] (-3,0)-- (-4,-4);
\draw [line width=1.6pt,color=red, dashed] (0,-3)-- (4,-4);
\draw [line width=1.6pt] (3,0)-- (4,-4);
\draw [line width=1.6pt] (3,0)-- (4,4);
\draw [line width=1.6pt,color=red, dashed] (0,3)-- (-4,4);
\draw [line width=1.6pt,color=red, dashed] (0,3)-- (4,4);
\draw [line width=1.6pt,color=red, dashed] (-4,-4)-- (0,-3);
\begin{scriptsize}\draw [fill=blue] (-4,4) circle (2.5pt);
\draw [fill=blue] (-4,-4) circle (2.5pt);
\draw [fill=blue] (4,-4) circle (2.5pt);
\draw [fill=blue] (4,4) circle (2.5pt);
\draw [fill=blue] (-1,1) circle (2.5pt);
\draw [fill=blue] (-1,-1) circle (2.5pt);
\draw [fill=blue] (1,-1) circle (2.5pt);
\draw [fill=blue] (1,1) circle (2.5pt);
\draw [fill=blue] (0,3) circle (2.5pt);
\draw [fill=blue] (-3,0) circle (2.5pt);
\draw [fill=blue] (3,0) circle (2.5pt);
\draw [fill=blue] (0,-3) circle (2.5pt);
\end{scriptsize}
\end{tikzpicture}
\end{minipage}
\qquad     
\begin{minipage}{0.25\linewidth}
\begin{tikzpicture}[line cap=round,line join=round,>=triangle 45,x=1cm,y=1cm, scale=0.6]
\clip(-3.2,-2.1) rectangle (2.7,3.8);

\draw [line width=1.6pt,color=red, dashed] (-1,-2)-- (0.5,-2);
\draw [line width=1.6pt] (0.5,-2)-- (1.799038105676658,-1.25);
\draw [line width=1.6pt,color=red, dashed] (1.799038105676658,-1.25)-- (2.549038105676658,0.049038105676657784);
\draw [line width=1.6pt] (2.549038105676658,0.049038105676657784)-- (2.549038105676658,1.5490381056766573);
\draw [line width=1.6pt,color=red, dashed] (2.549038105676658,1.5490381056766573)-- (1.799038105676659,2.848076211353315);
\draw [line width=1.6pt] (1.799038105676659,2.848076211353315)-- (0.5,3.598076211353316);
\draw [line width=1.6pt,color=red, dashed] (0.5,3.598076211353316)-- (-1,3.598076211353316);
\draw [line width=1.6pt] (-1,3.598076211353316)-- (-2.299038105676657,2.8480762113533173);
\draw [line width=1.6pt,color=red, dashed] (-2.299038105676657,2.8480762113533173)-- (-3.049038105676658,1.5490381056766584);
\draw [line width=1.6pt] (-3.049038105676658,1.5490381056766584)-- (-3.0490381056766584,0.049038105676660004);
\draw [line width=1.6pt,color=red, dashed] (-3.0490381056766584,0.049038105676660004)-- (-2.299038105676659,-1.25);
\draw [line width=1.6pt] (-2.299038105676659,-1.25)-- (-1,-2);
\draw [color=red, dashed] (-1,-2)-- (0.5,-2);
\draw [line width=1.6pt,color=red, dashed] (0.5,-2)-- (-0.25,-0.7009618943233418);
\draw [line width=1.6pt,color=red, dashed] (-0.25,-0.7009618943233418)-- (-1,-2);
\draw [color=red, dashed] (-3.0490381056766584,0.049038105676660004)-- (-2.299038105676659,-1.25);
\draw [line width=1.6pt,color=red, dashed] (-2.299038105676659,-1.25)-- (-1.5490381056766578,0.049038105676659116);
\draw [line width=1.6pt,color=red, dashed] (-1.5490381056766578,0.049038105676659116)-- (-3.0490381056766584,0.049038105676660004);
\draw [color=red, dashed] (-2.299038105676657,2.8480762113533173)-- (-3.049038105676658,1.5490381056766584);
\draw [line width=1.6pt,color=red, dashed] (-3.049038105676658,1.5490381056766584)-- (-1.5490381056766567,1.5490381056766578);
\draw [line width=1.6pt,color=red, dashed] (-1.5490381056766567,1.5490381056766578)-- (-2.299038105676657,2.8480762113533173);
\draw [color=red, dashed] (0.5,3.598076211353316)-- (-1,3.598076211353316);
\draw [line width=1.6pt,color=red, dashed] (-1,3.598076211353316)-- (-0.25,2.2990381056766584);
\draw [line width=1.6pt,color=red, dashed] (-0.25,2.2990381056766584)-- (0.5,3.598076211353316);
\draw [color=red, dashed] (2.549038105676658,1.5490381056766573)-- (1.799038105676659,2.848076211353315);
\draw [line width=1.6pt,color=red, dashed] (1.799038105676659,2.848076211353315)-- (1.0490381056766584,1.5490381056766582);
\draw [line width=1.6pt,color=red, dashed] (1.0490381056766584,1.5490381056766582)-- (2.549038105676658,1.5490381056766573);
\draw [color=red, dashed] (1.799038105676658,-1.25)-- (2.549038105676658,0.049038105676657784);
\draw [line width=1.6pt,color=red, dashed] (2.549038105676658,0.049038105676657784)-- (1.0490381056766582,0.04903810567665828);
\draw [line width=1.6pt,color=red, dashed] (1.0490381056766582,0.04903810567665828)-- (1.799038105676658,-1.25);
\draw [line width=1.6pt] (-1.5490381056766578,0.049038105676659116)-- (1.0490381056766584,1.5490381056766582);
\draw [line width=1.6pt] (-1.5490381056766567,1.5490381056766578)-- (1.0490381056766582,0.04903810567665828);
\draw [line width=1.6pt] (-0.25,-0.7009618943233418)-- (-0.25,2.2990381056766584);
\begin{scriptsize}\draw [fill=blue] (-1,-2) circle (2.5pt);
\draw [fill=blue] (0.5,-2) circle (2.5pt);
\draw [fill=blue] (1.799038105676658,-1.25) circle (2.5pt);
\draw [fill=blue] (2.549038105676658,0.049038105676657784) circle (2.5pt);
\draw [fill=blue] (2.549038105676658,1.5490381056766573) circle (2.5pt);
\draw [fill=blue] (1.799038105676659,2.848076211353315) circle (2.5pt);
\draw [fill=blue] (0.5,3.598076211353316) circle (2.5pt);
\draw [fill=blue] (-1,3.598076211353316) circle (2.5pt);
\draw [fill=blue] (-2.299038105676657,2.8480762113533173) circle (2.5pt);
\draw [fill=blue] (-3.049038105676658,1.5490381056766584) circle (2.5pt);
\draw [fill=blue] (-3.0490381056766584,0.049038105676660004) circle (2.5pt);
\draw [fill=blue] (-2.299038105676659,-1.25) circle (2.5pt);
\draw [fill=blue] (-0.25,-0.7009618943233418) circle (2.5pt);
\draw [fill=blue] (-1.5490381056766578,0.049038105676659116) circle (2.5pt);
\draw [fill=blue] (-1.5490381056766567,1.5490381056766578) circle (2.5pt);
\draw [fill=blue] (-0.25,2.2990381056766584) circle (2.5pt);
\draw [fill=blue] (1.0490381056766584,1.5490381056766582) circle (2.5pt);
\draw [fill=blue] (1.0490381056766582,0.04903810567665828) circle (2.5pt);
\end{scriptsize}\end{tikzpicture} 
\end{minipage}
\caption{The three finite orbit types which arise from non-Hadamard models in
  $\{-2,-1,0,1\}^2$. Two have cardinality 12, the third one {has}
  cardinality 18.  
We call these orbit structures,
{from} left to right, $O_{12}$, $\tilde O_{12}$ and $O_{18}$.}
\label{fig:orbits}
\end{figure}
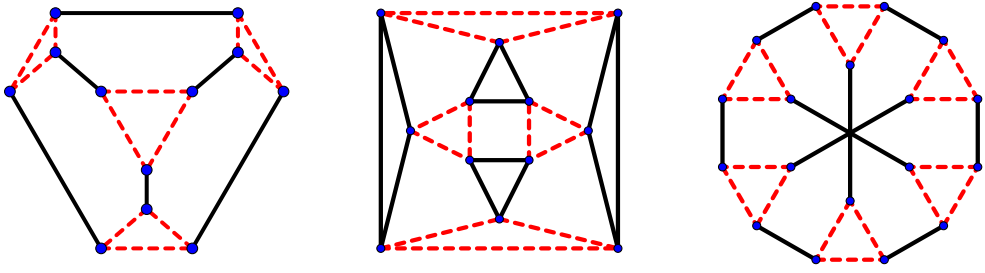

%=========================================================
\subsubsection{Sixteen models with a rational exponent but an infinite
  orbit}
\label{sec:embarassing}
%=========================================================

For each of the remaining 16 models, listed in Table~\ref{tab:embarassing}, we
ran our semi-algorithm by specializing $x=1$ and $y=2$ until we found at least
200 distinct orbit elements (the sum of the degrees of the polynomials in
$\PP$ --- or $\cQ$ --- gives a lower bound on the size of the orbit). We found
in each case minimal polynomials of degree over 100. The following proposition
explains why.

\begin{table}[htb]
\centering
\begin{tabular}{cCc|cCc|cCc}
&steps & $\alpha$ && steps & $ \alpha$ && steps &$ \alpha$ \\
&&&&&\\
\#1 &$\diag{-21, 01, 1-2}$ &   -5 & \#7&$\diag{-21, -11, 01, 1-1}$ &   -7&\#13& $\diag{-21, -10, 0-1, 01, 1-2, 10}$ &   -4\\
\#2&$\diag{-21, 1-2, 11}$ &   -4&\#8&$\diag{-2-1, -1-1, 0-1, 11}$ &   -11/5&\#14&$\diag{-21, -10, 01, 1-2, 1-1, 11}$ &   -4\\
\#3&$\diag{-21, 1-1, 10}$ &   -7&\#9& $\diag{-2-2, -20, 10, 11}$ &   -7/3 &\#15& $\diag{-2-1, -11, 01, 1-2, 10, 11}$ &-3\\
\#4&$\diag{-21, 1-1, 11}$ &   -5&\#10&$\diag{-2-2, -20, -1-1, -10, 10, 11}$& -7/3 &\#16&$\diag{-21, -10, -11, 0-1, 01, 1-2, 1-1, 10, 11}$ &   -4 \\
\#5&$\diag{-2-1, 1-1, 11}$&   -7/3&\#11&$\diag{-2-1, -1-1, 0-1, 01, 1-1, 11}$ &   -5/2\\
\#6 &$\diag{-2-1, 10, 11}$ &  -11/5&\#12& $\diag{-21, -11,  0-1, 01, 1-1, 11}$ &   -4\\
\end{tabular}
\vskip 4mm
\caption{Sixteen models with a rational excursion  exponent $\alpha$ and an infinite
  orbit.}
\label{tab:embarassing}
\end{table}

\begin{Proposition}
  The $16$ models of Table~\ref{tab:embarassing} have an infinite orbit.
\end{Proposition}
\begin{proof}
The proof is based on Proposition~\ref{prop:ab}, and mimics the  proof
used in the third example of Section~\ref{sec:ex}. For each model, we start from the
positive critical point $(a,b)$, define $\Phi$ and $\Psi$ as in
Section~\ref{sec:group}, and compute the expansion of
$\Theta:=\Psi\circ \Phi$ to cubic order.  There exists some
integer $m>0$ such that $\Theta^m$ is the identity at first order
(otherwise the excursion exponent would be irrational). Moreover, we
observe that  the
quadratic term in $\Theta^m$ vanishes, but  there
is a non-zero cubic term. This implies that all elements $\Theta^{km}$
are distinct, so that the orbit is infinite.

  We give below  the
 values of $a$, $b$, and $m$ 
(for model \#10 the value of $a$ is the positive root of $a^3=a+2$).
Since models \#5, \#6, \#8 and \#12 are obtained from another model in
the table by a reflection in the $x$-axis, we omit them (their
orbits are infinite by Proposition~\ref{prop:sym}).  However, the method works as
well for them (with $b$ replaced by $1/b$, and
$a$ and $m$ unchanged).
\newcommand\Tstrut{\rule{0pt}{3.0ex}}         
\newcommand\Bstrut{\rule[-1.9ex]{0pt}{2.0ex}}

\[\begin{array}{c||c|c|c|c|c|}
  \hbox{model} &  \#1 & \#2 & \#3 &\#4 &\#7 \\
  \hline
(a,b)&  ( 3^{1/2},  3^{1/2}/ 2^{1/3}) & (1,1) &(1,1) & (2^{-1/3},
                                                       3^{-1/2})& (1,
                                                                  3^{-1/2}) \Tstrut
\\
m &4 & 3 & 6 & 4&6 
\end{array}\]

\[\begin{array}{c||c|c|c|c|c|c|c|}
  \hbox{model}  &\#9& \# 10 & \#11  &\#13 &\#14 &\#15 &\#16 \\
  \hline
(a,b)& (2^{1/3},1)&   (a,1) & (1,\sqrt 2) &(\sqrt 2 , \sqrt 2)& (1,1) &
  (1,1) & (1,1)\Tstrut\\
m & 4& 4 & 3 & 3& 3& 2& 3
\end{array}\]
\end{proof}

For each model, we have tested D-finiteness experimentally, by generating
$10\, 000$ coefficients of the series $Q(0,0)$, and trying to guess from them
a linear recurrence relation for the coefficients or a {linear} differential
equation for the {generating function}. The guessing procedure is detailed in
Section~\ref{sec:interesting}. We could not find any recurrence nor
{differential} equation, and are tempted to believe that $Q(x,y;t)$ is not
D-finite for these 16 models. However, it must be noted that, for some models
of Section~\ref{sec:interesting}, it takes more than $10\,000$ coefficients to
guess a differential equation.

%=================================================
\subsection{Solving models with a finite orbit}
\label{sec:works}
%===================================================

As written above, 227 of the 240 models that have a finite orbit are Hadamard.
Our method applies systematically to them, as proved in
Section~\ref{sec:hadamard}. In particular, all these models have a D-finite
\gf, and, because they make small forward moves, $Q(x,y)$ is expressed as the
non-negative part of a simple rational function (see~\eqref{M=Mp=1}). {The
excursion exponent being~$-3$ in all cases, these series are
transcendental~\cite{flajolet-context-free}.}

We are left with the 13 models shown in Table~\ref{tab:finite}. For each of
them, there exists a unique section-free equation, {in agreement with
Conjecture~\ref{conj:section-free}} (up to a multiplicative factor, as usual).
For the 4 models shown in the first column, this equation defines $Q(x,y)$
uniquely and we are able to extract it as the positive part of a rational
series. In particular, these four series are D-finite (but transcendental,
{because of the exponent $-4$}). {Details are given below, and we} work out
detailed asymptotic behaviour of their coefficients in
Section~\ref{sec:asympt}. For the remaining 9 models, the right-hand side of
the section-free equation vanishes, so that this equation does not
characterize $Q(x,y)$ (for a start, any constant is a solution). These models
are the counterparts of the 4 algebraic models from the small step case, shown
in the second branch of Figure~\ref{fig:class-2D}. Clearly they deserve a
specific study, and we state conjectures regarding the nature of their \gfs\
in Section~\ref{sec:interesting}.

%============================================================
\subsubsection{Case $\boldsymbol{\cS=\{10, \bar 1 0, 0 \bar 1, \bar 2
    1\}}$}
\label{sec:solved1}
%============================================================
This is model F, which we have studied as one of our examples in this paper. Our main result is stated in Proposition~\ref{prop:F}:
\[
Q(x,y)= [x^{\ge} y^{\ge} ]{\frac { \left( {x}^{2}+1 \right) \left( x+y \right)  \left( y-x
    \right) 
 \left( {x}^{2}y-2\,x-y \right) 
  \left( {x}^{3}-x-2\,y \right) }{{x}^{7}{y}^{3} \left(
    1-t(x+\bx+\bx^2 y +\by) \right)}}.
\]
The excursion exponent ${\alpha\equiv}\alpha_e$, given by
Theorem~\ref{thm:exponent}, is $-4$. The exponent of {all
  quadrant walks} -- that is,
the exponent associated with the coefficients of $Q(1,1)$  -- can be
determined using multivariate singularity analysis, and is 
 found to be $\alpha_w=-4$. This is detailed for all four models shown on the
left of
Table~\ref{tab:finite} in Section~\ref{sec:asympt}.

For comparison with the next cases, we recall that the orbit, given by~\eqref{orbit:hard1},  has type $O_{12}$.

%============================================================
\subsubsection{Case $\boldsymbol{\cS=\{01, 1\bar 1 , \bar 1 \bar 1, \bar 2 1\}}$}
%============================================================
\begin{Proposition}\label{prop:12-a}
  For $\cS=\{01, 1\bar 1 , \bar 1 \bar 1, \bar 2 1\}$, we have
\[
Q(x,y)=[x^{\ge } y^{\ge }]{\frac { \left( {x}^{3}-2\,{y}^{2}-x \right)  
\left( {y}^{2}-x \right)  \left( {x}^{2}{y}^{2}-{y}^{2}-2\,x \right) }{{x}^{5}{y}^{4}
\left(1-t ( y+x\by +\bx\by +\bx^2 y)\right)}
},
\]
where the right-hand side is seen as a power series in $t$ with coefficients
in $\qs[x, \bx, y, \by]$.
The coefficients are hypergeometric: $q(i,j;n)$ is zero unless $n$ is of the form $n=2i+j+4m$, in which case
\[
q(i,j;n)= \frac{(i+1)(j+1)(i+j+2)n!(n+2)!}{m!(3m+2i+j+2)!(2m+i+1)!(2m+i+j+2)!}.
\]
The excursion exponent $\alpha_e$ and the walk exponent $\alpha_w$ are both $-4$.
\end{Proposition}
\begin{proof} 
The proof is very similar to the solution of Example F. The step polynomial is $S(x,y)= y+ x\by +\bx \by +\bx^2 y$. All elements of the orbit belong to the extension of $\qs(x,y)$ generated by $\sqrt{ (x+y^2)^2 +4x^3y^2}$. More precisely, denoting
\[
x_{1,2} = \frac{x+y^2 \pm \sqrt{ (x+y^2)^2 +4x^3y^2}}{2x^2},
\]
the orbit has type $O_{12}$ and consists of the following 12 pairs:
\begin{equation}\label{orbit-hard2}
\begin{array}{ccc}
  (x,y) & (x_1, y) & (x_2, y) \\
(x,x\by) & (-\bxun,x\by) & (-\bxde,x\by) \\
(x_1, x_1 \by) &(-\bx, x_1 \by) &(-\bxde, x_1 \by) \\
(x_2, x_2 \by) &(-\bx, x_2 \by) &(-\bxun, x_2 \by) .
\end{array}
\end{equation}
Note the similarities with the orbit~\eqref{orbit:hard1} obtained for model F. The functional equation reads:
\[
(1-tS(x,y)) Q(x,y)=1  -t x\by Q(x,0) -
t\bx\by (Q_{0,-}(y)+Q(x,0)-Q(0,0)) -t\bx^2 y \left( Q_{0,-}(y)+xQ_{1,-}(y)\right),
\]
where as before $x^iQ_{i,-}(y)$ is the \gf\ of walks ending at
abscissa $i$.
There is a unique section-free equation. To form it, the orbit
equation associated with $(x',y')$ must be weighted by $\pm
x'^2(x'_1-x'_2)$, where $(x'_1,y') \approx (x',y') \approx (x'_2, y')$
and $x_1'\not = x_2'$. More precisely, the weights associated with the 12 above orbit elements are:
\begin{equation}\label{weights-hard2}
\begin{array}{rrr}
  x^2(x_1-x_2) & x_1^2 (x_2-x) & -x_2^2(x_1-x) \\
x^2(\bxde -\bxun) & -\bxun^2( x+\bxde) & \bxde^2 (x+\bxun) \\
x_1^2(\bx-\bxde) &\bx^2(x_1+\bxde) &-\bxde^2(\bx+x_1) \\
-x_2^2(\bx-\bxun) &-\bx^2(x_2+\bxun) &\bxun^2(\bx+x_2) .
\end{array}
\end{equation}
Again, note the similarities with~\eqref{weights:hard1}.
We now divide the section-free equation by $x^2(x_1-x_2)$, so as to isolate $Q(x,y)$. This gives an equation similar to the one obtained with Example F (see~\eqref{eqA}):
\begin{equation}\label{eqA-bis}
Q(x,y)+\bxun \bxde Q(x,x\by) +A_1 +A_2 +A_3 +A_4 +A_5 = R(x,y),
\end{equation}
where $R(x,y)$ is the rational function occurring in Proposition~\ref{prop:12-a}, and each $A_i$ involves two of the series $Q(x',y')$ (again, as in Example F), chosen so that the expression of $A_i$ is symmetric in $x_1$ and $x_2$. More precisely, the orbit elements occurring in $A_1$ (resp. $A_2, A_3, A_4, A_4, A_5$) are $(x_1,y)$ { and } $(x_2,y)$ $\big($resp.   $(-\bxun, x\by)$  and $ (-\bxde, x\by)$,  $(x_1, x_1 \by)$  and  $(x_2, x_2\by)$,  $(-\bx, x_1\by)$  and $ (-\bx, x_2 \by)$, $(-\bxun, x_2 \by)$ and $(-\bxde, x_1 \by)\big)$. We now examine the symmetric functions of the $x_i$'s and of their reciprocals. They are Laurent polynomials in $x$ and $y$, which are, respectively, negative in $x$ and non-positive in $y$:
\[
x_1+x_2=\bx+ \bx^2y^2, \quad x_1 x_2 = -\bx y^2,
\]
while
\[
\bxun+\bxde= -\bx-\by^2 \quad \hbox{and} \quad \bxun\bxde= -x\by^2.
\]
With this, we conclude that the series $A_i$ also have coefficients in
$\qs[x,\by,y,\by]$,  that $A_1$, $A_3$ and $A_4$ are negative in $x$,
and that $A_2$ is negative in $y$. As in Example F, the case of $A_5$
is a bit trickier, due to the mixture of positive and negative powers
of the $x_i$'s. Following the same lines as in Example
F, one can prove that $A_5$ contains no monomial that would be
non-negative in $x$ and $y$.  The counterpart of Lemma~\ref{lem:extr} is that every
monomial $\bx^e y^f$ occurring in {the expression $E_a$ defined by~\eqref{eq:Ea}}, for the values of $x_1$ and $x_2$ here, satisfies $f\le 2e$. 

The simplicity of the coefficients $q(i,j;n)$ {comes}
 from the fact that
the expansion of $S(x,y)^n=(1+\bx^2)^n (y+x\by)^n$ in $x$ and $y$ has
simple coefficients.

The excursion exponent can be  determined using Theorem~\ref{thm:exponent}, but it is  more natural to start
from the explicit expression of $q(0,0;4m)$, for which we derive:
\[
q(0,0;4m)\sim \frac {4\sqrt 3}{27  \pi m^4} \left( \frac {16}
  3\right)^{3m}.
\]
 The asymptotic behaviour of the number of quadrant walks  is determined in Section~\ref{sec:asympt}.
\end{proof}

%============================================================
\subsubsection{Case $\boldsymbol{\cS=\{01, 1\bar 1 , \bar 1 \bar 1, \bar 2 1, \bar 1 0\}}$}
%============================================================
\begin{Proposition}  \label{prop:model3}
For $\cS=\{01, 1\bar 1 , \bar 1 \bar 1, \bar 2 1, \bar 1 0\}$, we have:
 \[
Q(x,y)=[x^{\ge } y^{\ge }]
{\frac { \left( {y}^{2}-x \right)  \left( {x}^{2}{y}^{2}-xy-{y}^{2}-2
\,x \right)  \left( {x}^{3}-xy-2\,{y}^{2}-x \right) }{ 
 {x}^{5}{y}^{4}\left(1-t(y+x \by +\bx \by + \bx^2y + \bx )\right)}},
\]
where the right-hand side is seen as a power series in $t$ with
coefficients in $\qs[x, \bx, y, \by]$.

The excursion exponent $\alpha_e$ and the walk exponent $\alpha_w$ are both $-4$.
\end{Proposition}
\begin{proof}
This example is very close to the previous one, from which it only differs by one West step. The orbit is still given by~\eqref{orbit-hard2}, but with different values of $x_1$ and $x_2$:
\[
x_{1,2}= \frac{x+xy+y^2\pm\sqrt{(x+xy+y^2)^2+4x^3y^2}}{2x^2}.
\] 
The symmetric functions of the $x_i$'s, and of their reciprocals, have
promising non-negativity properties 
(the same as in the previous example):
\[
x_1+x_2=\bx+\bx y +\bx^2y^2 , \quad x_1x_2=-\bx y^2 , \quad \bxun
+\bxde=-\bx-\by-\by^2  \quad \hbox{and}  \quad \bxun \bxde= -x\by^2.
\]
The functional equation differs from the previous one by the new term $-t\bx
Q_{0,-}(y)$ in the right-hand side. There is a unique section-free equation,
with weights again given by~\eqref{weights-hard2}. Thus this equation is
again~\eqref{eqA-bis}, with the same expressions of the series $A_i$. The
series $Q(x,y)$ is extracted in the same way as in the previous example. In
particular, only one series, $A_5$, raises difficulties in the extraction
procedure. They are solved as before by proving that $f\le 2e$ for every
monomial $\bx^e y^f$ occurring in {the expression $E_a$ defined
by~\eqref{eq:Ea}}.

The excursion exponent is determined from Theorem~\ref{thm:exponent}, and the
walk exponent in Section~\ref{sec:asympt}. 
\end{proof}

%============================================================
\subsubsection{Case $\boldsymbol{\cS=\{10,  1 \bar 1, \bar 2 1, \bar 2
    0\}}$}
\label{sec:solved4}
%============================================================

\begin{Proposition}\label{prop:hard-18}
For $\cS=\{10,  1 \bar 1, \bar 2 1, \bar 2 0\}$, we have:
\[
Q(x,y)=[x^{\ge } y^{\ge }]
\frac{(2-y)(x^3-y^2)(x^6y-3x^3y-x^3-y^2)(x^3-2y)}{x^9y^4(1-t(\bx^2+\bx^2y+x\by+x))},
\] 
where the right-hand side is seen as a power series in $t$ with
coefficients in $\qs[x, \bx, y, \by]$.
The coefficients are  hypergeometric: $q(i,j;n)$ is zero unless $n$ is of the form $n=i+3j+3m$, in which case
\[
q(i,j;n)=\frac{(i+1)(j+1)(i+3j+4)\big((i+2j+2)(i+2j+3)+m(2i+3j+4)\big) n!(n+3)!}{m!(m+j+1)!(2m+i+2j+3)!(2m+i+3j+4)!}.
\]
The excursion exponent $\alpha_e$ and the walk exponent $\alpha_w$ are both $-5$.
\end{Proposition}
\begin{proof}
The step polynomial is $S(x,y)=x+x\by+ \bx^2+\bx^2y$. All elements of the orbit belong to the extension of $Q(x,y)$ generated by $\sqrt{y(y+4x^3)} $ and $\sqrt{1+4y}$. More precisely, let us define
\begin{equation}\label{xu}
x_{1,2}= \frac{y\pm \sqrt{y(y+4x^3)}}{2x^2} \quad \hbox{and} \quad 
u_{3,4}= \frac{1\pm\sqrt{1+4y}}{2y}.
\end{equation}
Then the orbit consists of the following 18 pairs:
\[
\begin{array}{ccc}
(x,y) & (x_1,y) & (x_2, y) \\
(x, x^3\by) & (xu_3, x^3\by) & (xu_4, x^3\by) \\
(x_1, x_1^3\by) &(x_1u_3, x_1^3\by) &(x_1u_4, x_1^3\by) \\
(x_2, x_2^3\by) &(x_2u_3, x_2^3\by) &(x_2u_4, x_2^3\by) \\
(xu_3,u_3^3y) &(x_1u_3,u_3^3y) &(x_2u_3,u_3^3y) \\
(xu_4,u_4^3y)&(x_1u_4,u_4^3y)&(x_2u_4,u_4^3y)
\end{array}
\]
and its structure is shown in Figure~\ref{fig:orbit18}.
\begin{figure}[htb]
  \centering
  \scalebox{0.9}{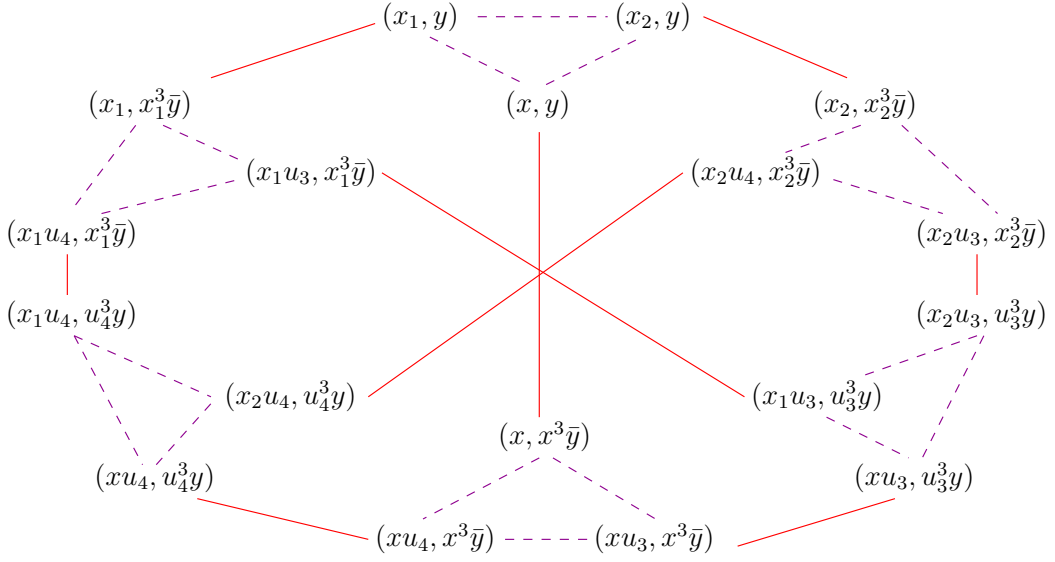}
  
  \caption{The orbit of $\cS=\{10,  1 \bar 1, \bar 2 1, \bar 2
    0\}$.  The
    values $x_i$ and $u_i$ are given by~\eqref{xu}. }
  \label{fig:orbit18}
\end{figure}

The functional equation reads
\[
(1-tS(x,y)) Q(x,y)=1 -tx\by
Q(x,0) -t\bx^2(1+y) \left( Q_{0,-}(y)+xQ_{1,-}(y)\right),
\]
and there is a unique section-free equation. To form it, the orbit
equation associated with $(x',y')$ must be weighted by $\pm x'y'
(x'_1-x'_2)(x'_3-x'_4)$, where
\[
\left. {\atopfix{(x'_1,y'')}{(x'_2,y'')}} \right\}
\approx (x',y') \approx (x',y'') \approx 
\left\{ {\atopfix{(x'_3,y'')}{(x'_4,y'')}} \right.
\]
and both $x'_1\not = x'_2$ and $x'_3 \not = x'_4$.
More precisely, the weights associated with the 18 above pairs are
\begin{equation}\label{weights-18}
\begin{array}{rrr}
x^2y(x_1-x_2)(u_3-u_4) & - x_1^2y(x-x_2)(u_3-u_4) & x_2^2 y (x-x_1)(u_3-u_4) \\
 -x^5\by(x_1-x_2)(u_3-u_4) & x^5\by u_3^2 (1-u_4)(x_1-x_2)& 
 -x^5\by u_4 ^2(1-u_3)(x_1-x_2) \\
 x_1^5\by(x-x_2)(u_3-u_4) & - x_1^5\by u_3^2(1-u_4)(x-x_2) &
 x_1^5\by u_4^2(1-u_3)(x-x_2)\\
-x_2^5\by (x-x_1)(u_3-u_4) & x_2^5\by u_3 ^2(x-x_1)(1-u_4) &
-x_2^5\by u_4^2( x-x_1)(1-u_3) \\
-x^2u_3^5y(1-u_4)(x_1-x_2) &x_1^2u_3^5y (x-x_2)(1-u_4) &-x_2^2u_3^5y(x-x_1) (1-u_4)\\
x^2u_4^5y (1-u_3)(x_1-x_2)&-x_1^2u_4^5y(1-u_3)(x-x_2)&x_2^2u_4^5y(1-u_3)(x-x_1).
\end{array}
\end{equation}
We now divide the section-free equation by $x^2y(x_1-x_2)(u_3-u_4)$,
so as to isolate $Q(x,y)$. This gives:
\begin{equation}\label{Qeq18}
Q(x,y)-x^3\by^2 Q(x,x^3\by)+ A_1+A_2+A_3+A_4+A_5+A_6
= R(x,y),
\end{equation}
where $R(x,y)$ is the rational function occurring in
Proposition~\ref{prop:hard-18} and each $A_i$ involves two or four
instances of the series $Q$, as described below:
\newcommand\Tstrut{\rule{0pt}{3.0ex}}         
\newcommand\Bstrut{\rule[-1.9ex]{0pt}{2.0ex}}
\[
\begin{array}{cccccc}
  A_1 & A_2 &A_3 & A_4 &A_5 &A_6 \\ \hline
(x_1,y)&(x_1, x_1^3\by)&(xu_3, x^3\by)&(xu_3,u_3^3y) &(x_1u_3, x_1^3\by)&(x_1u_3,u_3^3y)\Tstrut\\
(x_2,y)&(x_2, x_2^3\by)&(xu_4, x^3\by)&(xu_4,u_4^3y)&(x_2u_3, x_2^3\by)&(x_2u_3,u_3^3y)\\
&&&&(x_1u_4, x_1^3\by)&(x_1u_4,u_4^3y)\\
&&&&(x_2u_4, x_2^3\by)&(x_2u_4,u_4^3y)
\end{array}
\]
Then each $A_i$ is a series in $t$ whose coefficients are 
polynomials in $x, \bx, y, \by, x_1, x_2, u_3, u_4$. 
The symmetric functions of the $x_i$'s (resp. $u_i$'s) are Laurent
polynomials in $x$ and $y$, negative in $x$ (resp. in $y$):
\[
x_1+x_2= \bx^2y, \qquad x_1 x_2= -\bx y, \qquad u_3+u_4=\by, \qquad
u_3u_4=-\by.
\]
Hence each $A_i$ is a series in $t$ whose coefficients are Laurent
polynomials in $x$ and $y$. We now want to extract the non-negative
part, in $x$ and $y$, of~\eqref{Qeq18}. Clearly the second term
is $y$-negative. Then the above properties, and the
form~\eqref{weights-18} of the weights, imply that
\begin{itemize}
\item $A_1$, $A_2$, $A_5$, and $A_6$ are $x$-negative;
\item $A_3$ is $y$-negative.
\end{itemize}
There remains to examine
\[
A_4= \frac{-u_3^5(1-u_4)Q(xu_3, u_3^3y)+u_4^5(1-u_3)Q(xu_4,
  u_4^3y)}{u_3-u_4}.
\]
Since $Q(x,y)$ has polynomial coefficients in $x$ and $y$, it suffices
to prove that for any $i,j  \ge 0$, the expression
\[
\frac{-u_3^5(1-u_4)u_3^i u_3^{3j}y^j+u_4^5(1-u_3)u_4^i
  u_4^{3j}y^j}{u_3-u_4},
\]
which is a Laurent polynomial in $y$, is in fact $y$-negative. This is
readily checked, using the fact that $u_3u_4=-\by$ and
\[
E_a:=\frac {u_3^{a+1}- u_4^{a+1}}{u_3-u_4}
\]
is a polynomial in $\by$ of valuation $\lceil a/2\rceil$ (this is
proved by  induction on $a$, as $E_a=\by (E_{a-1}+E_{a-2})$).
The expression of $Q(x,y)$ follows by extracting the non-negative
part, in $x$ and $y$, of~\eqref{Qeq18}.

The simplicity of the coefficients $q(i,j;n)$ 
{comes} from the fact that
the expansion of $S(x,y)^n=(1+y)^n (x\by+x^2)^n$ in $x$ and $y$ has
simple coefficients.

The excursion exponent can be computed from Theorem~\ref{thm:exponent}, but it is  more natural to start
from the explicit expression of $q(0,0;3m)$, for which we derive:
\[
q(0,0;3m)\sim \frac {81}{32\, \pi m^5} \left( \frac {27}
  4\right)^{2m}.
\]
The asymptotic behaviour of the number of quadrant walks  is
determined in Section~\ref{sec:asympt}.
\end{proof}

%=====================================================
\subsection{Nine interesting models with a finite orbit}
\label{sec:interesting}
%===========================================================

For nine models, shown in the second and third {columns} of
Table~\ref{tab:finite}, an {interesting} phenomenon occurs: the orbit is
finite and the right-hand side of the unique section-free equation vanishes.
These models come in two types, depending on whether they have an
$x/y$-symmetry or not. {They} cannot be solved using the method of {this
paper}, and we explore them \emph{experimentally}.

\smallskip\noindent{\bf Questions.} For each of the nine models, we focus on
two important univariate specializations of $Q(x,y) = Q(x,y;t)$, namely the
generating function of excursions $Q(0,0) = \sum_n e_n t^n$ and the generating
function of {all quadrant walks} $Q(1,1) = \sum_n q_n t^n$. For these 18 power
series we address, as before, three types of questions: \emph{qualitative}
(are they algebraic? are they D-finite transcendental? are they
non-D-finite?), \emph{quantitative} (do they admit closed-form expressions?)
and \emph{asymptotic} (what is the growth of the sequences $(e_n)$ and
$(q_n)$?).

\smallskip\noindent{\bf Answers.} In this section, most answers to these
questions are \emph{conjectural}, although with a high degree of confidence.
They are obtained by performing computer calculations that take as input a
finite amount of information on $Q(0,0)$ and $Q(1,1)$, namely the first
terms\footnote{{Precisely, $20\, 000$ integer coefficients, and even $100\,
000$ coefficients modulo the prime $p=2147483647$. For this time- and
memory-consuming step, we have appealed to highly efficient implementations
due to Axel Bacher.}} of the sequences $(e_n)$ and~$(q_n)$. The main technique
that we use is \emph{automated guessing}, a classical tool in experimental
mathematics~{\cite{BoKa09}}. In principle, the guessing part could be
complemented by an \emph{automated proof} part, which would make the
(algebraicity/D-finiteness) results fully rigorous, as in~\cite{BoKa10} and
\cite[\S8]{BoBoKaMe16}. This would require, among other things, to consider
more general {series} such as $Q(x,0)$ and $Q(0,y)$. Given that the equations
conjectured for $Q(0,0)$ and $Q(1,1)$ are already {quite} big (see
Tables~\ref{tab:kreweras-models} and \ref{tab:gessel-models}), we have decided
to conduct the guessing part only.

\smallskip\noindent{\bf Approach.} For each model, we have first tried to
guess linear recurrence relations with coefficients in $\zs[n]$ satisfied by
the sequences $(e_n)$ and $(q_n)$, starting from the integer values of their
first terms. When the available terms were not enough to recognize such a
recurrence, we have used more terms modulo {the prime~$p=2147483647$}, and
tried to recover recurrences with coefficients in $\zs/p\zs[n]$. In both
cases, we used the guessed recurrence relations to produce even more terms, on
which we repeated guessing procedures in order to get (hopefully)
minimal-order linear differential equations with polynomial coefficients
(in~$\zs[t]$, resp. in $\zs/p\zs[t]$) {for the associated series}. {On the one
hand, such minimal-order equations are hard to guess because they tend to have
many apparent singularities and {thus} coefficients of very large degrees};
sometimes, it is necessary to produce them indirectly, e.g., by taking (right)
gcd's of equations with higher orders but smaller degrees. On the other hand,
{they} are interesting because they contain a lot of information on their
solutions. For instance, minimal-order differential equations with
coefficients in $\zs[t]$ are helpful in proving transcendence of their
solutions. {This is detailed below in Section~\ref{sec:Kreweras}.}

Even when one can only guess differential equations with coefficients in
$\zs/p\zs [t]$, for {a sufficiently large prime such as~$p=2147483647$},
rational reconstruction allows one to predict {the small factors of the
leading coefficients of plausible differential operators over $\qs[t]$, and
thus} the growth constant in the asymptotics of $(e_n)$ and $(q_n)$. A similar
procedure applied to recurrences instead of differential equations allows one
to guess the critical exponents of these sequences. They {can} also give, via
$p$-curvature computations~\cite{BoCaSc15,BoCaSc16}, some insight on the
algebraic/transcendental nature of the power series in $\zs[[t]]$ (modulo
classical conjectures in the arithmetic theory of G-operators~\cite{Andre04}).
{Examples are provided in~\cite{BoKa09,BoKa10} and \cite[\S2.3.3]{Bostan2017}.
However, given the size of our conjectured equations, and especially of the
prime number~$p$, we have not applied these algorithms here.}

We refer to~\cite{BoKa09,Bostan2017} for more details on guessing techniques,
and now describe the results that we have obtained on the nine models.

\medskip
%==============================================================
\subsubsection{Five  models  of the Kreweras type.} 
\label{sec:Kreweras}
%===================================================
These models are symmetric in the first diagonal, and are shown in the central
column of Table~\ref{tab:finite} {and in Table~\ref{tab:kreweras-models}
below}. Their orbits are all of the same form: they consist of all pairs
$(x_i,x_j)$, with $0\le i \not = j \le 3$, where $x_3=y$ and $x_0=x$, and
$x_1$ and $x_2$ are the three roots of the equation $S(X,y)=S(x,y)$. In
particular, for a pair $(x',y')$ in the orbit, the symmetric pair $(y',x')$
{also lies} in the orbit. The orbit structure is $\tilde O_{12}$, as shown in
Figure~\ref{fig:orbits}.

Theorem~\ref{thm:exponent} gives for each model the growth constant $\mu$ of
excursions and the associated exponent ${\alpha\equiv}\alpha_e$, which happens
to be $-5/2$ in all cases. The first two models are not strongly aperiodic,
but it appears (numerically) that an asymptotic estimate $e_n \sim \kappa\,
\mu^n n^{-5/2}$ holds in all cases (provided $n$ is a multiple of 4 in the
first case, and of 2 in the second case). The growth constant of the total
number $q_n$ of quadrant walks of length $n$ can be determined using the
results of~\cite{garbit-raschel,JoMiYe18}: in all five cases, it coincides
with the excursion constant $\mu$. Observe that the \emm drift, $(S_x(1,1),
S_y(1,1))$ is always negative. When the model is, in addition, aperiodic (last
three models), we can apply the result of~\cite[Ex.~7]{duraj}: there exists a
constant $K$ such that
\[
q_n \sim K \mu^n n^{-5/2}.
\]
{Numerical computations (of two different types: floating point and
modulo~$p$)} suggest that this also holds for the first two (periodic) models,
with a constant~$K$ that depends on $n\!\mod 4$ (first model) and on $n\!\mod
2$ (second model). Such periodicity phenomena will be established in
Section~\ref{sec:asympt} for the four solved models of Section~\ref{sec:works}
(see for instance~\eqref{eq:asm1}).

\medskip

\begin{table}[htb]
\begin{tabular}{|@{}C@{}c|cccc|cccc|}
  model & $m$ &  $e_n$ &  $Q(0,0)$  &alg.& $\alpha_e$ &$q_n$ &  $Q(1,1)$  &alg.& $\alpha_w$ 
  \\
  \hline
  $\diag{-2-1,-1-2,01,10}$ & 4 & $[2,12]$
  & \begin{tabular}[t]{c}$[8, 13]$\\
    irred. \end{tabular}& no & $-5/2$ & $[32,76]$& \begin{tabular}[t]{c}$[17,296]$
      \\ red. min.
    \end{tabular}& no  &
  $-5/2$ ?
  \\
  $\diag{-20,-1-1,0-2,11}$ & 2&  $[4,5]$ & \begin{tabular}[t]{c}$[5,8]$
      \\ red. min.
    \end{tabular}& no & $-5/2$ & $[19,59]$& $ \begin{tabular}[t]{c}$[9,83]$
      \\ red. min.
    \end{tabular} $ & no & $-5/2$ ?
  \\
  $\diag{-2-1,-1-2,-1-1,01,10} $& 1 &$[12,37]$ & \begin{tabular}[t]{c}$[9,52]$
      \\ red. min.
    \end{tabular}& no & $-5/2$ & $[33,266]$& \begin{tabular}[t]{c}$[17,309]$ \\ red. min.
    \end{tabular}& no & $-5/2$
  \\
  $ \diag{-20,-1-1,-10,0-2,0-1,11} $& 1& $[20,75]$ & \begin{tabular}[t]{c}$[13,94]$
      \\ red. min.
    \end{tabular}& no  & $-5/2$ & $[60,118]^\star$& $[25,663]^\star$ & ? & $-5/2$
  \\
  $  \diag{-2-1,-20,-1-2,-1-1,0-2,01,10,11}$& 1 &$[36,520]^\star$ & \begin{tabular}[t]{c}$[26,573]^\star$
      \\
  \end{tabular}& ? & $-5/2$ & $[99,204]^\star$& $[44,652]^\star$ & ? & $-5/2$  
\end{tabular}
\vskip 4mm
   \caption{The five Kreweras-like models, with their periods~$m$. For each
     sequence $(e_{mn})$ and $(q_n)$, 
	resp. for the associated series $Q(0,0;t^{1/m})$ and $Q(1,1;t)$, a pair
	$[r,d]$ indicates
	the order~$r$ and the coefficients degree~$d$ of a (conjectural)
     recurrence relation, resp. of a differential equation.
     A star indicates that we have {only guessed} recurrences or
     differential equations modulo $p=2147483647$.  
    }
    \label{tab:kreweras-models}
  \end{table}

Algorithmic guessing has succeeded for all 10 sequences in
Table~\ref{tab:kreweras-models}, {but only modulo {$p=2147483647$} for 3 of
them}.
We are extremely confident that the guessed recurrences and
differential operators are correct. In particular, they pass with success the
filters described in~\cite[Sec.~2.4]{BoKa09}. For instance, the leading
coefficients of the differential operators that (conjecturally) annihilate
$Q(0,0)$ and $Q(1,1)$, or their rational reconstruction when operators are
available modulo~$p$ only, vanish at $t=1/\mu$. Also, the occurrence of $3/2$
among the local exponents of the operators around $t=1/\mu$ is in agreement
with the exponents $\alpha_e=\alpha_w=-5/2$.

Assuming these recurrences and equations correct, we can use them to derive
some properties of the sequences $(e_n)$ and $(q_n)$. For instance, guessing
already strongly indicates that there is no hypergeometric sequence among the
10 sequences. In cases where recurrences are guessed over the integers (not
only modulo~$p$), we have applied Petkov\v sek's algorithm~\cite{Petkovsek92}
to them, and obtained a proof that these sequences are indeed not
hypergeometric.

Guessing also strongly indicates that there is no algebraic generating
function for any of the 10 sequences. In cases where differential equations
are guessed over the integers (not only modulo~$p$), we have a proof for this
fact, based on the following strategy. Linear differential operators can be
factored algorithmically~\cite{Hoeij97}. Those that are irreducible in
$\qs(t)\langle \partial_t \rangle$ are necessarily minimal. We have proved
minimality of the others using the argument of~\cite[Prop.~8.4]{BoBoKaMe16}.
Next, we computed the first terms of a local basis of solutions at $t=0$. At
least one basis element contains logarithms, which, combined with minimality,
implies that the solution is transcendental~\cite[\S2]{CoSiTrUl02}. Note that
this cannot be directly deduced from estimates of the form $c\, \mu^n
n^{-5/2}$, which \emm are, compatible with
algebraicity~\cite{flajolet-context-free}. For the excursions of the second
model, we were even able to solve the differential equation, thus obtaining a
conjectural closed form expression of $Q(0,0;\sqrt t)$
(Conjecture~\ref{conj:K2-ex}). {When we have only guessed differential
equations modulo {$p=2147483647$}, we still conjecture that the corresponding
operators have minimal order.}

\medskip

We now review briefly the five Kreweras-like models and add a few
details completing Table~\ref{tab:kreweras-models}.

\medskip
\noindent\paragraph{$\bullet$ \bf Case $\boldsymbol{\cK_1=\{ \bar 2 \bar 1, \bar 1 
\bar 2, 0 1, 1 0 \}}$}

The excursion generating function $Q(0,0) = \sum_n e_n t^n$  starts
\[Q(0,0) = 1+6\, \, t^4+236\, t^8+14988\, t^{12}+ 1193748\, t^{16} + O(t^{20})\]
and the walk  generating function $Q(1,1) = \sum_n q_n t^n$ starts
\[Q(1,1) = 1+2\, t+4\, t^2+8\, t^3+22\, t^4+64\, t^5+178\, t^6+O(t^7).\]
The growth constant is $\mu= 8/(3^{3/4})$ for both sequences.

The model has period $m=4$, and $e_n=0$ if $n$ is not a multiple of
$4$. For  the subsequence $(u_n) = (e_{4n})$ we have guessed  that
\begin{multline*}
	(4608\, n^4+37504\, n^3+114144\, n^2+153992 \, n+77715)\times \\
	(2\, n+3)\, (2\, n+1)\, (4\, n+5)\, (4\, n+1)\,  (n+1)^2\, (4\, n+3)^2\, u_{n} - \\	
\left( 62208\, n^{12}+1159488\, n^{11}+9826272\, n^{10}+50056248\, n^9+\frac{341349339}{2}\, n^8+410259762\, n^7 + \right. \\
\frac{22807094283}{32}\, n^6 +\frac{28845939249}{32}\, n^5+\frac{421694744175}{512}\, n^4
+\frac{1085550761145}{2048}\, n^3+ \\
\frac{1868027110233}{8192}\, 
\left. n^2+\frac{1929023165205}{32768}\, n+\frac{1807811742825}{262144} \right)
\, u_{n+1}
+ \\	(n+3)\, (n+2)\, (2\, n+5)^2\, (6\, n+13)^2\, (6\, n+11)^2\, \times \hfill \\
	\left(\frac{81}{2048}\, n^4+ \frac{1341}{8192}\, n^3+\frac{8235}{32768}\, n^2+\frac{22257}{131072}\, n+\frac{44739}{1048576} \right)\ 
	 u_{n+2} 
= 0.
\end{multline*}
 The leading coefficient of the minimal differential operator 
$L_e$ annihilating $Q(0,0;t^{1/4})$ is 
\[t^7\, (27-4096\, t)^2\, \left(t^4-\frac{47}{640} \,
t^3-\frac{374489}{125829120} \, t^2-\frac{23644531}{2319282339840} \,
t+\frac{29645}{281474976710656} \right),
\] where the factor $27-4096\, t$ vanishes when  $t=27/4096=1/\mu^4$.
 
For walks ending anywhere in the quadrant, the leading coefficient of the operator $L_w$ is
\[t^{13} \, (4t-1) \, (16 \, t^3+8 \, t^2+11 \, t-4)^4 \, (4096 \, t^4-27)^4 
\, \times \left( \text{irreducible poly. of degree 254} \right),\]
which is again compatible with the value of $\mu$.

%========================================================
\medskip
\noindent\paragraph{$\bullet$ \bf Case $\boldsymbol{\cK_2=\{ \bar 2 0, \bar 1 \bar 
1, 0 \bar 2, 1 1 \}}$}

The excursion generating function $Q(0,0) = \sum_n e_n t^n$ starts \[Q(0,0) =
1+ t^2+4\, t^4+21\, t^6+138\, t^8+1012\, t^{10} +8064\, t^{12}+O(t^{14}) \]
while \[ Q(1,1) = 1+ t+2\, t^2+5\, t^3+12\, t^4+32\, t^5+86\, t^6+O(t^7).\]
The growth constant is $\mu=2\sqrt{3}$ for both sequences.

The model has period $m=2$, and $e_n=0$ if $n$ is odd. For  the nontrivial
subsequence $(u_n) = (e_{2n})$ we have guessed that
\begin{multline*}
	(4\, n+9)\, (n+5)^2\, (n+4)^2\, u_{n+4} - 
	4\, (n+2)\, (16\, n^2+100\, n+153)\, (n+4)^2\, u_{n+3} - \\
	4\, (32\, n^5+584\, n^4+4096\, n^3+13909\, n^2+22947\, n+14742)\, u_{n+2} + \\
	96\, (2\, n+3)\, (n+2)\, (16\, n^3+108\, n^2+239\, n+183)\, u_{n+1} + \\
	(9216\, n^5+76032\, n^4+230400\, n^3+319680\, n^2+201024\, n+44928)\, u_n=0.
\end{multline*}
The differential operator $L_e$ found for  $ Q(0,0; t^{1/2}) = \sum_n
e_{2n}  t^{n}$   has  leading coefficient  \[ t^3 \, (1 + 4 \, t)^2 \, (1 - 12 \, t)^3 \] where the factor $(1-12\, t)$ is
compatible with the value of $\mu$. Furthermore, $L_e$ is reducible in $\qs(t)\langle \partial_t
\rangle$; one can write $L= L_2^{(1)} L_2^{(2)} L_1$, where $L_1$ has order 1 and
$L_2^{(1)}$ and 
$L_2^{(2)}$ have order~2. More importantly, $L$ can be written as the
{least common left multiple} of the three following operators:
\begin{multline*}
\partial_t+\frac{1}{t}, \qquad
\partial_t^2+ \frac{120\, t^2+2\, t-3}{(-1+12\, t)t(4\, t+1)}\, \partial_t+ \frac{288\, t^3-48\, t^2+14\, t+1}{(4\, t+1) t^2 (-1+12\, t)^2}, \\
\partial_t^2+ \frac{120\, t^2+2\, t-3}{(-1+12\, t)t(4\, 
t+1)}\, \partial_t+ \frac{24\, t^2-8\, t-1}{t^2(4\, t+1)(-1+12\, t)}.
\end{multline*}
The use of $_2F_1$ solving algorithms~\cite{BoChHoPe11,ImHo17,BoChHoKaPe17}
leads us to the following conjectural expression. 
 \begin{Conjecture}\label{conj:K2-ex}
For the model 
${\cS=\{ \bar 2 0, \bar 1 \bar 1, 0 \bar 2, 1 1 \}}$,
the excursion generating function $Q(0,0; t^{1/2})$
is equal to
\[
\frac{1}{3t}
-
\frac{\sqrt{1-12t}}{6t}
\left(
\twoFone{\frac16}{\frac13}{1}{\frac{-108\,t(1+4t)^2}{(1-12t)^3}}
+
\twoFone{-\frac16}{\frac23}{1}{\frac{-108\,t(1+4t)^2}{(1-12t)^3}}
\right).
\]
\end{Conjecture}
\noindent{\bf Remark.} The first hypergeometric term above can be
rewritten with a simpler argument, as
\[
\frac 1 {\sqrt{1-12t}}\ \twoFone{\frac16}{\frac13}{1}{\frac{-108\,t(1+4t)^2}{(1-12t)^3}}=
\twoFone{\frac16}{\frac13}{1}{108t^2(1+4t)}.
\]
Moreover, the square of this power series is known to count excursions of
the face centered cubic lattice~{\cite[Appendix~A]{BoBoChHaMa13}, see also~\cite[\S4]{Joyce01}.}
This is entry A002899 in the on-line
encyclopedia of integer sequences~\cite{oeis}.
 {The guessed operator $L_e$ is the minimal-order operator canceling the
    conjectured series.}
 The leading coefficient of the operator $L_w$  contains the factor
\[
t^5\,(4\,t-1)\,(4\,t^2+1)^2\,(16\,t^3+8\,t^2+11\,t-4)^2\,(12\,t^2-1)^4
,\]
which is compatible with $\mu=2\sqrt{3}$.

\medskip
\noindent\paragraph{$\bullet$ \bf Case $\boldsymbol{\cK_3=\{ 
\bar 2 \bar 1, \bar 1 \bar 2, \bar 1 \bar 1, 0 1, 1 0 \}}$}	

  The excursion generating function   starts
\[Q(0,0) = 1+2\, t^3+6\, t^4+16\, t^6+ 122\, t^7+236\, t^8+ O(t^9)\]
while
\[Q(1,1) = 1+2\, t+4\, t^2+10\, t^3+32\, t^4+98\, t^5+292\, t^6+O(t^7).\]
The model is strongly aperiodic, with growth constant $\mu \sim 4.03$
for both sequences, where $\mu$ is the  unique positive root
of $4069 +768\, u-6\, u^2+u^3-27u^4$. 

The leading coefficient of
$L_e$ is 
\[ t^8 \, (1 + t^2)^2 \, (4069 \, t^4+768\, t^3-6\, t^2+t-27)^2 
 \times \left( \text{irreducible poly. of degree 32} \right),\]
and vanishes at $t=1/\mu$. 
Similarly, the leading coefficient of $L_w$  is
\[
t^{13} \,  (1- 5\, t) \,  (t^2+1)^2 \, 
(4069 \, t^4+768\, t^3-6\, t^2+ t-27)^4 \,  
(23\, t^3 + 32\, t^2 + 8\, t -4)^4 \,
 \times \left( \text{irreducible poly. of degree 263} \right).\]

\medskip
\noindent\paragraph{$\bullet$ \bf Case $\boldsymbol{\cK_4=\{
\bar 2 0, \bar 1\bar 1, \bar 1 0, 0 \bar 2, 0\bar 1, 11 \}}$}	

The excursion generating function   starts
\[Q(0,0) = 1+ t^2+2\, t^3+4 \, t^4+24\, t^5+37\, t^6+276\, t^7+O(t^8)\]
while
\[Q(1,1) = 1+t+4\, t^2+11\, t^3+42\, t^4+148\, t^5+576\, t^6+O(t^7).\]
The model is strongly aperiodic, with growth constant $\mu \sim 4.91$
for both sequences, where  $\mu $ is the largest  positive root of
$ 405-108\, u-72\, u^2+u^3+3u^4$.

 The leading coefficient of
$L_e$ is 
\[ t^8 \, (65\, t^2+8\, t+16)^2 \, (405\, t^4-108\, t^3-72\, t^2+t+3)^4 
 \times \left( \text{irreducible poly. of degree 66} \right),\]
and vanishes at  $t=1/\mu$.
Similarly, the leading coefficient of $L_w$  contains the factor
\[
t^{17} \,
(1-6t) \, 
(405\, t^4-108\, t^3-72\, t^2+t+3)^8 \, 
(3\, t^3+4\, t^2+20\, t-4)^6 \, 
(65\, t^2+8\, t+16)^2
.\]

\medskip
\noindent\paragraph{$\bullet$ \bf Case $\boldsymbol{\cK_5=\{
\bar 2\bar 1, \bar 2 0, \bar 1\bar 2, \bar 1\bar 1, 0\bar 2,
01, 10, 11 \}}$}	

The excursion generating function starts
\[Q(0,0) = 1+ t^2+8\, t^3+10\, t^4+106\, t^5+467\, t^6+1850\, t^7+O(t^8)\]
while
\[Q(1,1) =1+3\, t+10\, t^2+51\, t^3+260\, t^4+1350\, t^5+7568\, t^6+O(t^7).\]
The model is strongly aperiodic, with growth constant $\mu=
2\,\sqrt{3} +8/3^{3/4}\approx 6.97$ for both sequences. The
value $\mu   $ is
the unique positive  root of $208+4608 \,u+648\, u^2-27\,u^4$.

In this case we only have conjectures modulo $p=2147483647$ for both $e_n$
  and $q_n$.
For excursions, rational reconstruction shows that the leading coefficient of the operator $L_e$ contains the factor
\[
t^{23} \,
(4\, t^2 + 1)^4 \,
(208\, t^4+4608\, t^3+648\, t^2-27)^7
\]
which vanishes at $t=1/\mu$.
{Similarly, the  leading coefficient of $L_w$} contains the factor
\[
t^{35} \,
(8\, t - 1) \,
(4\, t^2+1)^4 \,
(64\, t^3+16\, t^2+11\, t-2)^{10} \,
(208\, t^4+4608\, t^3+648\, t^2-27)^{12} \,
\]
which vanishes again at $t=1/\mu$.

%===============================================================
\subsubsection{Four models  of the Gessel type.}
%=======================================================

The remaining 4 models, shown on the right of Table~\ref{tab:finite} {and in
Table~\ref{tab:gessel-models}}, do not have a symmetry property. They are
obtained from the models shown on the left of {Table~\ref{tab:finite}} (solved
in Section~\ref{sec:works}) by a reflection in a horizontal line. By
Proposition~\ref{prop:sym}, their orbit type is $O_{12}$ for the first three,
and $O_{18}$ for the last one. More precisely, in the first three cases the
orbit consists of
\[
\begin{array}{ccc}
(x, y) &(x_1 , y)  & (x_2 , y) \\
(x, \bx^e\by)  &(-\bxun , \bx^e\by)  &(-\bxde  , \bx^e\by)  \\
(x_1 , \bx_1 ^e\by)  &(-\bx, \bx_1 ^e\by) &(-\bxde  , \bx_1 ^e\by) \\
(x_2 , \bx_2 ^e\by)  &(-\bx, \bx_2 ^e\by) &(-\bxun , \bx_2 ^e\by) 
\end{array}
\]
where $x_1, x_2$ are the two solutions of $S(X,y)=S(x,y)$ (different
from $x$), $e=2$ for the first model and $e=1$ for the next two. In the
fourth case, the orbit consists of
\[
\begin{array}{ccc}
(x,y) & (x_1,y) & (x_2, y) \\
(x, \bx ^3\by) & (xu_3, \bx ^3\by) & (xu_4, \bx ^3\by) \\
(x_1, \bx _1^3\by) &(x_1u_3, \bx _1^3\by) &(x_1u_4, \bx _1^3\by) \\
(x_2, \bx _2^3\by) &(x_2u_3, \bx _2^3\by) &(x_2u_4, \bx _2^3\by) \\
(xu_3,\bu _3^3y) &(x_1u_3,\bu _3^3y) &(x_2u_3,\bu _3^3y) \\
(xu_4,\bu _4^3y)&(x_1u_4,\bu _4^3y)&(x_2u_4,\bu _4^3y)
\end{array}
\]
where
\[
x_{1,2}= \frac{1\pm\sqrt{1+4x^3y}}{2x^2y} \qquad \hbox{and} \qquad
u_{3,4}= \frac{y \pm\sqrt{y^2+4y}}{2}.
\]
For each of these four models, there exists a unique section-free
equation, and its right-hand side vanishes.

Theorem~\ref{thm:exponent} gives for each model the excursion constant $\mu$
and the corresponding exponent, which is $\alpha_e=-5/2$ for the first three
models, and $\alpha_e=-7/3$ for the last one. Only the third model is strongly
aperiodic, {the other models having respectively period $m=2$
  (first model), $m=4$ (second model) and $m=3$ (last model)}. But it appears {numerically} that an asymptotic estimate
$q(0,0;n)\sim \kappa\, \mu^n n^{\alpha_e}$ holds in all cases (provided $n$ is
a multiple of $m$ in the periodic cases).
The growth constant $\bar \mu$ of the sequence $(q_n)$ can be
determined using the results of~\cite{garbit-raschel,JoMiYe18}: in all four
cases, it is larger than the excursion constant~$\mu$. Observe that the \emm
drift, $(S_x(1,1), S_y(1,1))$ is always of the form $(-\delta, 0)$ with
$\delta$ positive. The second component being $0$, we cannot apply the result
of~\cite[Ex.~7]{duraj}, and indeed, the walk exponent $\alpha_w$, {that we
conjecture numerically}, turns out to differ from $\alpha_e$. In fact, we
believe that for each of the four models,
\[
q_n \sim K \bar \mu^n n^{-3/2}.
\]
with a constant $K$ that depends on $n\!\mod m$ {in the periodic cases.}

\medskip

\noindent{\bf What we have done.} We have applied to these four models the
same guessing procedures as for the Kreweras-like models. Remarkably, we
discovered two {possibly} algebraic models among them. More precisely, for the
second and third models, the series $Q(0,0)$ seems to be algebraic of
degree~$32$. But it must be noted that in contrast with Kreweras-like models,
for three of the four models we could not guess any recurrence for the
sequence $(q_n)$, even modulo the prime $p=2147483647$.

\begin{table}[htb]
\begin{tabular}{|@{}C@{}c|c@{}c@{}cc|cccc|}
  model & $m$ &  $e_n$ &  $Q(0,0)$  &alg.& $\alpha_e$ &$q_n$ &  $Q(1,1)$  &alg.& $\alpha_w$ 
  \\
  \hline
 $\diag{-2-1,-10,01,10}$     & 2 & $[8,5]$ & \begin{tabular}[t]{c}$[9,18]$
      \\ red. min.
  \end{tabular}& no & $-5/2$ & $[46,176]$& \begin{tabular}[t]{c}$[17,400]$
      \\ red. min.
    \end{tabular}& no & $-3/2$ ?
  \\ 
 $\diag{-2-1,-11,0-1,11}$     & 4  &$[2,12]$ & \begin{tabular}[t]{c}$[8,13]$
      \\ irred.
  \end{tabular}& $[32,14]$ & $-5/2$ & $?$& $?$ & ? & $-3/2$ ?
  \\ 
  $\diag{-2-1,-10,-11,0-1,11}$    & 1  &$[12,37]$ & \begin{tabular}[t]{c}$[9,52]$
      \\ irred.
  \end{tabular}& $[32,57]$ & $-5/2$ & $?$& $?$ &?  & $-3/2$ ?
  \\ 
   $\diag{-2-1,-20,10,11}$    & 3  &$[23,572]^\star$ & \begin{tabular}[t]{c}$[48,589]^\star$
      \\
  \end{tabular}& ? & $-7/3$ & $?$& $?$ & ? & $-3/2$ ?  
  \end{tabular}
\vskip 4mm
   \caption{The four Gessel-like models, with their periods~$m$. 
     The table gives, for each
     sequence $(e_{mn})$ and $(q_n)$, and for the associated series $Q(0,0;t^{1/m})$
     and $Q(1,1;t)$, the order and degree of the guessed recurrence
     relation or differential  equation (in the two algebraic cases,
     the first/second value is the degree in the series/variable). A star indicates that we have {only guessed} recurrences or differential equations 
modulo $p=2147483647$.
    }
    \label{tab:gessel-models}
  \end{table}

Our results are summarized in Table~\ref{tab:gessel-models}, and completed
with a few details below.

\medskip
\noindent\paragraph{$\bullet$ \bf Case $\boldsymbol{\cG_1=\{
\bar 2\bar 1, \bar 1 0, 01, 10  \}}$}	

The excursion generating function starts
\[ Q(0,0) = 1+t^2+5\, t^4+27\, t^6+188\, t^8+1414\, t^{10}+O(t^{12})\]
while
\[ Q(1,1) = 1+2\, t+5\, t^2+13\, t^3+38\, t^4+112\, t^5+346\, t^6+1071\, t^7+O(t^8).\]
The model has period $m=2$, and $e_n=0$ if $n$ is odd. The growth
constant is $\mu=2\sqrt{3}$ for the excursion  sequence, and 
\begin{equation}
  \label{bar-mu-G1}
\bar \mu= \frac{\sqrt[3]{6371+624 \, \sqrt{78}}}{12} 
+ \frac{217}{12 \, \sqrt[3]{6371+624 \, \sqrt{78}}}
+\frac{11}{12}
\, \sim 3.61
\end{equation}
for all quadrant walks. The value $\bar \mu$ is the unique (positive)
real root of $16+8u+11u^2 -4u^3$.

The leading coefficient of the operator
$L_e$ annihilating $Q(0,0;t^{1/2})$ is 
\[ t^5 \, (1 + 4\, t)^3 \, (1-12\, t)^5 
\, (279936\, t^5-62208\, t^4+13608\, t^3-5796\, t^2+675\, t-20)
\]
where the factor $(1-12\, t)$ is
compatible with the growth constant $\mu$.
 The leading coefficient of~$L_w$ is
\[
t^{10} \,  
(1 + 4\, t) \,  
(1 - 4\, t)^4 \,  
(1 + 4 \, t^2)^5 \,  
(1 - 12 \, t^2)^9 \,  
\, (16\, t^3+8\, t^2+11\, t-4)^4 \,
\times \left( \text{irreducible poly. of degree 345} \right),\]
which is compatible with the value of $\bar \mu$.

%====================================================================
\medskip
\noindent\paragraph{$\bullet$ \bf Case $\boldsymbol{\cG_2=\{
\bar 2\bar 1, \bar 1 1, 0\bar 1, 11 \}}$}	

The excursion generating function starts
\[ Q(0,0) = 1+5\, t^4+190\, t^8+11892\, t^{12} +939572\, t^{16}+O(t^{20})\]
while
\[ Q(1,1) = 1+ t+3\, t^2+8\, t^3+24\, t^4+65\, t^5+211\, t^6+649\, t^7+O(t^8).\]
The model has period $m=4$, and $e_n=0$ if $n$ is not a multiple of
$4$. The growth constant of excursions is $\mu=8/3^{3/4}$, while the
constant for all quadrant walks is again given by~\eqref{bar-mu-G1}.

For the nontrivial subsequence $(u_n) = (e_{4n})$ we have guessed  that
\begin{multline*}
3\, (6\, n+11)\, (18\, n+41)\, (2\, n+5)\, (3\, n+7)\, (18\, n+35)\, (6\, n+13)\, (n+2)\, (18\, n+29)\, \\
(41472\, n^4+150144\, n^3+200864\, n^2+117704\, n+25491)\, u_{n+2}-\\
(47552535724032\, n^{12}+\
798266178404352\, n^{11}+
6092888790269952\, n^{10}+27954969361514496\, n^9 \\
+
85850716160655360\, n^8 + 185860480394330112\, n^7 + 290753615920332800\, n^6
+ \\
331020927507759104\, n^5 + 272073153165252608\, n^4 + 157356059182977536\, n^3+ \\
60749526504280448\, n^2 + 14046784950077600\, n + 1470033929525700)\, u_{n+1} +\\
1048576\, (12\, n+5
)\, (4\, n+1)\, (3\, n+2)\, (2\, n+1)\, (12\, n+11)\, (4\, n+3)\, (6\, n+7)\, (n+1)\, \\ 
(41472\, n^4+316032\, n^3+900128\, n^2+1135752\, n+535675) u_n = 0.
\end{multline*}

Remarkably, the series $E(t):=Q(0,0; t^{1/4})$ appears to be algebraic,
of degree 32. More precisely, $E(t)^2$ seems to have degree 16 and to
satisfy an equation $P(t,E(t)^2)=0$ with coefficients of degree at
most 14 in $t$.
The guessed polynomial $P(t,z)$ seems plausible because: it
has a small bitsize compared to the bitsize of the {expansion of
  $E(t)$ that we used to 
produce it; we have then checked, using more terms of $E(t)$, that} it annihilates $E(t)^2$ to much higher
orders; its discriminant factors as
\[
t^{418} \,
(268435456 \, t^3+57671680 \, t^2-69632 \, t-27)^2 \,
(4096 \,t - 27)^{48}\,
 \times \left( \text{irreducible poly. of degree 31} \right)^4,
\]
which is compatible with the value of $\mu$.
Moreover,  $P(t,z)$ defines a rational curve,
parametrized by
\[
t = \frac{U\, (1-2\,U)^3\,(1-3\,U)^3\,(1-6\,U)^9}{(1-4\,U)^4}, \qquad
z = \frac{(1-4\,U)^2\,(1-24\,U+120\,U^2-144\,U^3)^2}{(1-3\,U)^2\,(1-2\,U)^3\,(1-6\,U)^9}.
\]
This leads to the following conjectural statement.
\begin{Conjecture}
For the model
$\cS=\{\bar 2\bar 1, \bar 1 1, 0\bar 1, 11 \}$,
the excursion generating function
$Q(0,0;t)$ is equal to
\[
\frac{(1-4\,U)\,(1-24\,U+120\,U^2-144\,U^3)}{(1-3\,U)\, (1-2\,U)^{3/2}\,(1-6\,U)^{9/2}},
\]
where $U = t^4 +53\, t^8+4363\, t^{12} + \cdots$
is the unique power series in $\qs[[t]]$ satisfying
\[
{U\, (1-2\,U)^3\,(1-3\,U)^3\,(1-6\,U)^9} = t^4 \, {(1-4\,U)^4}.
\]
\end{Conjecture}

{As mentioned above, we could not guess any differential nor algebraic
equation for $Q(1,1;t)$, even with 100\,000 terms and modulo
$p{=2147483647}$.}

%====================================================
\medskip
\noindent\paragraph{$\bullet$ \bf Case $\boldsymbol{\cG_3=\{
\bar 2\bar 1, \bar 1 0, \bar 1 1, 0\bar 1, 11 \}}$}	
The excursion generating function starts
\[ Q(0,0) = 1+2\, t^3+5\, t^4+16\, t^6+107\, t^7+190\, t^8+O(t^{9})\]
while
\[ Q(1,1) = 1+t+4\, t^2+12\, t^3+39\, t^4+133\, t^5+485\, t^6+1746\, t^7+O(t^8).\]
This model is strongly {aperiodic}. The growth constants are $\mu  \sim
4.03$, the  unique positive root of 
$4069+768\,u-6\,u^2+u^3-27\,u^4$, and
\[ \bar \mu  = 
\frac{\sqrt[3]{1261+57 \, \sqrt{57}}}{6} 
+ \frac{56}{3 \, \sqrt[3]{1261+57 \, \sqrt{57}}}
+\frac{2}{3}
\, \sim 4.22.
\]
This is the unique (positive) real root of $4\, u^3-8\, u^2-32\, u-23$.

Again, $Q(0,0; t)$ appears to be algebraic of degree $32$, this time
with coefficients of degree 57.
The guessed polynomial seems plausible for various reasons, including the nice factorization of its discriminant as
\begin{multline*}
t^{1732} \,
(t^2+1)^{32} \,
(4069\, t^4+768\, t^3-6\, t^2+t-27)^{48} \\
\times \left( \text{irreducible poly. of degree 13} \right)^2
 \times \left( \text{irreducible poly. of degree 31} \right)^4,
\end{multline*}
which vanishes at $t=1/\mu$.

\medskip \noindent{\bf Remark.} There are analogies between excursions of this
model and those of the Kreweras-type model ${\cK_3=\{ \bar 2 \bar 1, \bar 1
\bar 2, \bar 1 \bar 1, 0 1, 1 0 \}}$. Indeed, the sizes of the recurrence
relation and of the differential equation match. The growth constant and the
singular exponent are also the same for both models.

%===============================================
\medskip
\noindent\paragraph{$\bullet$ \bf Case $\boldsymbol{\cG_4=\{
\bar 2\bar 1, \bar 2 0, 10, 11 \}}$}	

The excursion generating function starts
\[ Q(0,0) = 1+3\, t^3+41\, t^6+850\, t^9+21538\, t^{12}+614530\, t^{15}+O(t^{16})\]
while
\[ Q(1,1) = 1+2\, t+4\, t^2+15\, t^3+45\, t^4+121\, t^5+471\, t^6+1533\, t^7+O(t^8).\]
The model has period $m=3$, and $e_n=0$ if $n$ is not a multiple of $3$. The
growth constants are $\mu=9/2^{4/3}$ and $\bar \mu= 3 \cdot 2^{1/3}$.

The leading coefficient of $L_e$ contains the factor $t^{42} \, (729\,
t-16)^{23}$ which is compatible with the value of $\mu$.

%%%%%%%%%%%%%%%%%%%%%%%%%%%%%%%%%%%%%%%%%%%%%%%%%%%%%%%%%%%%%%%%%
\section{A glimpse at asymptotics}
\label{sec:asympt}
%%%%%%%%%%%%%%%%%%%%%%%%%%%%%%%%%%%%%%%%%%%%%%%%%%%%%%%%%%%%%%%%%%%%%ù

The method that we develop in this paper {provides} expressions for \gfs\ of
walks confined to an orthant, as positive parts of certain rational or
algebraic series. We now demonstrate that these expressions are often well
suited to a multivariate singularity analysis. The use of analytic techniques
in this fashion is the domain of \emph{analytic combinatorics in several
variables} (ACSV)~\cite{PemantleWilson2013}; recent work has shown the
strength of this approach, proving conjectures in lattice path
asymptotics~\cite{MelczerWilson2016}, generalizations in higher
dimensions~\cite{MelczerMishna2016}, and handling families of models with
weighted steps~\cite{CourtielMelczerMishnaRaschel2017}. Much of the
singularity analysis is effective~\cite{Melczer2017} when the multivariate
generating function under consideration is represented in the form
$Q(x,y;t)=[x^{\geq }y^{\geq}]R(x,y;t)$ for a {\emm rational,} function
$R(x,y;t)$. Although some asymptotic techniques have been developed to perform
a singularity analysis on multivariate functions with algebraic
singularities~\cite{Greenwood2018}, this is a more difficult task. For the
purposes of this paper, we show how dominant asymptotics for the number of
walks in the four models of Section~\ref{sec:works} can be determined through
the simple use of analytic techniques. {We focus on the series $Q(1,1)$
counting all quadrant walks.} Future work could extend this argument to deal
with the multivariate algebraic functions which arise, for instance, in the
generating functions for {2D} Hadamard models given by
Proposition~\ref{prop:hadamard}.

The first step is to convert our expression of the form $Q(x,y;t)=[x^{\geq }y^{\geq}]R(x,y;t)$ for the multivariate generating function $Q(x,y;t)$
into an expression for the univariate generating function $Q(1,1;t)$
which is amenable to asymptotic computations.  Given an element
\begin{equation}
  R(x,y;t) = \sum_{n \geq 0} \left( \sum_{i,j} r(i,j;n)x^iy^jt^n
  \right) \in \mathbb{Q}[x,\bx,y,\by][[t]], \label{eq:takediag}
\end{equation}
the \emph{diagonal} operator $\Delta$ takes $R(x,y;t)$ and returns the
univariate power series $(\Delta R)(t) := \sum_{n \geq 0}
r(n,n;n)t^n$. The relationship between positive parts and diagonals
is given by the following lemma.

\begin{Lemma} \label{lem:diag}
Given $R(x,y;t)$ as in~\eqref{eq:takediag}, and $(a,b) \in \{0,1\}^2$, one has
\[ [x^\geq y^\geq]R(x,y;t)\bigg|_{x=a,y=b} = \Delta\left( \frac{R\left(\bx,\by;xyt\right)}{(1-x)^a(1-y)^b}\right). \]
\end{Lemma}
The proof follows from basic formal series manipulations; see
Proposition 2.6 of~\cite{MelczerMishna2016} for details.  In
particular, this lemma, combined with the expressions obtained for
$Q(x,y;t)$ in Section~\ref{sec:works}, gives us diagonal representations for
the generating functions  of quadrant walks ending anywhere $(a=b=1)$,
returning to the origin (excursions, $a=b=0$), or returning to the
$x$- or $y$-axes ($a=1$, $b=0$ or $a=0$, $b=1$). 

At its most basic level, the theory of ACSV takes a multivariate Cauchy
residue integral representation for power series coefficients and reduces it
to an integral expression where saddle-point techniques can be used to
determine asymptotics. Because of the simple rational functions which are
obtained for many lattice path models, the usual analysis can be greatly
simplified. In particular, for each of the four models detailed in
Section~\ref{sec:works} we obtain the generating function $Q(1,1;t)$ as a
diagonal of the form
\[
Q(1,1;t) = \Delta
\left(\frac{P(x,y)}{(1-x)(1-y)(1-txyS(\bx,\by))}\right),
\]
where $P(x,y)$ is a Laurent polynomial which is coprime with $1-x$ and $1-y$. 
Expanding the rational function on the right-hand side of this equation as a 
power series in $t$ then gives
\[ 
q_n = [t^n] Q(1,1;t) = [x^ny^nt^n]
\left(\frac{P(x,y)}{(1-x)(1-y)(1-txyS(\bx,\by))}\right) =
[x^0y^0]\frac{P(x,y)S(\bx,\by)^n}{(1-x)(1-y)}, 
\]
and the multivariate Cauchy integral formula~\cite[Prop. 7.2.6]{PemantleWilson2013} implies
\begin{align*} 
q_n &= \frac{1}{(2\pi i)^2}\int_{|x|=r_1,|y|=r_2}
      \frac{P(x,y)S(\bx,\by)^n}{(1-x)(1-y)} \cdot \frac{\mathrm
      dx\,\mathrm dy}{xy} \notag \\[+1mm]
&= \frac{1}{(2\pi i)^2} \int_{|x|=r_1,|y|=r_2}
  \frac{P(x,y)}{xy(1-x)(1-y)}e^{n \log S(\bx,\by)} \mathrm d
  x\,\mathrm d y 
\end{align*}
for any $0 < r_1,r_2 < 1$.  
{Making the substitutions $x=r_1e^{i\theta_1}$ and $y=r_2e^{i\theta_2}$
converts this integral into a
\emm Fourier-Laplace, integral; that is, an integral of the form
\[
\int_T A(\theta_1,\theta_2) e^{-n \phi(\theta_1,\theta_2)}
\mathrm d\theta_1 \mathrm d\theta_2 .
\]
Here  $T = [-\pi, \pi]^2$, while
\[
A(\theta_1, \theta_2)=\frac 1{(2\pi)^2} \frac{P(r_1e^{i\theta_1}, r_2
  e^{i\theta_2})}{(1-r_1e^{i\theta_1})(1-r_2 e^{i\theta_2})},
\]
and
\[
\phi(\theta_1, \theta_2):= -\log S\left(r_1^{-1} e^{-i\theta_1},r_2^{-1}
  e^{-i\theta_2}\right) .
\]
The asymptotics of Fourier-Laplace integrals have been well studied.} In
particular, suppose the \emph{amplitude} $A$ and \emph{phase} $\phi$ are
analytic functions on the domain $T$. If $\phi$ admits a non-empty finite set
of critical points\footnote{For the purposes of this discussion, points in $T$
where the gradient of $\phi$ vanishes.}, at which the Hessian of $\phi$ is
non-singular and the real part of $\phi$ is locally minimized, then explicit
asymptotic formulas in terms of the Taylor coefficients of $A$ and~$\phi$ are
known~\cite[Theorem 7.7.5]{Hormander1990a} (see also~\cite[Prop.
53]{Melczer2017} for the explicit formulas used in our calculations). Each
critical point of $\phi$ has an asymptotic contribution, and one simply sums
up the contributions of all critical points to determine dominant asymptotics
of the Fourier-Laplace integral.

For the above value of $\phi(\theta_1, \theta_2)$, the chain rule shows that
in order to find real values $r_1$ and~$r_2$ such that $\phi$ admits critical
points, it is sufficient to find the complex points $(x,y)$ such that
\begin{equation}\label{crit}
S_x(\bx,\by) = S_y(\bx,\by) = 0
\end{equation}
and take $r_1$ and $r_2$ to be their moduli. One then determines the
corresponding critical pairs $(\theta_1, \theta_2)$, that is,
  the arguments of $x$ and $y$ satisfying~\eqref{crit},
%having these moduli, 
and computes the Hessian of $\phi$
at these points. 
 In each of the
four cases that we consider there are critical points with $0<r_1,r_2<1$, and the Hessian is never singular. 
The next step is to show that the real part of $\phi(\theta_1,\theta_2)$,
\[ 
\Re \phi(\theta_1,\theta_2) = - \log
|S(r_1^{-1}e^{-i\theta_1},r_2^{-1}e^{-i\theta_2})|, 
\]
is locally minimized at critical values of $(\theta_1, \theta_2)$. 
Minimizing this quantity  means
maximizing $|S(\bx,\by)|$ on $\{(x,y): |x|= r_1, |y|=r_2\}$.  Since 
$S(\bx,\by)$ is a {(Laurent)} polynomial with non-negative coefficients, when $|x|$ and $|y|$
are fixed then $|S(\bx,\by)|$ is maximized (in particular)
when~$x$ and~$y$ are positive and real {(that is, $x=r_1, y=r_2$)}.  The triangle inequality then
shows that the maximizers of $|S(\bx,\by)|$ occur when the arguments
of all monomials occurring in 
$S(\bx,\by)$ are equal. 
{When this holds for all critical values} $(\theta_1, \theta_2)$,  explicit asymptotics can be obtained
by direct computation. {In particular, the exponential growth
associated with the critical point $(\theta_1, \theta_2)$ is
$e^{-\phi(\theta_1,\theta_2)}=S(\bx, \by)$.}

We now list our results; full details of the computations can be found in an
accompanying {\sc Maple} worksheet, available on {the authors'
webpages}\footnote{For lattice path examples with more exotic critical point
behaviour, see~\cite[Ch. 10 and 11]{Melczer2017}.}.

%===============================================
\subsection{Case $\boldsymbol{\cS=\{10, \bar 1 0, 0 \bar 1, \bar 2 1\}}$}
%===============================================
Specializing Lemma~\ref{lem:diag} to Proposition~\ref{prop:F} gives the
diagonal representation
\[ Q(1,1;t) = \Delta
\left(\frac{(x^2+1)(x^2+2xy-1)(2x^3+x^2y-y)(x^2-y^2)}{x^2y(1-x)(1-y)(1-t(x^3+x^2y+xy^2+y))}\right).
\]
Solving~\eqref{crit} for $x$ and $y$ gives two solutions with coordinates of modulus less than 1,
\[ (x,y) = \left(3^{-1/2},3^{-1/2}\right) \quad  \text{ and } \quad (x,y) =\left(-3^{-1/2},-3^{-1/2}\right), \]
along with solutions $(i,-i),$ and $(-i,i)$ which are irrelevant
to asymptotics. Taking $r_1=r_2 = 3^{-1/2}$ in the argument above, one
gets a Fourier-Laplace integral with critical points at
$(\theta_1,\theta_2) = (0,0)$ and $(\pi,\pi)$.
A direct computation shows that the
  Hessian of $\phi$ is non-singular at these critical
  points. Following the above lines, we then check that 
\[ 
{\left| S(\bx,\by)\right| =\left| \bx+x+y+x^2\by\right|}
\]
is indeed maximal on the integration domain for angles $(0,0)$ and
$(\pi, \pi)$, as desired.

The exponential growth of the resulting Fourier-Laplace integral is
given by the value of $e^{-\phi(\theta_1,\theta_2)}{= S(\bx, \by)}$ at the critical points, in this case $S(\sqrt{3},\sqrt{3}) = 2\sqrt{3}$ and $S(-\sqrt{3},-\sqrt{3}) = -2\sqrt{3}$.  One then computes successively higher order terms in an asymptotic expansion 
\[ (2\sqrt{3})^n\left( A_0 + \frac{A_1}{n} + \frac{A_2}{n^2} + \cdots \right) + (-2\sqrt{3})^n\left( A_0' + \frac{A_1'}{n} + \frac{A_2'}{n^2} + \cdots \right) \]
until finding terms which are non-zero (see~\cite[Prop. 53]{Melczer2017}). The
vanishing of the highest order terms is related to, but not completely
determined by, the order of vanishing of the amplitude $A(\theta_1,\theta_2)$
at the critical points under consideration. Ultimately, we obtain the
asymptotic expansion
\[ q_n = \frac{(2\sqrt{3})^n}{\pi n^4}\left(C_n + O\left(\frac{1}{n}\right)\right), \]  
where
\begin{equation} C_n = \begin{cases} {5616\sqrt{3}} &: n \text{ even} \\ {9720} &: n
  \text{ odd} \end{cases}. \label{eq:asm1} \end{equation}

%===========================================
\subsection{Case $\boldsymbol{\cS=\{01, 1\bar 1 , \bar 1 \bar 1,
    \bar 2 1\}}$}
%===========================================
Applying Lemma~\ref{lem:diag}  to the generating function expression in Proposition~\ref{prop:12-a} gives a diagonal representation
\[ 
Q(1,1;t) = \Delta \left(
  \frac{(2xy^2+x^2-1)(x-y^2)(x^2y^2+2x^3-y^2)}{xy^2(1-x)(1-y)(1-t(x^2y^2+x^3+y^2+x))}
\right). 
\]
This time the system of equations~\eqref{crit}
admits four solutions whose coordinates have moduli less than $1$,
\[ 
\left(3^{-1/2},3^{-1/4}\right),
  \left(3^{-1/2},-3^{-1/4}\right), \left(-3^{-1/2},i3^{-1/4}\right),
  \left(-3^{-1/2},-i3^{-1/4}\right), 
\]
all of which have coordinate-wise moduli $(r_1,r_2) = \left(3^{-1/2},3^{-1/4}\right)$. A similar analysis to the first case gives
\[ 
q_n = \frac{\left(8\cdot3^{-3/4}\right)^n}{\pi n^4}\left(C_n +
  O\left(\frac{1}{n}\right)\right), 
\]
where
\[ 
C_n = \begin{cases} 
{5120\sqrt{3}} &: n \equiv 0 \mod{4} \\
{6656 \cdot 3^{1/4}} &: n \equiv 1 \mod{4} \\
{26624}/{3} &: n \equiv 2 \mod{4} \\
{3840\cdot3^{3/4}} &: n \equiv 3 \mod{4}
\end{cases}.
\]

%===========================================
\subsection{Case $\boldsymbol{\cS=\{01, 1\bar 1 , \bar 1 \bar 1,
    \bar 2 1, \bar 1 0\}}$}
%===========================================
Specializing Lemma~\ref{lem:diag}
to Proposition~\ref{prop:model3} gives a diagonal representation
\[ Q(1,1;t) = \Delta \left( \frac{(x-y^2)(2xy^2+x^2+xy-1)(x^2y^2+2x^3+x^2y-y^2)}{xy^2(1-x)(1-y) (1-t(x^2y^2+x^3+x^2y+y^2+x)))} \right). \]
Here the system~\eqref{crit}
has four solutions $(x,y)$ with coordinates of modulus less than 1,
which make up the set
\begin{equation} \left\{ (y^2,y) : 3y^4+y^3-1=0
  \right\}. \label{eq:crit3} 
\end{equation}
The polynomial $3y^4+y^3-1$ has a unique positive root, $y_c\simeq
0.688...$, and we consider the solution  $(y_c^2,y_c)$. 
 None of the three other  solutions
 has the same coordinate-wise moduli, {hence our only critical point
  associated with moduli $(r_1, r_2)=(y_c^2,y_c)$ is $(\theta_1,
  \theta_2)=(0,0)$. The Hessian
  of~$\phi$ is not singular at $(0,0)$, and by positivity, this point
  maximizes the modulus of $S$  in $[-\pi, \pi]^2$.}
In the end, one obtains asymptotics
\[ 
q_n = \frac{(8y_c^3 +3y_c^2)^n}{2313\pi n^4}\left(C +
  O\left(\frac{1}{n}\right)\right) \approx
(1112.183\cdots)\frac{(4.03164\cdots)^n}{n^4}, 
\]
where 
\[
C = \sqrt{3}\left(2527386y_c^3 +2727881y_c^2 + 1805111y_c +
  1306017\right). 
\]
It can be checked that the three other solutions in~\eqref{eq:crit3}
are \emm not, local maximizers of $|S(\bx,\by)|$ among points with the
same coordinate-wise moduli. 

%==============================================
\subsection{Case $\boldsymbol{\cS=\{10,  1 \bar 1, \bar 2 1, \bar 2
    0\}}$}
%==============================================
Specializing Lemma~\ref{lem:diag}
to Proposition~\ref{prop:hard-18} gives a diagonal representation
\[ Q(1,1;t) = \Delta \left(
  \frac{(1-2y)(x^3-y^2)(x^6+x^3y^2+3x^3y-y)(2x^3-y)}
{x^3y^2(1-x)(1-y)(1-t(x^3y+x^3+y^2+y))} \right). \]
Here there are three 
solutions to~\eqref{crit} with moduli less than $1$: 
\[ \left(4^{-1/3},1/2\right), \quad \left(e^{2\pi
    i/3}4^{-1/3},1/2\right), \quad \left(e^{-2\pi
    i/3}4^{-1/3},1/2\right). \]
All of them have moduli $(r_1, r_2)=\left(4^{-1/3},1/2\right)$.  They
give rise to three critical points of $\phi$, where the Hessian is
non-singular and $|S(\bx, \by)|$ is maximized. An analysis similar to those above gives another periodic asymptotic expansion
\[ q_n  = \frac{\left(9 \cdot 4^{-2/3}\right)^n}{\pi n^5}\left(C_n + O\left(\frac{1}{n}\right)\right), \]
where
\[ C_n = \begin{cases} 
{216513}/2 &: n \equiv 0 \mod{3} \\
{1358127 \cdot 2^{-11/3}} &: n \equiv 1 \mod{3} \\
{124659\cdot 2^{-1/3}} &: n \equiv 2 \mod{3} 
\end{cases}.
\]

%%%%%%%%%%%%%%%%%%%%%%%%%%%%%%%%%%%%%%%%%%%%%%%%%%%%%%%%%%%%%%
\section{Final questions and comments}
\label{sec:final}
%%%%%%%%%%%%%%%%%%%%%%%%%%%%%%%%%%%%%%%%%%%%%%%%%%%%%%%%%%%%%%

We have outlined above the first general approach to count walks confined to
an orthant with {arbitrary} steps, and demonstrated its efficacy across
several families and a large number of sporadic cases. In addition to the
examples presented here, the power of this method is illustrated by the fact
that it {solves} another family of quadrant models, with steps $\cS=\{(-p,0),
(-p+1, 1), \ldots, (0,p), (1, -1)\}$, which arose naturally in other
applications; the details of this family (containing both large forward and
large backward steps) are given in~\cite{BoFuRa17}. {The current} work
attempts to lay a basis for the systematic study of lattice walks with longer
steps, and we suggest here some possible research directions.

\begin{itemize}
\item{\bf{Uniqueness of the section-free equation.}} Is it true that, for a
model with no large forward step and a finite orbit, there exists a unique
section-free equation (Conjecture~\ref{conj:section-free})? Can one describe
it generically?

\item{\bf{Walks with steps in $\boldsymbol{\{\bar 2, \bar 1, 0, 1\}^2}$.}} In
our study of these walks (Section~\ref{sec:m21}) we have left open the case of
nine models which have analogies with the four tricky-but-algebraic small step
models of Figure~\ref{fig:alg} (see Tables~\ref{tab:kreweras-models}
and~\ref{tab:gessel-models}). Can one apply to them some of the techniques
used for the small step algebraic
models{~\cite{BeBMRa-17,BoKa10,BKR-13,Bous05,mbm-gessel,BoMi10,gessel-proba,KaKoZe08,kurkova-raschel,mishna-jcta}}?
In particular, are the associated series D-finite? Which ones are algebraic?

Can one prove the non-D-finiteness of the 16 models of
Table~\ref{tab:embarassing}, {which have a rational excursion exponent but an
infinite orbit}?

\item {\bf{Walks with steps in $\boldsymbol{\{\bar 1,0,1,2\}^2 }$.}}
Symmetrically, one can examine the $14\,268$ interesting (non-isomorphic and
non-trivial) models with steps in $\{-1,0,1,2\}^2$, having at least one large
forward step. Proceeding as in Section~\ref{sec:m21} reveals that $1\,189$ of
them are included in a half-space, and thus analogous to the 5 half-space
models with small steps. Of the remaining $13\, 079$ {models}, $12\,828$ have
an irrational excursion exponent, and hence a non-D-finite \gf\ and an
infinite orbit (Section~\ref{sec:infinite-orbit}). {The $251$ that have a
rational exponent split in three families}:
  \begin{itemize}
  \item   11  have yet an infinite orbit. They are the reverses
  of the 11 models of Table~\ref{tab:embarassing} that contain a step
  in $\zs_{-}^2$ (for the other 5 models in this table, there is no non-trivial walk  starting at the origin after reversing steps);
\item 227 are Hadamard, and thus solvable by
  Proposition~\ref{prop:hadamard} and D-finite. They are the
  reverses of the 227 Hadamard models of
  Section~\ref{sec:works};
\item  13 are the reverses of the models 
  in Table~\ref{tab:finite}, and thus share their orbit structure: $O_{12}$,
  $\tilde O_{12}$ or $O_{18}$. {They also share their excursion
    \gf, which we have either proved or conjectured to be D-finite in
    all 13 cases.}
  \end{itemize}

\item {It has been proved~\cite{BoChHoKaPe17} that for the 19 small step models in
    the quadrant that are D-finite but transcendental, the series
    $Q(x,y;t)$ has an explicit expression involving integrals and
    specializations of the
    hypergeometric series $_2F_1$. For which models with larger steps
    is this still true? Corollary~\ref{cor:explicit} and Conjecture~\ref{conj:K2-ex} show
    that a similar property may indeed hold in some cases.}

\item {We have focussed in this paper on 2D examples, because
    the quadrant is already a rich source of interesting problems. But
    the four stages of the method, described in Sections~\ref{sec:eqfunc}
  to~\ref{sec:extract}, apply just as well to higher dimensional models. In fact,
  they were already successfully applied to 3D models with small steps
  in~\cite{BoBoKaMe16}.}
\end{itemize}

%%%%%%%%%%%%%%%%%%%%%%%%%%%%%%%%%%%%%%%%%%%%%%%%%%%%%%%%%%%%%%

\bigskip \noindent {\bf Acknowledgments.} We thank Axel Bacher for great help
with computations of walk sequences, {Andrew Elvey Price and Michael Wallner
for providing the parity argument of Proposition~\ref{prop:small-group},} Mark
van Hoeij for useful discussions on hypergeometric solutions of differential
equations, and J\'er\^ome Leroux for pointing us to Farkas' Lemma.

\bigskip
 
\bibliographystyle{plain}
\bibliography{qdp}

\end{document}